\documentclass[10pt,final]{siamltex}
\usepackage{amsmath}
\usepackage{bm}
\usepackage{amssymb,version}
\usepackage{cases}
\usepackage{color}
\usepackage{verbatim}
\newtheorem{rem}{Remark}[section]
\usepackage{graphicx}
\usepackage{subfig}
\allowdisplaybreaks

\newcommand{\normmm}[1]{{\left\vert\kern-0.25ex\left\vert\kern-0.25ex\left\vert #1
    \right\vert\kern-0.25ex\right\vert\kern-0.25ex\right\vert}}
\newcommand{\ttt}[1]{(-\frac{t^{n+1}}{T})}

\begin{document}
    \title{Novel linear, decoupled, and  energy dissipative schemes  for the Navier--Stokes--Darcy model and extension to related two-phase flow
\thanks{This work is supported in part by the National Natural Science Foundation of China (Grant Nos.12271302, 12131014, W2431008 and 12371409) and Shandong Provincial Natural Science Foundation for Outstanding Youth Scholar (Grant No. ZR2024JQ030)}}

   \author{ Xiaoli Li\thanks{ School of Mathematics and State Key Laboratory of Cryptography and Digital Economy Security, Shandong University, Jinan, Shandong, 250100, P.R. China. Email: xiaolimath@sdu.edu.cn}.
   \and Jie Shen\thanks{Corresponding Author. School of Mathematical Science, Eastern Institute of Technology, Ningbo, China. Email: jshen@eitech.edu.cn}.
   \and Xinhui Wang\thanks{School of Mathematics, Shandong University, Jinan, Shandong, 250100, P.R. China. Email: xinhuiwang0108@163.com}.
}

\maketitle

\begin{abstract}
We construct  efficient original-energy-dissipative schemes for the Navier--Stokes--Darcy model and related two-phase flows using a prediction-correction framework. A new relaxation technique is incorporated in the correction step to guarantee dissipation of the original energy, thereby ensuring unconditional boundedness of the numerical solutions for velocity and hydraulic head in the $l^{\infty}(L^2)$ and $l^2(H^1)$ norms. At each time step, the schemes require solving only a sequence of linear equations with constant coefficients. We rigorously prove that the schemes dissipate the original energy and, as an example, carry out a rigorous error analysis of the first-order scheme for the Navier--Stokes--Darcy model. Finally, a series of benchmark numerical experiments are conducted to demonstrate the accuracy, stability, and effectiveness of the proposed methods.
\end{abstract}

 \begin{keywords}
Navier--Stokes--Darcy model; two-phase flows; energy dissipation; error estimates
 \end{keywords}
   \begin{AMS}
 76D05, 76T10, 65M12
\end{AMS}

\section{Introduction}
The fluid-porous media  system has broad applications across industrial processes and geophysical systems, attracting considerable research attention \cite{Hill2008Poiseuille,mccurdy2019convection,dis2004domain,Layton2002Coupling,Hanspal2006numeri,mardal2002robust}. The coupled  fluid-porous media system can be modeled mathematically  by the incompressible Navier--Stokes equations for the free flow and Darcy's law for the flow in the porous media, coupled through interface conditions that enforce fundamental conservation laws: mass continuity and momentum balance, with tangential slip behavior described by the Beavers--Joseph--Saffman (BJS) condition \cite{cao2010coupled,he2015domain}.

Numerical simulation of free-flow--porous-media coupled systems has become a critical tool in subsurface engineering and environmental hydrology. Construction of structure-preserving schemes is crucial, as numerical solutions must faithfully approximate the exact solutions of the coupled model.  Considerable research efforts have been devoted to developing numerical methods for the Navier--Stokes--Darcy model \cite{badea2010numeri,cai2009numeri,Chidyagwai2009on,cao2013decoupling,Girault2009DG,he2015domain,Shiue2018Convergence,ce2013time,Qiu2020domain,Chen2013efficient}. For instance, McCurdy et al. \cite{mccurdy2019convection} conducted linear and nonlinear stability analyses of thermal convection in a fluid overlying a saturated porous medium, investigating the transition between full convection and fluid-dominated convection. They further developed a coarse-grained model and fully nonlinear schemes to predict convection patterns in coupled fluid-porous systems \cite{mccurdy2022predicting}. Rui and Zhang \cite{RUI2017stabilized} designed a stabilized mixed finite element method for solving the coupled Stokes--Darcy system with solute transport. Chen et al. \cite{Chen2013efficient} proposed two second-order implicit--explicit schemes for the coupled Stokes--Darcy equations and established their unconditional stability. More recently, Chen et al. \cite{Chen2020Uniquely} introduced a family of fully decoupled numerical schemes for the extended Cahn--Hilliard--Navier--Stokes--Darcy--Boussinesq system, allowing the Navier--Stokes equations, Darcy equations, heat equation, and Cahn--Hilliard equation to be solved independently at each time step. In another direction, Qiu et al. \cite{Qiu2020domain} proposed a domain decomposition method for a time-dependent Navier--Stokes--Darcy model with Beavers--Joseph interface and defective boundary conditions, and rigorously established its convergence.

As far as we know, all the aforementioned works rely on either an implicit or an implicit--explicit (IMEX) treatment of the nonlinear convection term in the Navier--Stokes equations, which leads to solving either a nonlinear system or a linear system with variable coefficients at each time step. From a computational perspective, it is highly desirable to treat the nonlinear term explicitly without imposing stability constraints. For instance, the IMEX schemes for the Navier--Stokes equations proposed in \cite{lin2019numerical,li2020error}, based on the scalar auxiliary variable (SAV) approach \cite{shen2018scalar,shen2018convergence}, achieve unconditional energy stability while handling the nonlinear term explicitly. Improved variants can be found in \cite{li2022new,li2022fully}. Building on the ideas in \cite{li2022new}, Jiang and Yang \cite{jiang2021SAV,jiang2023fast} developed efficient artificial compressibility ensemble schemes for Stokes--Darcy flow ensembles, while Wang et al. \cite{wang2024class} constructed new IMEX schemes for the Navier--Stokes--Darcy model, combining the SAV approach in time with the finite element method in space. However, all these methods preserve only a modified energy, which differs from the model’s original energy. Furthermore, the computational cost of solving the resulting linear systems \cite{jiang2021SAV,jiang2023fast} is typically about twice that of classical algorithms. Although there exists a class of Lagrange multiplier methods that can preserve the original energy dissipation \cite{cheng2020new}, their implementation requires solving nonlinear algebraic systems. To the best of our knowledge, no existing work addresses the Navier--Stokes--Darcy model using a linear, decoupled, and explicit discretization of both the nonlinear convection and the coupled interface terms while still maintaining the original energy dissipation, let alone providing a convergence analysis.

Extending the methodology from the Navier--Stokes--Darcy model to Cahn--Hilliard--Navier--Stokes--Darcy system is crucial for accurately modeling multiphase flow dynamics in realistic applications. The latter model \cite{gao2018decoupled, wu2025optimal,chen2014numerical,han2021existence,chen2017uniquely,gao2023fully} incorporates the Cahn--Hilliard equation, which governs phase separation processes. While this substantially increases the mathematical and computational complexity, it provides a more comprehensive framework for applications such as contaminant transport and oil recovery. Developing energy-dissipative schemes for this highly nonlinear multi-physics model highlights the broader applicability of our approach.

The primary objective of this work is to construct linear, original-energy-dissipative schemes for the Navier--Stokes--Darcy model, and to extend them to two-phase flow problems.
The main contributions of this paper are as follows:

\begin{itemize}
	\item[$\bullet$] We develop first- and second-order linear prediction--correction schemes for the Navier--Stokes--Darcy model by proposing a new relaxation technique. These schemes are decoupled, requiring only the solution of a sequence of linear differential equations with constant coefficients at each time step. They preserve the model’s original energy and consequently guarantee unconditional boundedness of the velocity and hydraulic head in the $l^{\infty}(L^2)$ and $l^2(H^1)$ norms.
	
	\item[$\bullet$] Both schemes are proved to be uniquely solvable and unconditionally stable without any restriction on the time step. In addition, we establish optimal error estimates for the first-order scheme.
	
	\item[$\bullet$] We extend the proposed framework to the Cahn--Hilliard--Navier--Stokes--Darcy phase-field model. This extension retains key features of the original schemes, including linearity and preservation of the original energy.
\end{itemize}

The remainder of the paper is organized as follows. Section 2 formulates the Navier--Stokes--Darcy model and establishes its energy stability. In Section 3, we introduce first- and second-order linear, decoupled, and original-energy-dissipative schemes. Section 4 presents the theoretical analysis, proving unconditional energy stability, existence of solutions, and optimal convergence of velocity, hydraulic head, and energy via mathematical induction. Section 5 extends the methodology to the Cahn--Hilliard--Navier--Stokes--Darcy model with phase dynamics. Finally, Section 6 extends the three proposed numerical schemes to the EMAC formulation and presents numerical experiments, such as convergence tests, filtration processes, and buoyancy-driven bubble dynamics. Some conclusions are given in the last section.

\section{The Navier--Stokes--Darcy model}
We start with some basic definitions. Consider a bounded domain $\Omega\subset \mathbb{R}^d$, $d=2,3$, which includes a free-flow domain $\Omega_f$ and a porous-media domain $\Omega_p$. We denote the interface $\Gamma$ between these two areas by $\Gamma=\partial\Omega_f\cap\partial\Omega_p$, and set $\Omega={\Omega_f}\cup{\Omega_p}\cup\Gamma$.  Here, the unit normal vector from $\Omega_f$ to $\Omega_p$ on $\Gamma$ is defined as $\mathbf{n}=\mathbf{n}_f=-\mathbf{n}_p$. To simplify the notation, we further define $\Gamma_f=\partial\Omega_f\backslash\Gamma$, $\Gamma_p=\partial\Omega_p\backslash\Gamma$.

\subsection{Model definition and concepts}
In the free-flow domain $\Omega_f$, the fluid satisfies the incompressible Navier--Stokes equations.
\begin{align}
\partial_t \mathbf{u}
-\nu \Delta \mathbf{u}+\nabla p +(\mathbf{u}\cdot\nabla)\mathbf{u}
&=\mathbf{f}_f, \ \ \  {\rm in}~\Omega_f\times(0,T],\label{stokes1}\\
\nabla\cdot \mathbf{u}&=0, \ \ \ \  {\rm in}~\Omega_f\times(0,T],\label{stokes2}\\
\mathbf{u}(\mathbf{x},0)&=\mathbf{u}^0, \ \ \ {\rm in}~\Omega_f,\label{stokes t=0}\\
\mathbf{u}(\mathbf{x},t)&=\mathbf{0}, \  \ \ \  {\rm on}~\Gamma_f\times (0,T].\label{stokes gamma}
\end{align}
Here, the coefficient $\nu$ denotes the kinematic viscosity, which is inversely proportional to the Reynolds number $Re$. The final time is denoted by $T>0$. $\mathbf{u}$ and $p$ are the velocity and the pressure of free flow in $\Omega_f$, respectively. $\mathbf{u}^0$ is the initial solution. The function $\mathbf{f}_f$ is an external body force.

In the porous-media domain $\Omega_p$, the flow is governed by Darcy's law.
\begin{align}
S_0\partial_t \phi-\nabla\cdot(\mathbb{K}\nabla\phi)&=f_p, \ \ \ \qquad {\rm in}~\Omega_p\times(0,T],\label{darcy1}\\
\phi(\mathbf{x},0)&=\phi^0, \ \ \ \ \ \ \ \ \ {\rm in}~\Omega_p,\label{darcy t=0}\\
\phi(\mathbf{x},t)&=0, \ \ \ \ \ \ \ \ ~~~{\rm on}~\Gamma_p\times (0,T], \label{darcy gamma}
\end{align}
where $\phi$ is the hydraulic head in $\Omega_p$, $S_0$ is a specific mass storativity coefficient, $f_p$ is the source/sink term, $\phi^0$ is the initial solution, and the hydraulic conductivity tensor $\mathbb{K}$ is a symmetric positive-definite matrix.

On the interface $\Gamma$ between the free-flow domain and the porous-media domain, we impose the following  interface conditions:
\begin{align}
\mathbf{u} \cdot\mathbf{n} + \mathbb{K}\nabla\phi\cdot\mathbf{n}&=0, \qquad~~{\rm on}~\Gamma\times (0,T],\label{gamma1}\\
p-\nu \mathbf{n}\cdot(\nabla \mathbf{u} \cdot \mathbf{n})+\frac{1}{2}|\mathbf{u}|^2&=g\phi, \qquad{\rm on}~\Gamma\times (0,T],\label{gamma2}\\
-\nu\boldsymbol{\tau}\cdot(\nabla \mathbf{u} \cdot \mathbf{n})&=\alpha\sqrt{\frac{\nu g}{tr(\mathbb{K})}}(\mathbf{u}\cdot\boldsymbol{\tau}), \qquad{\rm on}~\Gamma\times (0,T].\label{gamma3}
\end{align}
The first equation represents the continuity of the normal component of the velocity on the interface $\Gamma$, which means the conservation of mass. The second condition represents the balance of normal forces. In particular, we consider the inertial forces here, that is, the Lions interface condition \cite{mccurdy2022predicting,Chidyagwai2009on,Girault2009DG,mccurdy2019convection}. The last equation is the Beavers--Joseph--Saffman interface condition \cite{beavers1967boundary,saffman1971boundary}.

We use $(\cdot,\cdot)_X$ (with $X = f \text{ or } p$) to denote the $L^2$ inner product in $\Omega_X$, and $\|\cdot\|_X$ to denote the corresponding $L^2$ norm in $\Omega_X$. In addition, we define
\begin{align*}
H_f&=\{\mathbf{v}\in (H^1({\Omega_f}))^d:\mathbf{v}|_{\Gamma_f}=0\},\qquad
H_p=\{\psi \in H^1(\Omega_p):\psi|_{\Gamma_p}=0\},\\
Q_f&=L^2(\Omega_f),\qquad
W=H_f\times H_p.
\end{align*}
Then, a weak formulation of the Navier--Stokes--Darcy model (\ref{stokes1})-(\ref{gamma3}) is described as follows: $\forall~t\in (0,T]$, find $\mathbf{u}\in H_f,~p\in Q_f,~\phi\in H_p$ such that
\begin{equation}\label{nsd problem}
\left\{
\begin{aligned}
&(\partial_t \mathbf{u},\mathbf{v})_f +gS_0(\partial_t \phi,\psi)_p +\nu(\nabla \mathbf{u},\nabla \mathbf{v})_f
-(p,\nabla\cdot\mathbf{v})_f
+a(\mathbf{u},\mathbf{u},\mathbf{\mathbf{v}})
+\alpha\sqrt{\frac{\nu g}{tr(\mathbb{K})}}
\int_\Gamma(\mathbf{u}\cdot\boldsymbol{\tau})
(\mathbf{v}\cdot\boldsymbol{\tau})ds\\
&\qquad+g(\mathbb{K}\nabla\phi,\nabla\psi)_p
+c_\Gamma(\mathbf{v},\phi)
-c_\Gamma(\mathbf{u},\psi)
=(\mathbf{f}_f,\mathbf{v})_f +g(f_p,\psi)_p,
\qquad\forall\mathbf{v}\in H_f,~\psi\in H_p,\\
&(\nabla\cdot \mathbf{u},q)_f=0,\qquad\qquad~\forall q\in Q_f,\\
&(\mathbf{u}(\mathbf{x},0),\mathbf{v})_f=(\mathbf{u}^0,\mathbf{v})_f,\qquad\forall \mathbf{v}\in H_f,\\
&(\phi(\mathbf{x},0),\psi)_p=(\phi^0,\psi)_p,\qquad\forall \psi\in H_p,
\end{aligned}
\right.
\end{equation}
where
\begin{align*}
&a(\mathbf{u},\mathbf{v},\mathbf{w})=((\mathbf{u}\cdot\nabla) \mathbf{v},\mathbf{w})_f- \frac{1}{2} \int_\Gamma ( \mathbf{u}\cdot \mathbf{\mathbf{v}} )  (\mathbf{w}\cdot \mathbf{n} ) ds, \qquad\forall \mathbf{u},\mathbf{\mathbf{v}},\mathbf{w}\in H_f,\\
&c_\Gamma(\mathbf{v},\phi)=g\int_\Gamma\phi \mathbf{v}\cdot \mathbf{n} ds, \qquad\forall\mathbf{v}\in H_f,~\phi\in H_p.
\end{align*}
\medskip
\begin{lemma}\label{a=0}
Let $\mathbf{u}\in H_f$ and $\nabla\cdot \mathbf{u}=0$, then, we have \cite{wang2024class}
\begin{align} \label{an=0}
a(\mathbf{u},\mathbf{v},\mathbf{v})= 0,
\quad\forall \mathbf{v}\in H_f.
\end{align}
\end{lemma}
\medskip
\begin{lemma}\label{aEstimate}
For the nonlinear convective term $a(\mathbf{u},\mathbf{v},\mathbf{w})$, the following inequalities hold for $d=2,3$ \cite{shen1992on,Connors2012fluid}:
\begin{align}
&a(\mathbf{u},\mathbf{v},\mathbf{w})\leq C\|\mathbf{u}\|_{1,f}\|\mathbf{v}\|_{1,f}\|\mathbf{w}\|_{1,f},\\
&a(\mathbf{u},\mathbf{v},\mathbf{w})\leq C\|\mathbf{u}\|_{0,f}\|\mathbf{v}\|_{2,f}\|\mathbf{w}\|_{1,f},\\
&a(\mathbf{u},\mathbf{v},\mathbf{w})\leq C\|\mathbf{u}\|_{2,f}\|\mathbf{v}\|_{0,f}\|\mathbf{w}\|_{1,f};
\end{align}
and when  $d=2$, we have
\begin{align}
&\left|\int_\Gamma (\mathbf{u}\cdot \mathbf{v}) (\mathbf{w}\cdot \boldsymbol{n}_f) ds\right|\leq
C\|\mathbf{u}\|_{0,f}^{1/2}\|\mathbf{u}\|_{1,f}^{1/2}
\|\mathbf{v}\|_{0,f}^{1/2}\|\mathbf{v}\|_{1,f}^{1/2}
\|\mathbf{w}\|_{0,f}^{1/2}\|\mathbf{w}\|_{1,f}^{1/2},\label{gamma term}\\
&a(\mathbf{u},\mathbf{v},\mathbf{w})
\leq C\|\mathbf{u}\|_{0,f}^{1/2}\|\mathbf{u}\|_{1,f}^{1/2}
\|\mathbf{v}\|_{0,f}^{1/2}\|\mathbf{v}\|_{1,f}^{1/2}\|\mathbf{w}\|_{1,f},
\end{align}
where $C>0$ is a constant.
\end{lemma}

\subsection{Energy dissipation law}
We show below that the problem (\ref{nsd problem})  is energy dissipative if
$\mathbf{f}_f=\mathbf{0}$ and $f_p=0$. Indeed,  substituting $\mathbf{\mathbf{v}}=\mathbf{u}$ and $\psi=\phi$ into (\ref{nsd problem}), we obtain
\begin{align*}
&(\partial_t \mathbf{u},\mathbf{u})_f+gS_0(\partial_t \phi,\phi)_p
+\nu(\nabla \mathbf{u},\nabla \mathbf{u})_f+a(\mathbf{u},\mathbf{u},\mathbf{u})
+\alpha\sqrt{\frac{\nu g}{tr(\mathbb{K})}}
\int_\Gamma(\mathbf{u}\cdot\boldsymbol{\tau})^2 ds\\
&\qquad
+g(\mathbb{K}\nabla\phi,\nabla\phi)_p
+c_\Gamma({\mathbf{u}},\phi)
-c_\Gamma({\mathbf{u}},\phi)
=0,\qquad\forall\mathbf{v}\in H_f,~\psi\in H_p.
\end{align*}
Thanks to  (\ref{an=0}), we derive from the above that
\begin{align}\label{energy_id}
\frac{1}{2}\frac{d}{dt}\|\mathbf{u}\|^2_f
+\frac{gS_0}{2}\frac{d}{dt}\|\phi\|^2_p
=
-\nu(\nabla \mathbf{u},\nabla \mathbf{u})_f
-\alpha\sqrt{\frac{\nu g}{tr(\mathbb{K})}}
\int_\Gamma(\mathbf{u}\cdot\boldsymbol{\tau})^2 ds
-g(\mathbb{K}\nabla\phi,\nabla\phi)_p
\leq 0.
\end{align}
In other words, we have
\begin{align*}
 \frac{dE}{dt}\leq 0, \quad\text{ where } E=\frac{1}{2}\|\mathbf{u}\|^2_f+\frac{gS_0}{2}\|\phi\|^2_p.
\end{align*}

\section{Energy stable schemes} In order to construct efficient and unconditionally energy stable schemes,  we introduce a scalar auxiliary variable  $\xi(t)\equiv 1$, which serves as a relaxation factor, and expand the original system (\ref{nsd problem})  with a reformulated energy dissipation law:
\begin{align}\label{energy_id2}
	\frac{1}{2}\frac{d}{dt}\|\mathbf{u}\|^2_f
	+\frac{gS_0}{2}\frac{d}{dt}\|\phi\|^2_p
	=\xi(t)\Big(
	-\nu(\nabla \mathbf{u},\nabla \mathbf{u})_f
	-\alpha\sqrt{\frac{\nu g}{tr(\mathbb{K})}}
	\int_\Gamma(\mathbf{u}\cdot\boldsymbol{\tau})^2 ds
	-g(\mathbb{K}\nabla\phi,\nabla\phi)_p\Big).
\end{align}
\subsection{Scheme I: a first-order scheme}

 We  now construct   a first-order scheme for  the expanded system, i.e.,  (\ref{nsd problem})  and \eqref{energy_id2}, using a prediction--correction approach as follows:

{\it Step 1 (Prediction).} Given $(\mathbf{u}^n,p^n,\phi^n)$ and $(\tilde{\mathbf{u}}^n,\tilde{\phi}^n)$, for any $(\mathbf{v},q,\psi)\in H_{f}\times Q_{f}\times H_{p}$, find $(\tilde{\mathbf{u}}^{n+1},p^{n+1},\tilde{\phi}^{n+1})$ satisfying
\begin{align}
&\left(\frac{\tilde{\mathbf{u}}^{n+1}-\tilde{\mathbf{u}}^n}{\Delta t},\mathbf{\mathbf{v}}\right)_f
+\nu(\nabla \tilde{\mathbf{u}}^{n+1},\nabla \mathbf{\mathbf{v}})_f
+{\alpha \sqrt{\frac{\nu g}{tr(\mathbb{K})}}}\int_{\Gamma}
(\tilde{\mathbf{u}}^{n+1}\cdot\boldsymbol{\tau})
(\mathbf{\mathbf{v}}\cdot \boldsymbol{\tau})ds
-(p^{n+1},\nabla\cdot \mathbf{\mathbf{v}})_f\nonumber\\
&+a(\mathbf{u}^n,\mathbf{u}^n,\mathbf{\mathbf{v}})
+c_\Gamma(\mathbf{\mathbf{v}},\phi^n)
=(\mathbf{f}_f^{n+1},\mathbf{\mathbf{v}})_f,\label{A1ns1}\\
&(\nabla\cdot \tilde{\mathbf{u}}^{n+1},q)_f=0,\label{A1ns2}\\
&gS_0\left(\frac{\tilde{\phi}^{n+1}-\tilde{\phi}^n}{\Delta t},\psi\right)_p
+g({\mathbb{K}}\nabla\tilde{\phi}^{n+1},\nabla\psi)_p
-c_\Gamma(\mathbf{u}^n,\psi)
=g(f_p^{n+1},\psi)_p.\label{A1darcy}
\end{align}

%
{\it Step 2 (Correction).} We denote
\begin{align}\label{modiE}
	\tilde{E}^{n+1}=\frac{1}{2}\|\tilde{\mathbf{u}}^{n+1}\|_f^2
	+\frac{gS_0}{2}\|\tilde{\phi}^{n+1}\|_p^2,
\end{align} and  set
\begin{align}\label{oriE}
	E^{n+1}=\xi^{n+1} \tilde{E}^{n+1},
\end{align}
where  $\xi^{n+1}$, which is an approximation of 1,  satisfies the linear algebraic equation
\begin{align}\label{xi_n+1}
\frac{{E}^{n+1}-E^n}{\Delta t}
=-\xi^{n+1}\left( \nu \|\nabla \tilde{\mathbf{u}}^{n+1}\|_f^2
+\alpha \sqrt{\frac{\nu g}{tr(\mathbb{K})}}\int_\Gamma (\tilde{\mathbf{u}}^{n+1}\cdot\boldsymbol{\tau})^2 ds
+g\|\sqrt{\mathbb{K}}\nabla\tilde{\phi}^{n+1}\|_p^2 \right).
\end{align}
Finally, we update $ \tilde{\mathbf{u}}^{n+1}$ and $ \tilde{\phi}^{n+1}$ by
\begin{align}
  \mathbf{u}^{n+1} =\sqrt{\xi^{n+1}} \tilde{\mathbf{u}}^{n+1},\qquad
  \phi^{n+1}  =\sqrt{\xi^{n+1}}  \tilde{\phi}^{n+1} .\label{update_u_phi}
\end{align}
\begin{theorem}\label{Stable}
If $E^n\ge 0$,  we have $\xi^{n+1}\ge 0$, $E^{n+1}\ge 0$. Hence, the scheme (\ref{A1ns1})-(\ref{update_u_phi})  dissipates the original energy as shown by \eqref{xi_n+1}. Moreover, we have
	\begin{align}\label{lemmaE}
	\frac{1}{2}\|\mathbf{u}^{k+1}\|_f^2
	+\frac{gS_0}{2}\|\phi^{k+1}\|_p^2
	+ & \Delta t\sum_{n=0}^{k} \left(\nu \|\nabla \mathbf{u}^{n+1}\|_f^2
	+\alpha \sqrt{\frac{\nu g}{tr(\mathbb{K})}}\int_\Gamma (\mathbf{u}^{n+1}\cdot\boldsymbol{\tau})^2 ds
+g\|\sqrt{\mathbb{K}}\nabla\phi^{n+1}\|_p^2\right) \nonumber\\
   {\color{black} =}& \frac{1}{2}\| \mathbf{u} ^{0}\|_f^2
	+\frac{gS_0}{2}\| \phi^{0}\|_p^2 \leq C,
\end{align}
{\color{black}where $C>0$ is a constant that depends on the initial conditions $\phi^{0}$ and $\mathbf{u}^{0}$.}
\end{theorem}
\medskip
\begin{proof}
Since $E^n\ge 0$, $\tilde{E}^{n+1}\ge 0$ and $E^{n+1}=\xi^{n+1} \tilde{E}^{n+1}$, 	one derives immediately from \eqref{xi_n+1} that
	\begin{align} \label{e_positivity_xi}
	\xi^{n+1}=\frac{{E}^n}{\tilde{E}^{n+1}
		+\Delta t \mathbb{I}(\tilde{\mathbf{u}}^{n+1},\tilde{\phi}^{n+1})} \ge 0,
\end{align}
	where ${\mathbb{I}}(\tilde{\mathbf{u}}^{n+1},\tilde{\phi}^{n+1})=
\nu \|\nabla {\tilde{\mathbf{u}}}^{n+1}\|_f^2
+\alpha \sqrt{\frac{\nu g}{tr(\mathbb{K})}}\int_\Gamma (\tilde{\mathbf{u}}^{n+1}\cdot\boldsymbol{\tau})^2 ds
+g \|\sqrt{\mathbb{K}}\nabla\tilde{\phi}^{n+1}\|_p^2 \geq 0$.

	Combining \eqref{xi_n+1} and \eqref{e_positivity_xi}, we find
	\begin{align}\label{lemE}
		\frac{E^{n+1}-E^n}{\Delta t}=-\xi^{n+1}{\mathbb{I}}
		(\tilde{\mathbf{u}}^{n+1},\tilde{\phi}^{n+1})\leq 0,
	\end{align}
	which indicates that the original energy is dissipative.
	
On the other hand, summing up  (\ref{xi_n+1})  from $n=0$ to $n=k$, we obtain
	\begin{align*}
		E^{k+1} +\Delta t\sum_{n=0}^{k}\xi^{n+1}\mathbb{I}(\tilde{\mathbf{u}}^{n+1},\tilde{\phi}^{n+1})
		{\color{black}= E^0 \leq C},
	\end{align*}
	which implies \eqref{lemmaE}.
\end{proof}
\subsection{{\textbf {Scheme \uppercase\expandafter{\romannumeral2}}}: a second-order scheme}
Similarly, we can construct a second-order scheme based on the Crank-Nicolson and Adam-Bashforth formula as follows:
\medskip

{\it Step 1 (Prediction).} Given $({\mathbf{u}}^{n-1},{\phi}^{n-1})$, $(\mathbf{u}^n,p^n,\phi^n)$ and $(\tilde{\mathbf{u}}^n,\tilde{\phi}^n)$, for any $(\mathbf{v},q,\psi)\in H_{f}\times Q_{f}\times H_{p}$, find $(\tilde{\mathbf{u}}^{n+1}$, $p^{n+1}$, $\tilde{\phi}^{n+1})$ satisfying
\begin{align}
	&\left(\frac{\tilde{\mathbf{u}}^{n+1}-\tilde{\mathbf{u}}^n}{\Delta t},\mathbf{\mathbf{v}}\right)_f
	+\nu\left(\nabla\left(\frac{\tilde{\mathbf{u}}^{n+1}
		+{\mathbf{u}}^{n}}{2}\right),\nabla \mathbf{\mathbf{v}} \right)_f
	+{ \alpha \sqrt{\frac{\nu g}{tr(\mathbb{K})}}}\int_{\Gamma}
	\left(\frac{\tilde{\mathbf{u}}^{n+1}+\mathbf{u}^{n}}{2}
	\cdot\boldsymbol{\tau}\right)
	(\mathbf{\mathbf{v}}\cdot \boldsymbol{\tau})ds\nonumber\\
	&
	-\left(\frac{p^{n+1}+p^{n}}{2},\nabla\cdot \mathbf{\mathbf{v}}\right)_f
	+a(\tilde{\mathbf{u}}^{n+\frac{1}{2}},\tilde{\mathbf{u}}^{n+\frac{1}{2}},\mathbf{\mathbf{v}})
	+c_\Gamma(\mathbf{\mathbf{v}},\tilde{\phi}^{n+\frac{1}{2}})
	=\left(\frac{\mathbf{f}_f^{n+1}+\mathbf{f}_f^{n}}{2},
	\mathbf{\mathbf{v}}\right)_f,\label{2A1ns1}\\
	&(\nabla\cdot \tilde{\mathbf{u}}^{n+1},q)_f=0,\label{2A1ns2}\\
	&gS_0\left(\frac{\tilde{\phi}^{n+1}-\tilde{\phi}^n}{\Delta t},\psi\right)_p
	+{g}\left({\mathbb{K}}\nabla\frac{\tilde{\phi}^{n+1}
		+{\phi}^{n}}{2},\nabla\psi\right)_p
	-c_\Gamma(\tilde{\mathbf{u}}^{n+\frac{1}{2}},\psi)
	={g}\left(\frac{f_p^{n+1}+f_p^{n}}{2},\psi\right)_p,\label{2A1darcy}
\end{align}
where  $\tilde{\mathbf{u}}^{n+\frac{1}{2}}=\frac{3}{2} \mathbf{u}^n-\frac{1}{2}\mathbf{u}^{n-1}$ and $\tilde{\phi}^{n+\frac{1}{2}}=\frac{3}{2}\phi^n-\frac{1}{2}\phi^{n-1}$.

{\it Step 2 (Correction).} We denote
\begin{align}\label{modiE2}
	\tilde{E}^{n+1}=\frac{1}{2}\|\tilde{\mathbf{u}}^{n+1}\|_f^2
	+\frac{gS_0}{2}\|\tilde{\phi}^{n+1}\|_p^2,
\end{align} and  set
\begin{align}\label{oriE2}
	E^{n+1}=\xi^{n+1} \tilde{E}^{n+1},
\end{align}
where  $\xi^{n+1}$, which is an approximation of 1,  satisfies the linear algebraic equation
\begin{align}\label{2xi_n+1}
	&\frac{{E}^{n+1}- E^n}{\Delta t}
	=-\xi^{n+1} \Bigg(  \nu \left\|\nabla \left( \frac{\tilde{\mathbf{u}}^{n+1}
		+{\mathbf{u}}^{n}}{2} \right) \right\|_f^2
	+\alpha \sqrt{\frac{\nu g}{tr(\mathbb{K})}}\int_\Gamma \left( \frac{\tilde{\mathbf{u}}^{n+1}
		+{\mathbf{u}}^{n}}{2} \cdot\boldsymbol{\tau} \right)^2 ds\notag\\
	&
	+g \left\|\sqrt{\mathbb{K}} \nabla\frac{\tilde{\phi}^{n+1}
		+{\phi}^{n}}{2} \right\|_p^2 \Bigg).
\end{align}
Finally, we update $ \tilde{\mathbf{u}}^{n+1}$ and $ \tilde{\phi}^{n+1}$ by
\begin{align}
	\mathbf{u}^{n+1} =\sqrt{\xi^{n+1}} \tilde{\mathbf{u}}^{n+1},\qquad
	\phi^{n+1}  =\sqrt{\xi^{n+1}}  \tilde{\phi}^{n+1} .\label{update_u_phi2}
\end{align}
Similarly as in the proof of Theorem \ref{Stable}, we can establish the following result:
\begin{theorem}\label{Stable2}
If $E^n\ge 0$,  we have $\xi^{n+1}\ge 0$, $E^{n+1}\ge 0$.  Hence, the scheme   (\ref{2A1ns1})-(\ref{update_u_phi2})  dissipates the original energy as shown by  \eqref{2xi_n+1}.
\end{theorem}


\medskip

\section{Error analysis}
The primary goal of this section is to carry out an error analysis for  {\textbf {Scheme \uppercase\expandafter{\romannumeral1}}} (\ref{A1ns1})-(\ref{update_u_phi}). Hereafter, we use $C>0$ to denote the general constant, independent of $\Delta t$, whose value may differ in different contexts. For the sake of simplicity, we assume that the hydraulic conductivity tensor $\mathbb{K}=K\mathbb{I}$ during the theoretical analysis, where $K>0$ is a scalar coefficient and $\mathbb{I}$ denotes the identity tensor.

Assume that the exact solution and initial data of problem (\ref{nsd problem}) satisfy the following regularity properties:
\begin{equation}\label{regu_pro}
\begin{cases}
\mathbf{u}\in L^\infty(0,T;H^2(\Omega_f)),\quad \partial_t \mathbf{u}\in L^2(0,T;H^2(\Omega_f)),\quad
\partial _{tt} \mathbf{u}\in L^2(0,T;H^{-1}(\Omega_f)),\\
p\in L^2(0,T;H^1(\Omega_f)),\\
\phi\in L^\infty(0,T;H^2(\Omega_p)),\quad \partial_t\phi\in L^2(0,T;H^2(\Omega_p)),
\quad \partial_{tt}\phi\in L^2(0,T;H^{-1}(\Omega_p)),\\
\mathbf{u}(t^0)\in L^2(\Omega_f),\quad
\phi(t^0)\in L^2(\Omega_p),\quad
\mathbf{f}_f\in L^2(0,T;H^{-1}(\Omega_f)),
\quad f_p\in L^2(0,T;H^{-1}(\Omega_p)).
\end{cases}
\end{equation}

For the subsequent error analysis, we restrict our attention to the two-dimensional case. The extension to three-dimension is more challenging and  will be addressed in future work.
\subsection{A rough error estimate}
The errors between the numerical and exact solutions are defined as follows:
\begin{align*}
&e_E^{n+1}=E^{n+1}-E(t^{n+1}),\quad
e_{\mathbf{u}}^{n+1}=\mathbf{u}^{n+1}-\mathbf{u}(t^{n+1}),\quad
e_\phi^{n+1}=\phi^{n+1}-\phi(t^{n+1}),\\
&\tilde{e}_E^{n+1}=\tilde{E}^{n+1}-E(t^{n+1}),\quad
\tilde{e}_{\mathbf{u}}^{n+1}=\tilde{\mathbf{u}}^{n+1}-\mathbf{u}(t^{n+1}),\quad
\tilde{e}_\phi^{n+1}=\tilde{\phi}^{n+1}-\phi(t^{n+1}).
\end{align*}

We set $\tilde{\mathbf{u}}^0=\mathbf{u}(t^0)$ and $\tilde{\phi}^0=\phi(t^0)$, so it is easy to see that
\begin{align*}
\|\tilde{e}_{\mathbf{u}}^0\|_f+\|\tilde{e}_{\phi}^{0}\|_p=0,\quad |e_{E}^0|=0.
\end{align*}

\medskip
\begin{lemma}\label{prior}
Given $1\le k\le T/{\Delta t}-1$, we assume that the following error estimates hold
	\begin{align}
		& \|\tilde{e}_{\mathbf{u}}^m\|_f
		+\|\tilde{e}_{\phi}^{m}\|_p \leq \hat{C}\Delta t^{\frac{1}{4}},\quad m=1,\cdots,k,\label{assum_u_phi}
		\\
		& |e_{E}^m|\leq C^* \Delta t^{1-\epsilon},\quad 0< \epsilon< 1/2,\quad m=1,\cdots,k,\label{assum_E}
	\end{align}
	where $\hat{C}>0$ and $C^*>0$ are two fixed constants  independent of $\Delta t$.
	Then, under the regularity assumption (\ref{regu_pro}),  and for $\Delta t$ sufficiently small, we have
\begin{align}\label{ess_u_phi}
\|\tilde{e}_{\mathbf{u}}^{k+1}\|_f
+ \|\tilde{e}_{\phi}^{k+1}\|_p
\leq \hat{C}\Delta t^{\frac{1}{4}},
\end{align}
where $\hat{C}>0$ is the constant in \eqref{assum_u_phi}.
\end{lemma}
\medskip
\begin{proof}
Let $1\le n\le k$. Based on the assumptions (\ref{assum_u_phi})-(\ref{assum_E}),
we can derive that
\begin{align}\label{boundEn}
0<\frac{1}{2}E(t^n)\leq \tilde{E}^n\leq \frac{3}{2}E(t^n).
\end{align}
Also, from Theorem \ref{Stable} and the regularity properties (\ref{regu_pro}), we have
\begin{align}\label{xin_bound}
\xi^n=\frac{E^n}{\tilde{E}^n}\leq \frac{2E^n}{E(t^n)}\leq C_1^2,
\end{align}
where $C_1>0$ is a fixed constant independent of $\Delta t$.
We derive from \eqref{nsd problem} that
\begin{align}
&\left(\frac{\mathbf{u}(t^{n+1})-\mathbf{u}(t^n)}{\Delta t},\mathbf{v}\right)_f
+\nu (\nabla \mathbf{u}(t^{n+1}),\nabla \mathbf{v})_f
+\alpha\sqrt{\frac{\nu g}{tr(\mathbb{K})}}\int_\Gamma (\mathbf{u}(t^{n+1})\cdot\boldsymbol{\tau})
(\mathbf{v}\cdot\boldsymbol{\tau})ds\notag\\
&
-(p(t^{n+1}),\nabla\cdot \mathbf{v})_f
+a(\mathbf{u}(t^{n+1}),\mathbf{u}(t^{n+1}),\mathbf{v})
+c_\Gamma(\mathbf{v},\phi(t^{n+1}))\notag\\
&
=(\mathbf{f}_f^{n+1},\mathbf{v})_f +R_u^{n+1}(\mathbf{v}),\label{exact_NS}
\end{align}
where $R_u^{n+1}(\mathbf{v})=\left(\frac{\mathbf{u}(t^{n+1})-\mathbf{u}(t^n)}{\Delta t}-\frac{\partial \mathbf{u}(t^{n+1})}{\partial t},\mathbf{v}\right)_f$.
Subtracting (\ref{exact_NS}) from (\ref{A1ns1}), we can obtain
\begin{align*}
&\left( \frac{\tilde{e}_{\mathbf{u}}^{n+1}-\tilde{e}_{\mathbf{u}}^n}{\Delta t},\mathbf{v} \right)_f
+\nu ( \nabla \tilde{e}_{\mathbf{u}}^{n+1},
\nabla \mathbf{v} )_f
+\alpha\sqrt{\frac{\nu g}{tr(\mathbb{K})}}\int_\Gamma (\tilde{e}_{\mathbf{u}}^{n+1}\cdot\boldsymbol{\tau})
(\mathbf{v}\cdot\boldsymbol{\tau})ds \\
&
-(p^{n+1}-p(t^{n+1}),\nabla\cdot\mathbf{v})_f
+a(\mathbf{u}^n,\mathbf{u}^n,\mathbf{v})
-a(\mathbf{u}(t^{n+1}),\mathbf{u}(t^{n+1}),\mathbf{v})\\
&
+c_\Gamma(\mathbf{v},\phi^n)
-c_\Gamma(\mathbf{v},\phi(t^{n+1}))
=-R_{\mathbf{u}}^{n+1}(\mathbf{v}).
\end{align*}
Taking $\mathbf{v}=\tilde{e}_{\mathbf{u}}^{n+1}$ in the above equation, we get
\begin{align}
&\frac{1}{2\Delta t}\left( \|\tilde{e}_{\mathbf{u}}^{n+1}\|_f^2
-\|\tilde{e}_{\mathbf{u}}^{n}\|_f^2 +\|\tilde{e}_{\mathbf{u}}^{n+1}-\tilde{e}_{\mathbf{u}}^{n}\|_f^2 \right)
+\nu \| \nabla \tilde{e}_{\mathbf{u}}^{n+1}\|_f^2
+\alpha\sqrt{\frac{\nu g}{tr(\mathbb{K})}}\int_\Gamma (\tilde{e}_{\mathbf{u}}^{n+1}\cdot\boldsymbol\tau)^2 ds\notag\\
&
=-R_{\mathbf{u}}^{n+1}(\tilde{e}_{\mathbf{u}}^{n+1})
-a(\mathbf{u}^n,\mathbf{u}^n,\tilde{e}_{\mathbf{u}}^{n+1})
+a(\mathbf{u}(t^{n+1}),\mathbf{u}(t^{n+1}),\tilde{e}_{\mathbf{u}}^{n+1})
-c_\Gamma(\tilde{e}_{\mathbf{u}}^{n+1},\phi^n)
+c_\Gamma(\tilde{e}_{\mathbf{u}}^{n+1},\phi(t^{n+1}))\notag\\
&
=\sum_{i=1}^{3}I_i.\label{NSerr1}
\end{align}
Now, we aim to estimate the right-hand side terms of (\ref{NSerr1}) successively. Based on Lemma \ref{aEstimate}, Cauchy-Schwarz and Young's inequalities, we get
\begin{align*}
|I_1|=|-R_{\mathbf{u}}^{n+1}(\tilde{e}_{\mathbf{u}}^{n+1})|
\leq \epsilon_{\mathbf{u}} \nu \|\nabla\tilde{e}_{\mathbf{u}}^{n+1}\|_f^2
+C \Delta t \int_{t^n}^{t^{n+1}} \|\partial_{tt}\mathbf{u}\|_{-1,f}^2 dt,
\end{align*}
\begin{align*}
|I_2|
&=| -a(\mathbf{u}^n,\mathbf{u}^n,\tilde{e}_{\mathbf{u}}^{n+1})
+a(\mathbf{u}(t^{n+1}),\mathbf{u}(t^{n+1}),
\tilde{e}_{\mathbf{u}}^{n+1})
|\\
&=
|a(\mathbf{u}(t^{n+1})-\mathbf{u}(t^n),\mathbf{u}(t^{n+1}), \tilde{e}_{\mathbf{u}}^{n+1} )
-a(e_{\mathbf{u}}^n,\mathbf{u}(t^{n+1}),\tilde{e}_{\mathbf{u}}^{n+1})
+a(\mathbf{u}^n,\mathbf{u}(t^{n+1})-\mathbf{u}(t^n),\tilde{e}_{\mathbf{u}}^{n+1} )\\
&\quad
-a(e_{\mathbf{u}}^n,e_{\mathbf{u}}^n,\tilde{e}_{\mathbf{u}}^{n+1})
-a(\mathbf{u}(t^n),e_{\mathbf{u}}^n,\tilde{e}_{\mathbf{u}}^{n+1})|\\
&
\leq
C\|\mathbf{u}(t^{n+1})-\mathbf{u}(t^n)\|_f
\|\mathbf{u}(t^{n+1})\|_{2,f}
\|\nabla \tilde{e}_{\mathbf{u}}^{n+1}\|_f
+C\|e_{\mathbf{u}}^n\|_f\|\mathbf{u}(t^{n+1})\|_{2,f}
\|\nabla\tilde{e}_{\mathbf{u}}^{n+1}\|_f\\
&\quad
+C\|\mathbf{u}^n\|_f\|\mathbf{u}(t^{n+1})-\mathbf{u}(t^n)\|_{2,f}
\|\nabla\tilde{e}_{\mathbf{u}}^{n+1}\|_f
+C\|e_{\mathbf{u}}^n\|_f^{1/2}\|\nabla e_{\mathbf{u}}^n\|_f^{1/2}
\|e_{\mathbf{u}}^n\|_f^{1/2}\|\nabla e_{\mathbf{u}}^n\|_f^{1/2}
\|\nabla\tilde{e}_{\mathbf{u}}^{n+1}\|_f\\
&\quad
+C\|\mathbf{u}(t^n)\|_{2,f}\|e_{\mathbf{u}}^n\|_{f}\|
\nabla\tilde{e}_{\mathbf{u}}^{n+1} \|_f\\
&
\leq
\epsilon_{\mathbf{u}} \nu \|\nabla\tilde{e}_{\mathbf{u}}^{n+1} \|_f^2
+C\|e_{\mathbf{u}}^n\|_f^2
+C\|\nabla {e}_{\mathbf{u}}^{n}\|_f^2\|e_{\mathbf{u}}^n\|_f^2
+C\Delta t \int_{t^n}^{t^{n+1}} \|\partial_t \mathbf{u}\|_f^2 dt
+C\Delta t \int_{t^n}^{t^{n+1}} \|\partial_t \mathbf{u}\|_{2,f}^2 dt,
\end{align*}
\begin{align*}
|I_3|&=|-c_\Gamma(\tilde{e}_{\mathbf{u}}^{n+1},\phi^n)
+c_\Gamma(\tilde{e}_{\mathbf{u}}^{n+1},\phi(t^{n+1}))|\\
&\leq |c_\Gamma(\tilde{e}_{\mathbf{u}}^{n+1},\phi(t^{n+1})-\phi(t^n))|
+|c_\Gamma(\tilde{e}_{\mathbf{u}}^{n+1},e_\phi^n)|\\
&\leq
\epsilon_{\mathbf{u}}\nu \|\nabla \tilde{e}_{\mathbf{u}}^{n+1}\|_f^2
+C\|e_\phi^n\|_p \|\nabla e_\phi^n\|_p
+C\Delta t \int_{t^n}^{t^{n+1}} \|\partial_t\phi\|_{1,p}^2dt.
\end{align*}
We notice that there are some remaining terms that still need to be analyzed. Based on the regularity properties (\ref{regu_pro}) and the boundedness of $\tilde{E}^n$ and $\xi^n$ in \eqref{boundEn}-\eqref{xin_bound}, we have
\begin{align}
\|e_{\mathbf{u}}^n\|_f
= \|\sqrt{\xi^n}\tilde{\mathbf{u}}^n-\mathbf{u}(t^{n})\|_f
&\leq
\|\sqrt{\xi^n}(\tilde{\mathbf{u}}^n-\mathbf{u}(t^n))\|_f
+\|(\sqrt{\xi^n}-1) \mathbf{u}(t^n)\|_f\notag\\
&\leq
C_1\|\tilde{e}_{\mathbf{u}}^n\|_f +C\left|\frac{\xi^n-1}{\sqrt{\xi^n}+1}\right|\notag\\
&\leq
C_1\|\tilde{e}_{\mathbf{u}}^n\|_f
+C\left|\frac{E^n-E(t^n)}{\tilde{E}^n}\right|
+C\left|\frac{E(t^n)-\tilde{E}^n}{\tilde{E}^n}\right|\notag\\
&
\leq
C \|\tilde{e}_{\mathbf{u}}^n\|_f
+C\|\tilde{e}_{\phi}^n\|_p
+C|e_E^n|
.\label{err_uu}
\end{align}
Similarly, we can also obtain that
\begin{align}
&\|e_{\phi}^n\|_p \leq C \|\tilde{e}_{\mathbf{u}}^n\|_f
+C\|\tilde{e}_{\phi}^n\|_p +C|e_E^n|,\label{e_phin}\\
&\|\nabla e_{\mathbf{u}}^n\|_f \leq C_1\|\nabla \tilde{e}_{\mathbf{u}}^n\|_f
+C \|\tilde{e}_{\mathbf{u}}^n\|_f
+C\|\tilde{e}_{\phi}^n\|_p +C|e_E^n|,\label{e_nablaun}\\
&\|\nabla e_{\phi}^n\|_p \leq C_1\|\nabla \tilde{e}_{\phi}^n\|_p
+C \|\tilde{e}_{\mathbf{u}}^n\|_f
+C\|\tilde{e}_{\phi}^n\|_p +C|e_E^n|.\label{e_nabla_phin}
\end{align}
Then we get
\begin{align}
&\frac{1}{2\Delta t} \left(\|\tilde{e}_{\mathbf{u}}^{n+1}\|_f^2
-\|\tilde{e}_{\mathbf{u}}^{n}\|_f^2
+\|\tilde{e}_{\mathbf{u}}^{n+1}-\tilde{e}_{\mathbf{u}}^{n}\|_f^2\right)
+\nu \|\nabla \tilde{e}_{\mathbf{u}}^{n+1}\|_f^2
+\alpha\sqrt{\frac{\nu g}{tr(\mathbb{K})}}\int_\Gamma (\tilde{e}_{\mathbf{u}}^{n+1}\cdot \boldsymbol{\tau})^2ds\notag\\
&\leq
3 \epsilon_{\mathbf{u}} \nu \|\nabla\tilde{e}_{\mathbf{u}}^{n+1}\|_f^2
+\epsilon_\phi gK \|\nabla\tilde{e}_\phi^n\|_p^2
+C\|\tilde{e}_{\mathbf{u}}^n\|_f^2
+C|e_E^n|^2 +C\|\tilde{e}_\phi^n\|_p^2
+C\|\nabla{e}_{\mathbf{u}}^{n} \|_f^2 \|\tilde{e}_{\mathbf{u}}^n\|_f^2\notag\\
&\quad
+C\|\nabla {e}_{\mathbf{u}}^{n} \|_f^2 \|\tilde{e}_{\phi}^n\|_p^2
+C\|\nabla {e}_{\mathbf{u}}^{n} \|_f^2 |e_E^n|^2
+C\Delta t \int_{t^n}^{t^{n+1}} \|\partial_{tt}\mathbf{u}\|_{-1,f}^2dt
+C\Delta t \int_{t^n}^{t^{n+1}} \|\partial_t\mathbf{u}\|_f^2 dt\notag\\
&\quad
+C\Delta t \int_{t^n}^{t^{n+1}} \|\partial_t \mathbf{u}\|_{2,f}^2 dt
+C\Delta t \int_{t^n}^{t^{n+1}} \|\partial_t \phi\|_{1,p}^2 dt.\label{err_ns}
\end{align}

Now, we focus on the error estimate of the hydraulic head. Consider the equation satisfied by the exact solution $\phi(t^{n+1})$.
\begin{align}
gS_0\left(\frac{\phi(t^{n+1})-\phi(t^n)}{\Delta t},\psi\right)_p
+g(\mathbb{K}\nabla\phi(t^{n+1}),\nabla\psi)_p
-c_\Gamma (\mathbf{u}(t^{n+1}),\psi)
=g(f_p^{n+1},\psi)_p +R_\phi^{n+1}(\psi)
,\label{exactDa}
\end{align}
where $R_\phi^{n+1}(\psi)=gS_0 \left(\frac{\phi(t^{n+1})-\phi(t^{n})}{\Delta t}-\frac{\partial\phi(t^{n+1})}{\partial t},\psi \right)_p$.
Subtracting (\ref{exactDa}) from (\ref{A1darcy}), we obtain
\begin{align*}
gS_0\left(\frac{\tilde{e}_{\phi}^{n+1}-\tilde{e}_{\phi}^n}{\Delta t},\psi\right)_p
+g(\mathbb{K}\nabla\tilde{e}_\phi^{n+1},\nabla\psi)_p
-c_\Gamma (\mathbf{u}^n,\psi)
+c_\Gamma (\mathbf{u}(t^{n+1}),\psi)
=-R_\phi^{n+1}(\psi).
\end{align*}
Substituting $\psi=\tilde{e}_\phi^{n+1}$ into the above equation yields
\begin{align*}
&\frac{gS_0}{2\Delta t} \left( \|\tilde{e}_\phi^{n+1}\|_p^2
-\|\tilde{e}_\phi^n\|_p^2 + \|\tilde{e}_\phi^{n+1}-\tilde{e}_\phi^n\|_p^2 \right)
+gK \|\nabla\tilde{e}_\phi^{n+1}\|_p^2
\\
&
=-R_\phi^{n+1}(\tilde{e}_\phi^{n+1})
+c_\Gamma(\mathbf{u}^n,\tilde{e}_\phi^{n+1})
-c_\Gamma(\mathbf{u}(t^{n+1}),\tilde{e}_\phi^{n+1}).
\end{align*}
Using Cauchy-Schwarz and Young's inequalities, we have
\begin{align*}
|I_4|=|-R_\phi^{n+1}(\tilde{e}_\phi^{n+1})|
\leq \epsilon_\phi gK\|\nabla \tilde{e}_\phi^{n+1}\|_p^2
+C\Delta t \int_{t^n}^{t^{n+1}} \|\partial_{tt} \phi\|_{-1,p}^2 dt,
\end{align*}
\begin{align*}
|I_5|&=|c_\Gamma(\mathbf{u}^n,\tilde{e}_\phi^{n+1})
-c_\Gamma(\mathbf{u}(t^{n+1}),\tilde{e}_\phi^{n+1})|\\
&\leq
|c_\Gamma(e_{\mathbf{u}}^n,\tilde{e}_{\phi}^{n+1})|
+|c_\Gamma(\mathbf{u}(t^{n+1})-\mathbf{u}(t^n)
,\tilde{e}_\phi^{n+1})|\\
&\leq
\epsilon_\phi gK \|\nabla \tilde{e}_\phi^{n+1}\|_p^2
+ C\|\nabla e_{\mathbf{u}}^n\|_f\|e_{\mathbf{u}}^n\|_f
+C\Delta t \int_{t^n}^{t^{n+1}} \|\partial _t\mathbf{u}\|_{1,f}^2 dt\\
&\leq
\epsilon_\phi gK \|\nabla \tilde{e}_\phi^{n+1}\|_p^2
+C\|e_{\mathbf{u}}^n\|_f^2
+\frac{\epsilon_{\mathbf{u}}\nu }{C_1^2}\|\nabla e_{\mathbf{u}}^n\|_f^2
+C\Delta t \int_{t^n}^{t^{n+1}} \|\partial _t\mathbf{u}\|_{1,f}^2 dt
.
\end{align*}
Recalling  (\ref{e_phin})-(\ref{e_nabla_phin}), then we have the following result
\begin{align}
&\frac{g S_0}{2\Delta t} \left (\|\tilde{e}_\phi^{n+1}\|_p^2
-\|\tilde{e}_\phi^{n}\|_p^2
+\|\tilde{e}_\phi^{n+1}-\tilde{e}_\phi^{n}\|_p^2\right)
+gK \|\nabla\tilde{e}_\phi^{n+1}\|_p^2
\notag\\
&\leq
2\epsilon_\phi gK \|\nabla \tilde{e}_\phi^{n+1}\|_p^2
+\epsilon_{\mathbf{u}}\nu \|\nabla \tilde{e}_{\mathbf{u}}^n \|_f^2
+C\|\tilde{e}_{\mathbf{u}}^n\|_f^2
+C\|\tilde{e}_{\phi}^n\|_p^2
+C|e_E^n|^2\notag\\
&\quad
+C\Delta t \int_{t^n}^{t^{n+1}} \|\partial_{tt}\phi\|_{-1,p}^2 dt
+C\Delta t \int_{t^n}^{t^{n+1}} \|\partial_t \mathbf{u}\|_{1,f}^2 dt
.
\label{errdarcy}
\end{align}
Adding (\ref{err_ns}) and (\ref{errdarcy}) together, taking $\epsilon_{\mathbf{u}}=\frac{1}{8},~\epsilon_\phi=\frac{1}{6}$, summing them with $n$ ranging from $0$ to $k$,
and then multiplying both sides by $2 \Delta t$, we can obtain that
\begin{align}
&\|\tilde{e}_{\mathbf{u}}^{k+1}\|_f^2
+\sum_{n=0}^{k} \nu \Delta t \|\nabla \tilde{e}_{\mathbf{u}}^{n+1} \|_f^2
+\sum_{n=0}^{m}2\alpha\sqrt{\frac{\nu g}{tr(\mathbb{K})}}\Delta t\int_\Gamma (\tilde{e}_{\mathbf{u}}^{n+1}\cdot \boldsymbol{\tau})^2ds
+gS_0\|\tilde{e}_\phi^{k+1}\|_p^2
+\sum_{n=0}^{k} gK \Delta t \|\nabla \tilde{e}_\phi^{n+1}\|_p^2
\notag\\
&\leq
\sum_{n=0}^{k} C\Delta t \|\tilde{e}_{\mathbf{u}}^{n}\|_f^2
+\sum_{n=0}^{k} C\Delta t \|\tilde{e}_{\phi}^{n}\|_p^2
+\sum_{n=0}^{k} C\Delta t |{e}_{E}^{n}|^2
+\sum_{n=0}^{k} C \Delta t \|\nabla e_{\mathbf{u}}^n \|_f^2
\|\tilde{e}_{\mathbf{u}}^n\|_f^2\notag\\
&\quad
+\sum_{n=0}^{k} C \Delta t \|\nabla e_{\mathbf{u}}^n \|_f^2
\|\tilde{e}_{\phi}^n\|_f^2
+\sum_{n=0}^{k} C \Delta t \|\nabla e_{\mathbf{u}}^n \|_f^2
|e_E^n|^2
+\sum_{n=0}^{k} \Delta t \Big\{
C\Delta t \int_{t^n}^{t^{n+1}} \|\partial_{tt}\mathbf{u}\|_{-1,f}^2dt\notag\\
&\quad
+C\Delta t \int_{t^n}^{t^{n+1}} \|\partial_t\mathbf{u}\|_f^2 dt
+C\Delta t \int_{t^n}^{t^{n+1}} \|\partial_t \mathbf{u}\|_{2,f}^2 dt
+C\Delta t \int_{t^n}^{t^{n+1}} \|\partial_t \phi\|_{1,p}^2 dt\notag\\
&\quad
+C\Delta t \int_{t^n}^{t^{n+1}} \|\partial_{tt}\phi\|_{-1,p}^2 dt
+C\Delta t \int_{t^n}^{t^{n+1}} \|\partial_t \mathbf{u}\|_{1,f}^2 dt
\Big\}.\label{err_u_phi}
\end{align}
From the assumption (\ref{assum_E}) and Theorem \ref{Stable}, and applying the discrete Gronwall lemma, we have
\begin{align}
 \|\tilde{e}_{\mathbf{u}}^{k+1}\|_f^2
+ & \sum_{n=0}^{k} \nu \Delta t \|\nabla \tilde{e}_{\mathbf{u}}^{n+1} \|_f^2
+\sum_{n=0}^{m}2\alpha\sqrt{\frac{\nu g}{tr(\mathbb{K})}}\Delta t\int_\Gamma (\tilde{e}_{\mathbf{u}}^{n+1}\cdot \boldsymbol{\tau})^2ds \notag\\
&+  gS_0\|\tilde{e}_\phi^{k+1}\|_p^2
+\sum_{n=0}^{k} gK \Delta t \|\nabla \tilde{e}_\phi^{n+1}\|_p^2
\leq C\Delta t^{2(1-\epsilon)}
.\label{roughuphi}
\end{align}
Hence,  when $\Delta t$ is sufficiently small such that $\Delta t^{3/4-\epsilon}\le \hat{C}/C$, we obtain
\begin{align*}
\|\tilde{e}_{\mathbf{u}}^{k+1}\|_f
+ \|\tilde{e}_{\phi}^{k+1}\|_p \leq C\Delta t^{1-\epsilon}
\leq \hat{C}\Delta t^{\frac{1}{4}},
\end{align*}
which completes the proof.
\end{proof}

\subsection{A refined error estimate}
Based on the rough error estimate established in Lemma \ref{prior}, we shall use an induction process to derive a refined error analysis.
\begin{theorem}\label{errthm}
Given $1\le k\le T/{\Delta t}-1$, the following  error estimate holds when $\Delta t$ is sufficiently small
\begin{align}\label{thm1}
|e_E^{k+1}|+\|\tilde{e}_{\mathbf{u}}^{k+1}\|_f
+\|\tilde{e}_\phi^{k+1}\|_p \leq C \Delta t,
\end{align}

{\color{black}
\begin{align}\label{e_final uphi}
\|e_{\mathbf{u}}^{k+1}\|_f + \| e_\phi^{k+1}\|_p \leq C \Delta t.
\end{align}
}
\end{theorem}

\begin{proof}
{\color{black} We can easily obtain that  \eqref{assum_u_phi} and \eqref{assum_E} hold with $m=0$. Next we assume that \eqref{assum_u_phi} and \eqref{assum_E} hold for given $m=k, \ \forall 1\le k\le T/{\Delta t}-1$. Below, we shall prove that \eqref{assum_u_phi} and \eqref{assum_E} hold for $m=k+1$ with the same constants $\hat{C}$ and $C^*$.}

For any $0\le n\le k$, we derive from Lemma \ref{prior} that
\begin{align}
0< \frac{1}{2}E(t^{n+1})\leq \tilde{E}^{n+1}\leq \frac{3}{2}E(t^{n+1}).\label{boundEtilde}
\end{align}
We also derive from \eqref{energy_id2} and \eqref{xi_n+1} the  following error equation for $E$:
\begin{align*}
\frac{e_E^{n+1}-e_E^n}{\Delta t}
=\frac{dE(t^{n+1})}{dt}
-\frac{E(t^{n+1})-E(t^n)}{\Delta t}
-\xi^{n+1} {\mathbb{I}}(\tilde{\mathbf{u}}^{n+1},\tilde{\phi}^{n+1})
+{\mathbb{I}}(\mathbf{u}(t^{n+1}),\phi(t^{n+1})),
\end{align*}
where ${\mathbb{I}}(\tilde{\mathbf{u}}^{n+1},\tilde{\phi}^{n+1})
=\nu \|\nabla \tilde{\mathbf{u}}^{n+1}\|_f^2
+\alpha \sqrt{\frac{\nu g}{tr(\mathbb{K})}}\int_\Gamma (\tilde{\mathbf{u}}^{n+1}\cdot\boldsymbol{\tau})^2 ds
+gK\|\nabla\tilde{\phi}^{n+1}\|_p^2$.\\
Multiplying both sides of the above equation by $e_E^{n+1}$, we obtain
\begin{align*}
&\frac{1}{2\Delta t}\left( |e_E^{n+1}|^2- |e_E^{n}|^2
+|e_E^{n+1}-e_E^{n}|^2 \right)\\
&=( -\xi^{n+1} {\mathbb{I}}(\tilde{\mathbf{u}}^{n+1},\tilde{\phi}^{n+1})
+{\mathbb{I}}(\mathbf{u}(t^{n+1}),\phi(t^{n+1})) )e_E^{n+1}
+\left( \frac{dE(t^{n+1})}{dt}
-\frac{E(t^{n+1})-E(t^n)}{\Delta t} \right) e_E^{n+1}
\\
&\leq
C|e_E^{n+1}|^2
+C\Delta t\int_{t^n}^{t^{n+1}} |E_{tt}|^2 dt
+I_6,
\end{align*}
where
{\color{black}
\begin{align*}
I_6&=\left|-\xi^{n+1} {\mathbb{I}}(\tilde{\mathbf{u}}^{n+1},\tilde{\phi}^{n+1})
+{\mathbb{I}}(\mathbf{u}(t^{n+1}),\phi(t^{n+1})) \right|
|e_E^{n+1}|\\
&\leq
 \left|
\frac{E^{n+1}}{\tilde{E}^{n+1}}
\left({\mathbb{I}}(\tilde{\mathbf{u}}^{n+1},\tilde{\phi}^{n+1})
-{\mathbb{I}}(\mathbf{u}(t^{n+1}),\phi(t^{n+1}))\right)
\right| \left|e_E^{n+1} \right|\\
&\quad
+\left|{\mathbb{I}}(\mathbf{u}(t^{n+1}),\phi(t^{n+1}))
 \left(1-\frac{E^{n+1}}{\tilde{E}^{n+1}}\right)  \right| \left|e_E^{n+1} \right|:=I_{6,1}+I_{6,2}.
\end{align*}}
As indicated by \eqref{lemmaE}, $E^{n+1}$ is bounded with $E^{n+1}\leq C$, and it follows from \eqref{boundEtilde} that
\begin{align*}
I_{6,1}&
= \left|
{\color{black}\frac{E^{n+1}}{\tilde{E}^{n+1}}}
\left({\mathbb{I}}(\tilde{\mathbf{u}}^{n+1},\tilde{\phi}^{n+1})
-{\mathbb{I}}(\mathbf{u}(t^{n+1}),\phi(t^{n+1}))\right)
\right| \left|e_E^{n+1} \right|\\
&\leq
C \left|
\left({\mathbb{I}}(\tilde{\mathbf{u}}^{n+1},\tilde{\phi}^{n+1})
-{\mathbb{I}}(\mathbf{u}(t^{n+1}),\phi(t^{n+1}))\right)
\right| \left|e_E^{n+1} \right|
\\
&\leq
C \left|\|\nabla \tilde{\mathbf{u}}^{n+1}\|_f + \|\nabla \mathbf{u}(t^{n+1})\|_f \right|
\|\nabla \tilde{e}_{\mathbf{u}}^{n+1}\|_f \left|e_E^{n+1} \right|
\\
&\quad
+C \left|\sqrt{\int_\Gamma (\tilde{\mathbf{u}}^{n+1}\cdot\boldsymbol{\tau})^2 ds}
+\sqrt{\int_\Gamma (\mathbf{u}(t^{n+1})\cdot\boldsymbol{\tau})^2 ds}\right|
\left|\sqrt{\int_\Gamma (\tilde{e}_{\mathbf{u}}^{n+1}\cdot\boldsymbol{\tau})^2 ds}\right| \left|e_E^{n+1} \right|\\
&\quad
+C \left| \|\nabla\tilde{\phi}^{n+1}\|_p + \|\nabla\phi(t^{n+1})\|_p \right| \left| \|\nabla\tilde{e}_{\phi}^{n+1}\|_p \right|  \left|e_E^{n+1} \right|
.
\end{align*}
From the rough error estimate (\ref{roughuphi}) and the regularity properties (\ref{regu_pro}), we know that $\|\nabla \tilde{\mathbf{u}}^{n+1}\|_f^2$, $\int_\Gamma (\tilde{\mathbf{u}}^{n+1}\cdot\boldsymbol{\tau})^2 ds$ and $\|\nabla\tilde{\phi}^{n+1}\|_p^2$ are both bounded, so
by applying the Cauchy-Schwarz and Young's inequalities, we have
\begin{align*}
I_{6,1}
\leq
C\left |e_E^{n+1} \right|^2
 +\epsilon_{\mathbf{u}} \nu
\|\nabla \tilde{e}_{\mathbf{u}}^{n+1} \|_f^2
+
\epsilon_{\Gamma}\alpha \sqrt{\frac{\nu g}{tr(\mathbb{K})}} {\int_\Gamma (\tilde{e}_{\mathbf{u}}^{n+1}\cdot\boldsymbol{\tau})^2 ds}
+ \epsilon_\phi gK\|\nabla\tilde{e}_{\phi}^{n+1}\|_p ^2.
\end{align*}
On the other hand,
\begin{align*}
I_{6,2}&
=\left|{\mathbb{I}}(\mathbf{u}(t^{n+1}),\phi(t^{n+1}))
 \left(1-{\color{black}\frac{E^{n+1}}{\tilde{E}^{n+1}}  }\right)  \right| \left|e_E^{n+1} \right|\\
 &\leq
 C \left|e_E^{n+1} \right|^2
+C \left| 1-{\color{black}\frac{E^{n+1}}{\tilde{E}^{n+1}}  } \right|^2\\
&\leq
 C \left|e_E^{n+1} \right|^2
+ C\|\tilde{e}_{\mathbf{u}}^{n+1}\|_f^2
+C\|\tilde{e}_{\phi}^{n+1}\|_p^2.
\end{align*}
Combining the above inequalities, we obtain
\begin{align}
&\frac{1}{2\Delta t}\left( |e_E^{n+1}|^2- |e_E^{n}|^2
+|e_E^{n+1}-e_E^{n}|^2 \right)\notag\\
&
\leq
C |e_E^{n+1}|^2
+\epsilon_{\mathbf{u}}\nu
\|\nabla \tilde{e}_{\mathbf{u}}^{n+1} \|_f^2
+\epsilon_\phi gK\|\nabla\tilde{e}_\phi^{n+1}\|_p^2
+\epsilon_\Gamma\alpha \sqrt{\frac{\nu g}{tr(\mathbb{K})}} \int_\Gamma (\tilde{e}_{\mathbf{u}}^{n+1}\cdot\boldsymbol{\tau})^2 ds\notag\\
&\quad
+C\|\tilde{e}_{\mathbf{u}}^{n+1}\|_f^2
+C\|\tilde{e}_{\phi}^{n+1}\|_p^2
+C\Delta t \int_{t^n}^{t^{n+1}} |E_{tt}|^2dt
.\label{err_E}
\end{align}
Adding (\ref{err_ns}), (\ref{errdarcy}) and (\ref{err_E}) together, setting $\epsilon_{\mathbf{u}}=\frac{1}{10}$, $\epsilon_\phi=\frac{1}{8}$, $\epsilon_\Gamma=\frac{1}{2}$, and summing them with $n$ ranging from $0$ to $k$, we arrive at
\begin{align*}
&\|\tilde{e}_{\mathbf{u}}^{k+1}\|_f^2
+\sum_{n=0}^{k} \nu\Delta t \|\nabla \tilde{e}_{\mathbf{u}}^{n+1} \|_f^2
+\sum_{n=0}^{k}\alpha\sqrt{\frac{\nu g}{tr(\mathbb{K})}}\Delta t
\int_\Gamma(\tilde{e}_{\mathbf{u}}^{n+1}\cdot\boldsymbol{\tau})^2 ds
+{gS_0}\|\tilde{e}_\phi^{k+1}\|_p^2\\
&
+\sum_{n=0}^{k}{gK}\Delta t \|\nabla\tilde{e}_\phi^{n+1}\|_p^2
+|e_E^{k+1}|^2\\
&
\leq
\sum_{n=0}^{k} C\Delta t \|\tilde{e}_{\mathbf{u}}^{n}\|_f^2
+\sum_{n=0}^{k} C\Delta t \|\tilde{e}_{\phi}^{n}\|_p^2
+\sum_{n=0}^{k} C\Delta t |{e}_{E}^{n}|^2
+\sum_{n=0}^{k} C \Delta t \|\nabla e_{\mathbf{u}}^n \|_f^2
\|\tilde{e}_{\mathbf{u}}^n\|_f^2\notag\\
&\quad
+\sum_{n=0}^{k} C \Delta t \|\nabla e_{\mathbf{u}}^n \|_f^2
\|\tilde{e}_{\phi}^n\|_f^2
+\sum_{n=0}^{k} C \Delta t \|\nabla e_{\mathbf{u}}^n \|_f^2
|e_E^n|^2
+\sum_{n=0}^{k} C \Delta t|e_E^{n+1}|^2\notag\\
&\quad
+\sum_{n=0}^{k} C \Delta t \|\tilde{e}_{\mathbf{u}}^{n+1}\|_f^2
+\sum_{n=0}^{k} C \Delta t \|\tilde{e}_{\phi}^{n+1}\|_p^2
+\sum_{n=0}^{k} \Delta t \Big\{
C\Delta t \int_{t^n}^{t^{n+1}} \|\partial_{tt}\mathbf{u}\|_{-1,f}^2dt\notag\\
&\quad
+C\Delta t \int_{t^n}^{t^{n+1}} \|\partial_t\mathbf{u}\|_f^2 dt
+C\Delta t \int_{t^n}^{t^{n+1}} \|\partial_t \mathbf{u}\|_{2,f}^2 dt
+C\Delta t \int_{t^n}^{t^{n+1}} \|\partial_t \phi\|_{1,p}^2 dt\notag\\
&\quad
+C\Delta t \int_{t^n}^{t^{n+1}} \|\partial_{tt}\phi\|_{-1,p}^2 dt
+C\Delta t \int_{t^n}^{t^{n+1}} \|\partial_t \mathbf{u}\|_{1,f}^2 dt
+C\Delta t \int_{t^n}^{t^{n+1}} |E_{tt}|^2dt
\Big\}
.
\end{align*}
Applying the discrete Gronwall lemma and using the regularity assumption to the above inequality, we obtain the desired result (\ref{thm1}). Finally, recalling \eqref{err_uu} and \eqref{e_phin}, we can obtain the error estimate \eqref{e_final uphi}.

It remains to verify that the assumptions \eqref{assum_u_phi} and \eqref{assum_E} hold with $m=k+1$. For  $\Delta t$ is sufficiently small such that {\color{black} $\Delta t^{3/4}\le \hat{C}/C$ and $\Delta t^{\epsilon}\le {C}^*/C$}, we derive from  (\ref{thm1}) that
\begin{align*}
&\|\tilde{e}_{\mathbf{u}}^{k+1}\|_f
+{gS_0}\|\tilde{e}_\phi^{k+1}\|_p \leq C\Delta t \leq \hat{C}\Delta t^{1/4},\\
&|e_E^{k+1}| \leq C\Delta t \leq C^* \Delta t^{1-\epsilon},\quad 0<\epsilon<1/2.\\
\end{align*}
Hence the induction process is complete.
\end{proof}


\section{Extension to the Cahn--Hilliard--Navier--Stokes--Darcy model}
In this section, we extend the above linear and original-energy-dissipative scheme to the more complex Cahn--Hilliard--Navier--Stokes--Darcy model, thereby providing a more comprehensive framework suitable for applications such as contaminant transport and enhanced oil recovery.

In the free-fluid domain $\Omega_f$, the fluid is governed by the Cahn--Hilliard--Navier--Stokes equations.
\begin{align}
&\partial_t \mathbf{u}_f +(\mathbf{u}_f\cdot\nabla)\mathbf{u}_f
- {\color{black} \nu_f \Delta \mathbf{u}_f} +\nabla p_f +\phi_f\nabla\mu_f=0,\label{CHNS1}\\
&\nabla\cdot\mathbf{u}_f=0,\label{CHNS2}\\
&\partial_t \phi_f + \nabla\cdot(\mathbf{u}_f \phi_f) -\nabla\cdot(M(\phi_f)\nabla \mu_f)=0,\label{CHNS3}\\
&\mu_f+\lambda \epsilon \Delta \phi_f -\frac{\lambda}{\epsilon} G'(\phi_f)=0,\label{CHNS4}\\
&\mathbf{u}_f|_{\Gamma_f}=0,\qquad \mathbf{u}_f(\mathbf{x},0)=\mathbf{u}_f^0,\label{CHNS5}\\
&\nabla\phi_f\cdot \boldsymbol{n}|_{\Gamma_f}=0,\qquad \phi_f(\mathbf{x},0)=\phi_f^0,\label{CHNS6}\\
&M(\phi_f) \nabla \mu_f \cdot \boldsymbol{n}|_{\Gamma_f}=0,\qquad \mu_f(\mathbf{x},0)=\mu_f^0,\label{CHNS7}
\end{align}
where $\mathbf{u}_f$ and $p_f$ are the velocity and pressure of the free flow, respectively. The phase field function is $\phi_i$ $(i=f,p)$, and the corresponding chemical potential is $\mu_i$ $(i=f,p)$. The free energy $G(\phi)$ is defined by a double-well polynomial, $G(\phi)=\frac{1}{4}(\phi^2-1)^2$. Here, $\nu_i$ $(i=f,p)$ is the kinematic viscosity, $\epsilon$ denotes the interfacial width, $M(\phi_i)$ $(i=f,p)$ is the mobility related to $\phi_i$, and $\lambda$ is the elastic relaxation time.

In the porous-media domain $\Omega_p$, the fluid satisfies the Cahn--Hilliard--Darcy equations.
\begin{align}
&{\chi}^{-1}\partial_t \mathbf{u}_p
+\mathbb{K}^{-1} \mathbf{u}_p+\nabla p_p +\phi_p \nabla \mu_p=0,\label{CHD1}\\
&\nabla\cdot \mathbf{u}_p=0,\label{CHD2}\\
&\partial_t \phi_p + \nabla\cdot(\mathbf{u}_p\phi_p) -\nabla\cdot(M(\phi_p)\nabla \mu_p)=0,\label{CHD3}\\
&\mu_p+\lambda \epsilon \Delta \phi_p -\frac{\lambda}{\epsilon} G'(\phi_p)=0,\label{CHD4}\\
&\mathbf{u}_p\cdot \boldsymbol{n}|_{\Gamma_p}=0,\qquad \mathbf{u}_p(\mathbf{x},0)=\mathbf{u}_p^0,\label{CHD5}\\
&\nabla\phi_p\cdot \boldsymbol{n}|_{\Gamma_p}=0,\qquad \phi_p(\mathbf{x},0)=\phi_p^0,\label{CHD6}\\
&M(\phi_p) \nabla \mu_p \cdot \boldsymbol{n}|_{\Gamma_p}=0,\qquad \mu_p(\mathbf{x},0)=\mu_p^0,\label{CHD7}
\end{align}
where $\mathbf{u}_p$ is the velocity, $p_p$ is the hydraulic head, $\mathbb{K}$ is the hydraulic conductivity tensor. $\chi$ represents the porosity of porous media.

The interface conditions are defined as follows.
\begin{align}
&\phi_f=\phi_p,\qquad \nabla\phi_f\cdot\boldsymbol{n}=\nabla\phi_p\cdot \boldsymbol{n},\label{CHNSDgamma1}\\
&\mu_f=\mu_p,\qquad M(\phi_f) \nabla w_f\cdot\boldsymbol{n}=M(\phi_p) \nabla w_p\cdot\boldsymbol{n},\label{CHNSDgamma2}\\
&\mathbf{u}_f\cdot\mathbf{n} =\mathbf{u}_p\cdot\mathbf{n},\label{CHNSDgamma5}\\
&p_f-\nu_f \mathbf{n}\cdot(\nabla \mathbf{u}_f \cdot \mathbf{n})+\frac{1}{2}|\mathbf{u}_f|^2=p_p,\label{CHNSDgamma6}\\
&-\nu_f\boldsymbol{\tau}\cdot(\nabla \mathbf{u}_f \cdot \mathbf{n})=\frac{\alpha \nu_f \sqrt{d}}{\sqrt{\nu_p tr(\mathbb{K})}}(\mathbf{u}_f\cdot\boldsymbol{\tau}).\label{CHNSDgamma7}
\end{align}

We also define the following spaces:
\begin{align*}
X_f&=\{\mathbf{v}\in (H^1({\Omega_f}))^d:\mathbf{v}|_{\Gamma_f}=0\},\qquad
X_p=\{\mathbf{v}\in (H^1({\Omega_p}))^d:\mathbf{v}\cdot \boldsymbol{n}|_{\Gamma_p}=0\},\\
M_f&=L^2(\Omega_f),\qquad M_p=\{q\in H^1(\Omega_p): \int_{\Omega_p} q d\mathbf{x}=0\},\qquad Y=H^1(\Omega).
\end{align*}
Then, a weak formulation for the above system is: Find $(\mathbf{u}_f,p_f,\mathbf{u}_p,p_p,\phi,w)\in (X_f\times M_f\times X_p\times M_p\times Y\times Y)$ such that for all $(\mathbf{v}_f,q_f,\mathbf{v}_p,q_p,\psi,\omega)\in (X_f\times M_f\times X_p\times M_p\times Y\times Y)$,
\begin{align}
&(\partial_t\phi,\psi)-(\mathbf{u} \phi,\nabla\psi)
+(M(\phi)\nabla \mu,\nabla \psi) =0,\label{CHNSD1}\\
&(\mu,\omega) -\lambda \epsilon (\nabla \phi,\nabla \omega)
-\frac{\lambda}{\epsilon} (G'(\phi),\omega)=0,\\
&\chi^{-1}(\partial _t \mathbf{u}_p,\mathbf{v}_p)_p +\mathbb{K}^{-1}(\mathbf{u}_p,\mathbf{v}_p)_p
+(\nabla p_p,\mathbf{v}_p)_p
+(\phi_p \nabla\mu_p,\mathbf{v}_p)_p=0,\\
&-(\mathbf{u}_p,\nabla q_p)_p-\int_\Gamma (\mathbf{u}_f\cdot \mathbf{n}) q_p ds=0,\\
&(\frac{\partial \mathbf{u}_f}{\partial t},\mathbf{v}_f)_f
+a(\mathbf{u}_f,\mathbf{u}_f,\mathbf{v}_f)
+\nu_f(\nabla \mathbf{u}_f,\nabla\mathbf{v}_f)_f
-(p_f,\nabla\cdot \mathbf{v}_f)_f
+\int_\Gamma p_p (\mathbf{v}_f\cdot \mathbf{n}) ds \nonumber\\
&\qquad
+\frac{\alpha \nu_f \sqrt{d}}{\sqrt{\nu_p tr(\mathbb{K})}} \int_\Gamma (\mathbf{u}_f\cdot\boldsymbol{\tau})(\mathbf{v}_f\cdot\boldsymbol{\tau}) ds
+(\phi_f \nabla\mu_f,\mathbf{v}_f)_f=0,\\
&(\nabla\cdot \mathbf{u}_f,q_f)_f=0.\label{CHNSD6}
\end{align}
It is easy to verify that the above Cahn--Hilliard--Navier--Stokes--Darcy model satisfies the following energy dissipation law
\begin{align*}
\frac{dE}{dt}=- \nu_f \|\nabla \mathbf{u}_f\|_f^2
+\frac{\alpha \nu_f \sqrt{d}}{\sqrt{\nu_p tr(\mathbb{K})}} \int_\Gamma (\mathbf{u}_f\cdot \boldsymbol{\tau})^2 ds
+\|\sqrt{\mathbb{K}^{-1}}\mathbf{u}_p\|_p^2
+M(\phi)\|\nabla \mu\|^2,
\end{align*}
where the original energy is defined as
\begin{align*}
E=\left(\frac{\lambda\epsilon}{2}\|\nabla \phi\|^2
+\frac{\lambda}{\epsilon} \int_{\Omega} G(\phi) d\mathbf{x}\right)
+\left(\frac{1}{2}\|\mathbf{u}_f\|_f^2
+\frac{1}{2\chi}\|\mathbf{u}_p\|_p^2\right):=E_0+E_1.
\end{align*}
In order to construct efficient and unconditionally energy stable schemes for the Cahn--Hilliard--Navier--Stokes--Darcy model, we introduce a scalar auxiliary variable $\xi(t)\equiv 1$, which serves as a relaxation factor, and expand the model with a reformulated energy dissipation law:
\begin{align}\label{energyid3}
	\frac{dE}{dt}= \xi(t)\left(- \nu_f \|\nabla \mathbf{u}_f\|_f^2
	+\frac{\alpha \nu_f \sqrt{d}}{\sqrt{\nu_p tr(\mathbb{K})}} \int_\Gamma (\mathbf{u}_f\cdot \boldsymbol{\tau})^2 ds
	+\|\sqrt{\mathbb{K}^{-1}}\mathbf{u}_p\|_p^2
	+M(\phi)\|\nabla \mu\|^2\right).
\end{align}
Then, following essentially the same principle as {\bf Scheme I}, we construct below a first-order scheme for the expanded system of  \eqref{CHNSD1}-\eqref{CHNSD6} with \eqref{energyid3}.

{\textbf {Scheme \uppercase\expandafter{\romannumeral3}}}: a first-order scheme for the Cahn--Hilliard--Navier--Stokes--Darcy model.

{\it Step 1 (Prediction).} Find $(\phi^{n+1},\mu^{n+1})\in (Y\times Y)$ such that
\begin{align}
&\left(\frac{\phi^{n+1}-\phi^n}{\Delta t},\psi\right)
-(\mathbf{u}^n \phi^{n},\nabla \psi)
+(M(\phi^n)\nabla \mu^{n+1},\nabla \psi) =0,\label{CHNSD_S1}\\
&(\mu^{n+1},\omega) -\lambda \epsilon (\nabla \phi^{n+1},\nabla \omega)
-\frac{\lambda}{\epsilon} (G'(\phi^n),\omega)=0.
\end{align}

Find $(\tilde{\mathbf{u}}_p^{n+1},p_p^{n+1})\in (X_p\times M_p)$ such that
\begin{align}
&\chi^{-1}\left(\frac{\tilde{\mathbf{u}}_p^{n+1}-\tilde{\mathbf{u}}_p^n}{\Delta t},\mathbf{v}_p\right)_p
+{\mathbb{K}}^{-1}(\tilde{\mathbf{u}}_p^{n+1},\mathbf{v}_p)_p
+(\nabla p_p^{n+1},\mathbf{v}_p)_p
+(\phi_p^{n+1} \nabla \mu_p^{n+1},\mathbf{v}_p)_p=0,\\
&-(\tilde{\mathbf{u}}_p^{n+1},\nabla q_p)_p-\int_\Gamma (\mathbf{u}_f^n\cdot \mathbf{n}) q_p ds=0.
\end{align}

Find $(\tilde{\mathbf{u}}_f^{n+1},p_f^{n+1})\in (X_f\times M_f)$ such that
\begin{align}
&\left(\frac{\tilde{\mathbf{u}}_f^{n+1}-\tilde{\mathbf{u}}_f^{n}}{\Delta t},\mathbf{v}_f\right)_f
+ {\color{black} a(\mathbf{u}_f^{n},{\mathbf{u}}_f^{n},\mathbf{v}_f) }
+\nu_f(\nabla \tilde{\mathbf{u}}_f^{n+1},\nabla\mathbf{v}_f)_f
- {\color{black} (p_f^{n+1},\nabla\cdot \mathbf{v}_f)_f }
+\int_\Gamma p_p^{n+1} (\mathbf{v}_f\cdot \mathbf{n}) ds \nonumber\\
&\qquad
+\frac{\alpha \nu_f \sqrt{d}}{\sqrt{\nu_p tr(\mathbb{K})}} \int_\Gamma (\tilde{\mathbf{u}}_f^{n+1}\cdot\boldsymbol{\tau})(\mathbf{v}_f\cdot\boldsymbol{\tau}) ds
+(\phi_f^{n+1}\nabla \mu_f^{n+1},\mathbf{v}_f)_f
=0,\\
&(\nabla\cdot \tilde{\mathbf{u}}_f^{n+1},q_f)_f=0.\label{CHNSD_S8}
\end{align}

{\it Step 2 (Correction).} We denote
\begin{align}\label{CHNSDmodiE}
	\tilde{E}^{n+1}= \left(\frac{\lambda\epsilon}{2}\|\nabla{\phi}^{n+1}\|^2
	+\frac{\lambda}{\epsilon}\int_{\Omega} G({\phi^{n+1}})d\mathbf{x}\right)
	+\left(\frac{1}{2}\|\tilde{\mathbf{u}}_f^{n+1}\|_f^2
	+\frac{1}{2\chi}\|\tilde{\mathbf{u}}_p^{n+1}\|_p^2)\right)
	:=E_0^{n+1}+\tilde{E}_1^{n+1},
\end{align}
and set
\begin{align}\label{CHNSDoriE}
	E^{n+1}={E}_0^{n+1}+\xi^{n+1} \tilde{E}_1^{n+1}.
\end{align}
Then, we determine  $\xi^{n+1}$ from the linear algebraic equation
\begin{align}
	\frac{E^{n+1}-E^n}{\Delta t}
&=-\xi^{n+1} \Big( \nu_f \|\nabla \tilde{\mathbf{u}}_f^{n+1}\|_f^2
+\frac{\alpha \nu_f \sqrt{d}}{\sqrt{\nu_p tr(\mathbb{K})}} \int_\Gamma (\tilde{\mathbf{u}}_f^{n+1}\cdot \boldsymbol{\tau})^2ds\nonumber\\
&\quad
+\|\sqrt{\mathbb{K}^{-1}}\tilde{\mathbf{u}}_p^{n+1}\|_p^2
+M({\phi}^{n+1})\|\nabla \mu^{n+1}\|^2
\Big).\label{Stabest3}
\end{align}
Finally, we set
\begin{align}
  \mathbf{u}_f^{n+1} =\sqrt{\xi^{n+1}} \tilde{\mathbf{u}}_f^{n+1},\qquad \mathbf{u}_p^{n+1} =\sqrt{\xi^{n+1}} \tilde{\mathbf{u}}_p^{n+1}.\label{CHNSDupdate_u}
\end{align}

Similarly, a linear and second-order version of {\textbf{Scheme \uppercase\expandafter{\romannumeral3}}} can be constructed to preserve the dissipation of the original energy. Following a similar approach to the proof of Theorem \ref{Stable}, we can establish the following stability result for {\textbf{Scheme \uppercase\expandafter{\romannumeral3}}}:


\begin{theorem}\label{Stable3}
The scheme  (\ref{CHNSD_S1})-(\ref{CHNSDupdate_u})  dissipates the original energy as shown by \eqref{Stabest3}.
Moreover, we have
	\begin{align}\label{e_stability_CHNSD}
 \|\nabla{\phi}^{n+1}\|^2
	+\|\mathbf{u}_f^{n+1}\|_f^2+\|\mathbf{u}_p^{n+1}\|_p^2
    + \nu_f \sum_{n=0}^{k} \Delta t  \|\nabla \mathbf{u}_f^{n+1}\|_f^2
    \leq  C,
\end{align}
where $C>0$ is a constant that depends on the initial conditions $\phi^{0}, \mathbf{u}_f^{0} $ and $\mathbf{u}_p^{0}$.
\end{theorem}


\section{Numerical experiments}
In order to validate the essential properties of the numerical schemes proposed in this paper, we conduct six numerical tests that focus on different aspects of the schemes.
We choose Taylor-Hood elements (\textbf{P}2-P1) for the velocity and pressure of free flow and the flow in the porous-media domain. For the two-phase flow, we choose P1 finite element for both phase function and chemical potential. All numerical tests are carried out using the open-source finite element software package FreeFem++ \cite{hecht2012new}.

\subsection{The EMAC formulation}
It is noted in \cite{olshanskii2020longer} that if the discrete velocity approximation is not  exactly divergence free,  errors for the conserved quantities, such as mass, momentum, and angular momentum, may accumulate rapidly overtime. The EMAC formulation, originally proposed for the Navier--Stokes equations in \cite{charnyi2017conservation}, can preserve energy, momentum, and angular momentum in a discrete setting. Since the Taylor-Hood elements do not lead to exactly divergence free velocity approximation, we shall adapt the EMAC formulation which is presented below.

As proposed in \cite{olshanskii2020longer}, we introduce $P=p-\frac{1}{2}|\mathbf{u}|^2$ and present an equivalent form for the convective term and pressure in the Navier--Stokes equations.
\begin{align*}
	\mathbf{u}\cdot \nabla\mathbf{u} +\nabla p =2\mathbb{D}(\mathbf{u})\mathbf{u} +(\nabla\cdot \mathbf{u})\mathbf{u} +\nabla P,
\end{align*}
where the strain tensor $\mathbb{D}(\mathbf{u})=\frac{1}{2}(\nabla\mathbf{u}+(\nabla \mathbf{u})^T)$. Similarly, we derive the weak form with the EMAC formulation for the problems defined by equations (\ref{stokes1})-(\ref{gamma3}).
For $\forall \ t\in (0,T]$, find $\mathbf{u}\in H_f,~P\in Q_f,~\phi\in H_p$ such that
\begin{equation}\label{nsd problem_EMAC}
	\left\{
	\begin{aligned}
		&(\partial_t \mathbf{u},\mathbf{v})_f +gS_0(\partial_t \phi,\psi)_p + \nu(\nabla \mathbf{u},\nabla\mathbf{v} )_f-(P,\nabla\cdot\mathbf{v})_f
		+b(\mathbf{u},\mathbf{u},\mathbf{\mathbf{v}})+\alpha\sqrt{\frac{\nu g}{tr(\mathbb{K})}}
		\int_\Gamma(\mathbf{u}\cdot\boldsymbol{\tau})(\mathbf{v}\cdot\boldsymbol{\tau})ds
		\\
		&\qquad+g(\mathbb{K}\nabla\phi,\nabla\psi)_p
		+c_\Gamma(\mathbf{v},\phi)
		-c_\Gamma(\mathbf{u},\psi)
		=(\mathbf{f}_f,\mathbf{v})_f +g(f_p,\psi)_p,
		\qquad\forall\mathbf{v}\in H_f,~\psi\in H_p,\\
		&(\nabla\cdot \mathbf{u},q)_f=0,\qquad\qquad~\forall q\in Q_f,\\
		&(\mathbf{u}(\mathbf{x},0),\mathbf{v})_f=(\mathbf{u}^0,\mathbf{v})_f,\qquad\forall \mathbf{v}\in H_f,\\
		&(\phi(\mathbf{x},0),\psi)_p=(\phi^0,\psi)_p,\qquad\forall \psi\in H_p,
	\end{aligned}
	\right.
\end{equation}
where
\begin{align*}
	&b(\mathbf{u},\mathbf{v},\mathbf{w})
	=(2\mathbb{D}(\mathbf{u})\mathbf{v},\mathbf{w})_f
	+((\nabla\cdot\mathbf{u})\mathbf{v},\mathbf{w})_f
	-\int_\Gamma ( \mathbf{u}\cdot \mathbf{v} )  (\mathbf{w}\cdot \mathbf{n}) ds, \qquad\forall \mathbf{u},\mathbf{v},\mathbf{w}\in H_f,\\
	&c_\Gamma(\mathbf{v},\phi)=g\int_\Gamma\phi \mathbf{v}\cdot \mathbf{n} ds, \qquad\forall\mathbf{v}\in H_f,~\phi\in H_p.
\end{align*}

\medskip
\begin{lemma}\label{b=0}
	The nonlinear term $b(\mathbf{u},\mathbf{v},\mathbf{w})$ has the following property:
	\begin{align}\label{trinew}
		b(\mathbf{u},\mathbf{u},\mathbf{u})=0,\quad\forall \mathbf{u}\in H_f.
	\end{align}
\end{lemma}
\medskip
\begin{proof}
	Since the nonlinear term $b(\mathbf{u},\mathbf{u},\mathbf{u})
	=(2\mathbb{D}(\mathbf{u})\mathbf{u},\mathbf{u})_f
	+((\nabla\cdot\mathbf{u})\mathbf{u},\mathbf{u})_f
	-\int_\Gamma ( \mathbf{u}\cdot \mathbf{u} )  (\mathbf{u}\cdot \mathbf{n}) ds$,
	by using the Green's formula, we can obtain
	\begin{align*}
		b(\mathbf{u},\mathbf{u},\mathbf{u})
		&=(2\mathbb{D}(\mathbf{u})\mathbf{u},\mathbf{u})_f
		-(\mathbf{u},\nabla(\mathbf{u}\cdot\mathbf{u}))_f
		+\int_\Gamma ( \mathbf{u}\cdot \mathbf{u} )  (\mathbf{u}\cdot \mathbf{n}) ds
		-\int_\Gamma ( \mathbf{u}\cdot \mathbf{u} )  (\mathbf{u}\cdot \mathbf{n}) ds\\
		&=(2(\nabla\mathbf{u})\mathbf{u},\mathbf{u})_f
		-(2(\nabla\mathbf{u})\mathbf{u},\mathbf{u})_f
		+\int_\Gamma ( \mathbf{u}\cdot \mathbf{u} )  (\mathbf{u}\cdot \mathbf{n}) ds
		-\int_\Gamma ( \mathbf{u}\cdot \mathbf{u} )  (\mathbf{u}\cdot \mathbf{n}) ds\\
		&=0,\qquad\forall  \mathbf{u}\in H_f.
	\end{align*}
\end{proof}

\begin{rem}
	We can replace the trilinear form $a(\cdot,\cdot,\cdot)$ and $(p^{n+1},\nabla\cdot\mathbf{v})_f$ in {\bf Schemes I, II and III} by $b(\cdot,\cdot,\cdot)$ and $(P^{n+1},\nabla\cdot\mathbf{v})_f$ to obtain the corresponding EMAC schemes. It should be noted that the stability results established for {\bf Schemes I, II and III} are also valid for the corresponding EMAC schemes. However, the error analysis will require further effort.
\end{rem}


\subsection{Example 1: Convergence tests A}
We first set the parameters $\nu = 0.001$, $\mathbb{K}=\mathbb{I}$ and all other parameters to $1$. We choose the computational domain as $\Omega_f=[0,1]\times [0,1],~\Omega_p=[0,1]\times [-1,0]$. The exact solution, along with the corresponding initial values and source terms, is constructed as follows.
\begin{equation*}
\left\{
\begin{aligned}
& \mathbf{u} =[\mathbf{u}_1, \mathbf{u}_2] = [x^2y^2 \cos( t),-\frac{2}{3}x y^3 \cos( t)],\\
&p=c x^2 (x-1)^2 y^2 (y-1)^2 \cos( t),\\
&\phi=c x^2 (x-1)^2 (y+1)^2 y^2 \cos( t).
\end{aligned}\right.
\end{equation*}

We set $c = 64\pi^2$ and consider a series of spatial mesh sizes $h = 1/16, 1/20, 1/24, 1/28, 1/32$ with the time step $\Delta t = h^2$. The corresponding numerical results are presented in Figure \ref{test1}. As shown in Figure \ref{test1}, the convergence order achieves first order, which is in consistent with Theorem \ref{errthm}. Subsequently, we set $\Delta t = h/8$ to examine the performance of {\textbf {Scheme \uppercase\expandafter{\romannumeral2}}}. These results, also shown in Figure \ref{test1}, clearly demonstrate second-order temporal accuracy.

\begin{figure}[htbp]
  \centering
  \includegraphics[scale=0.28]{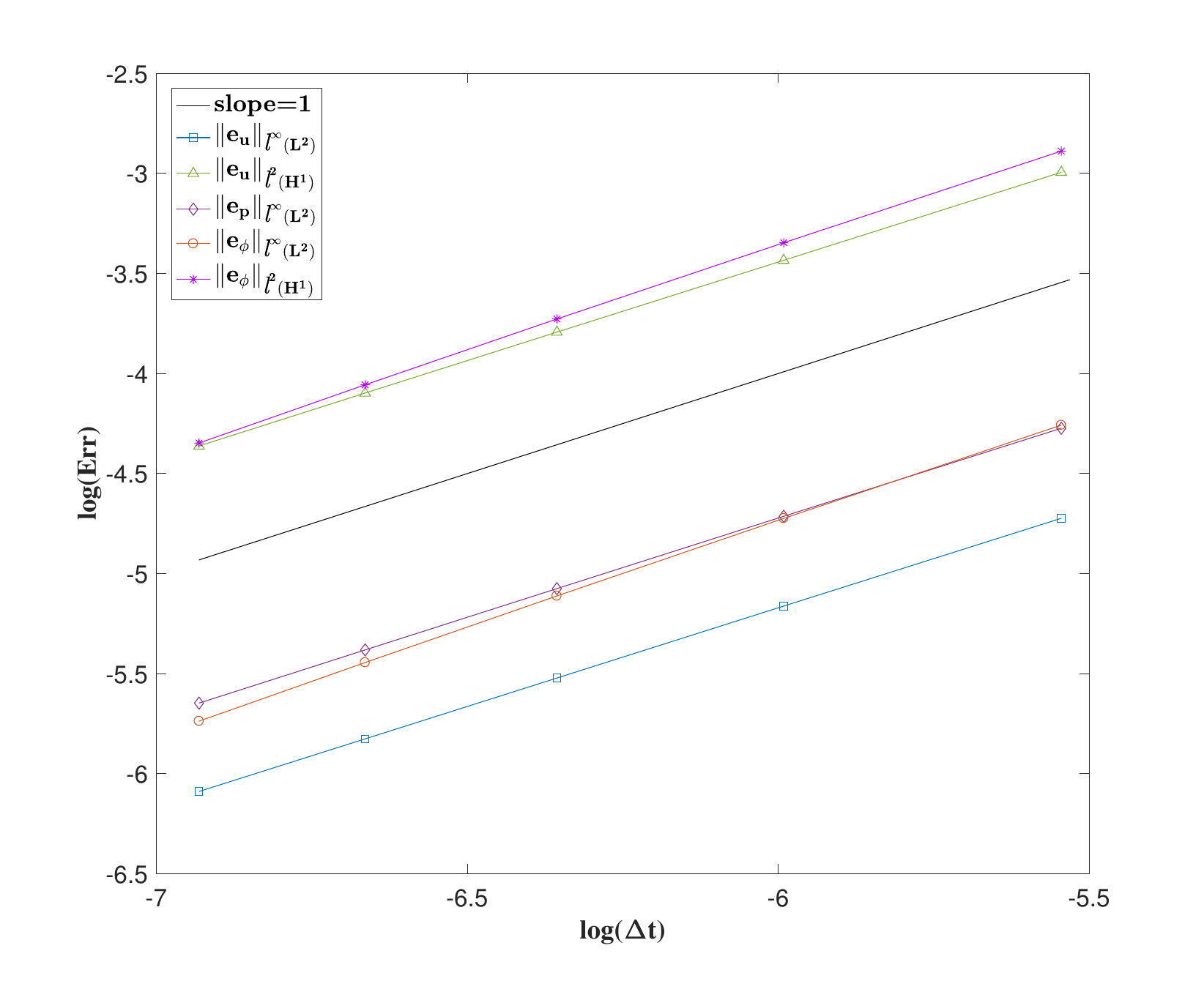}
  \includegraphics[scale=0.28]{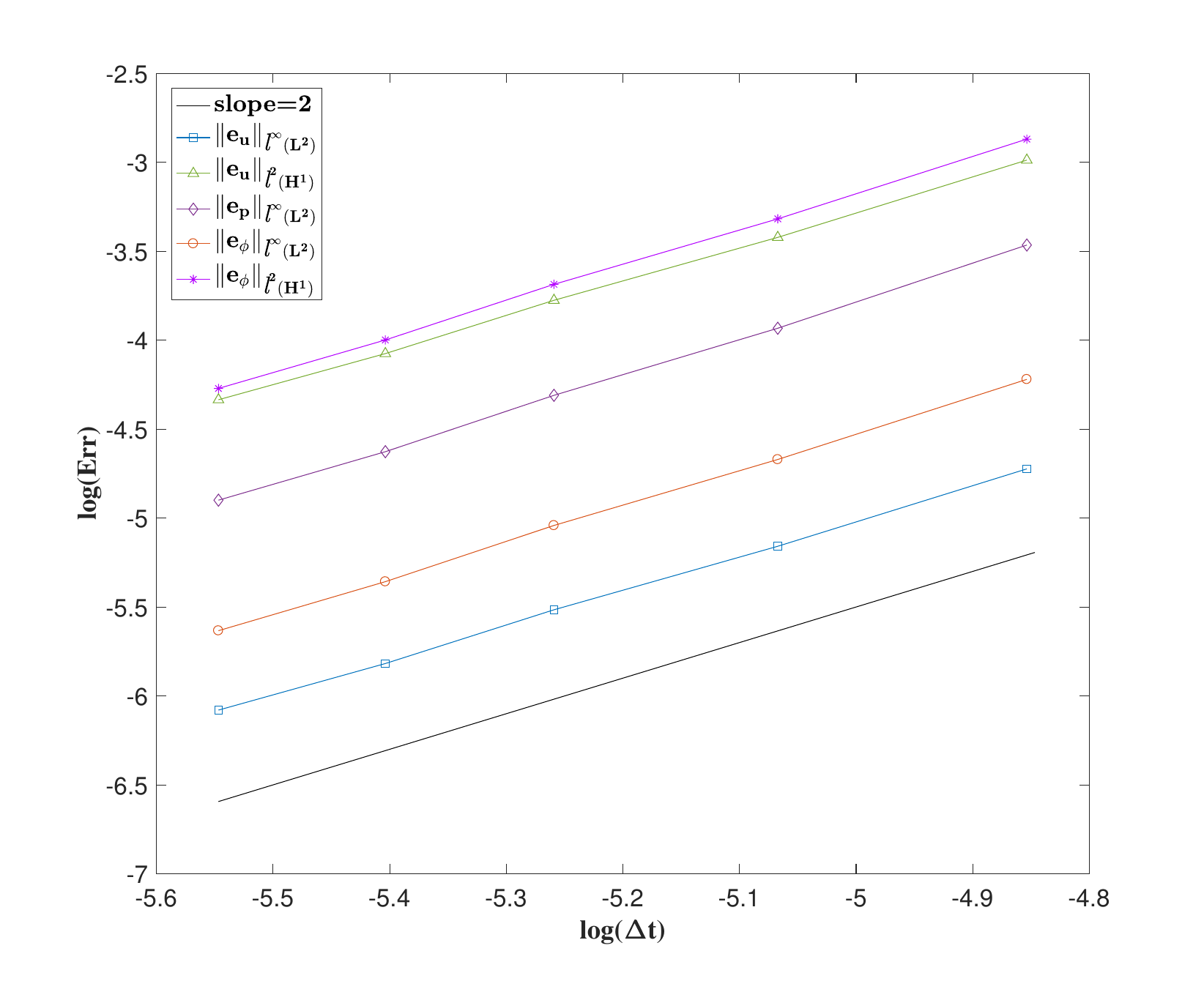}
  \caption{Comparison of numerical errors and convergence rates with $T=0.5$ (Example 1).
Left: {\textbf {Scheme \uppercase\expandafter{\romannumeral1}}}. Right: {\textbf {Scheme \uppercase\expandafter{\romannumeral2}}}.}\label{test1}
\end{figure}

\subsection{Example 2: Convergence tests B}
In this example, all parameters are kept identical to those in Example 1. The analytic solution is chosen as
\begin{equation*}
\left\{
\begin{aligned}
& \mathbf{u} =[\mathbf{u}_1, \mathbf{u}_2] =
\left[l\sin^2(\pi x) \sin(2\pi y) \cos(t), -l \sin(2\pi x) \sin^2(\pi y) \cos(t)\right],\\
& p=\sin(\pi y)\cos(\pi x)\cos(t),\\
& \phi=\sin(\pi x)\sin^2(\pi y)\cos(t).
\end{aligned}
\right.
\end{equation*}

 The numerical results using {\textbf {Scheme \uppercase\expandafter{\romannumeral1}}} with $l=1/20\pi^2$, $h = 1/16, 1/20, 1/24, 1/28, 1/32$, and $\Delta t = h^{2}$, are presented in Figure \ref{test2}. We can easily obtain the temporal convergence order of the first-order scheme. Then we set $\Delta t=h/4$ to explore the convergence rate of {\textbf {Scheme \uppercase\expandafter{\romannumeral2}}}, the desired results also provided in Figure \ref{test2}.

\begin{figure}[htbp]
  \centering
  \includegraphics[scale=0.28]{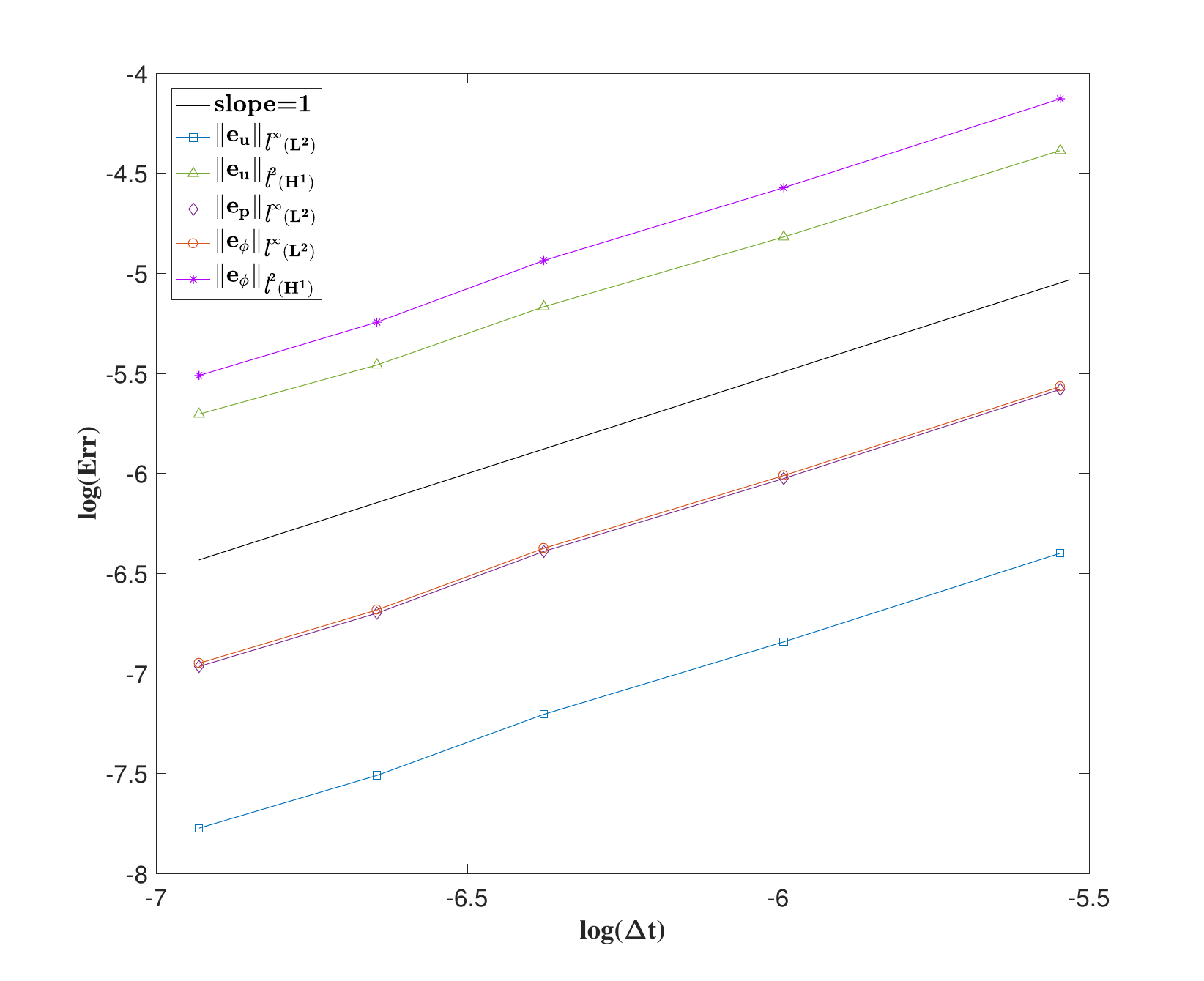}
  \includegraphics[scale=0.28]{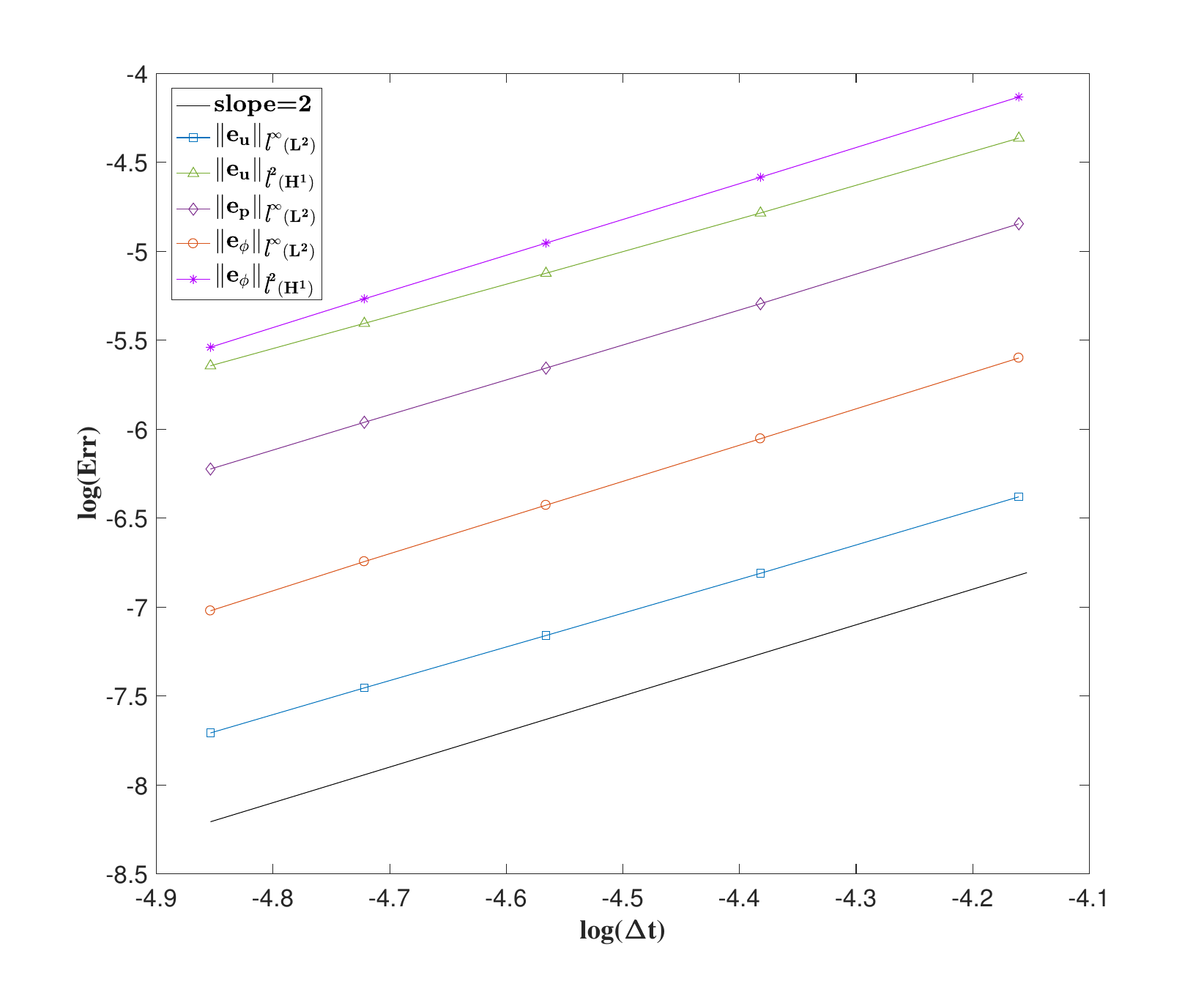}
  \caption{Comparison of numerical errors and convergence rates with $T=0.5$ (Example 2).
Left: {\textbf {Scheme \uppercase\expandafter{\romannumeral1}}}. Right: {\textbf {Scheme \uppercase\expandafter{\romannumeral2}}}.}\label{test2}
\end{figure}

\subsection{Example 3: Numerical simulation of filtration flows}
In this example, we study the impact of different hydraulic conductivities on free flow within porous media, which can be regarded as a simple simulation of a filter. In particular,
we will choose various location and length of the low hydraulic conductivity regions on porous media by using {\textbf {Scheme \uppercase\expandafter{\romannumeral1}}} with EMAC formulation. We set the computational domain as $\Omega_f=[0,2]\times [1.5,2]$, $\Omega_p=[0,2]\times [0,1.5]$ and $\Gamma=[0,2]\times \{1.5\}$ with final time $T=0.5$, $\Delta t=0.005$, $h={1}/{80}$, and other parameters all be 1. 
The initial conditions are given as $\mathbf{u}_f^0=(0,0.01x(x-2)),~\phi^0=0$.

The following configurations are prescribed: For cases (a)–(f), the hydraulic conductivity is set to $K = 10^{-6}$ within two small rectangular subdomains and $K = 1$ elsewhere. For case (g), the hydraulic conductivity tensor is defined as
$$\mathbb{K} = \begin{bmatrix}
\text{rand}(x,y) & 0.1 \text{rand}(x,y) \\
0.1  \text{rand}(x,y) & \text{rand}(x,y)
\end{bmatrix},$$
where $\text{rand}(x, y)$ denotes a uniformly distributed random variable on the interval $(0,1]$.

In Figure \ref{filtration}, we illustrate the fluid flow path driven by initial velocity with the diagonal hydraulic conductivity tensor in different porous-media configurations, and Figure \ref{sto_xi} displays the velocity streamlines with the random hydraulic conductivity tensor and the temporal evolution of $\xi$, we can see that $\xi$ remains consistently close to $1$ with various settings. As shown in these figures, upon encountering the first low hydraulic conductivity region, the flow aligns with the region's geometry, exhibiting localized velocity increases near its corners. A comparable flow pattern, including corner acceleration, occurs within the second low hydraulic conductivity region. Horizontal comparisons, achieved by fixing the width of the first region and varying the width of the second rectangular region, demonstrate that the magnitude of the velocity increase near the corners, which becomes more pronounced as the width of the low hydraulic conductivity region increases. Vertical comparisons yield similar results. This observed behavior is consistent with the Venturi effect in fluid dynamics or the narrow tube effect observed in airflow dynamics.

\begin{figure}[htbp]
  \centering
  \includegraphics[scale=0.25]{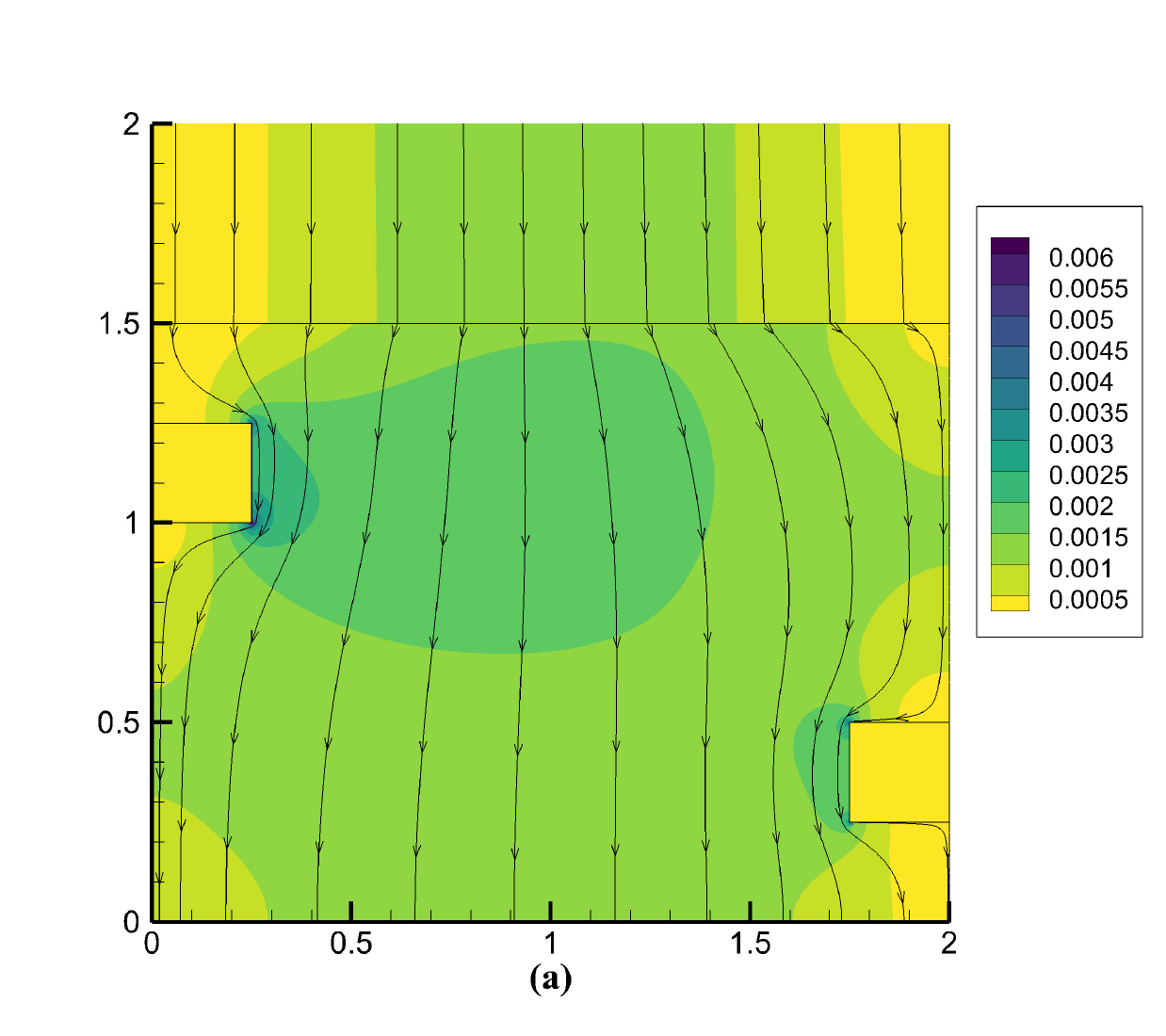}
  \includegraphics[scale=0.25]{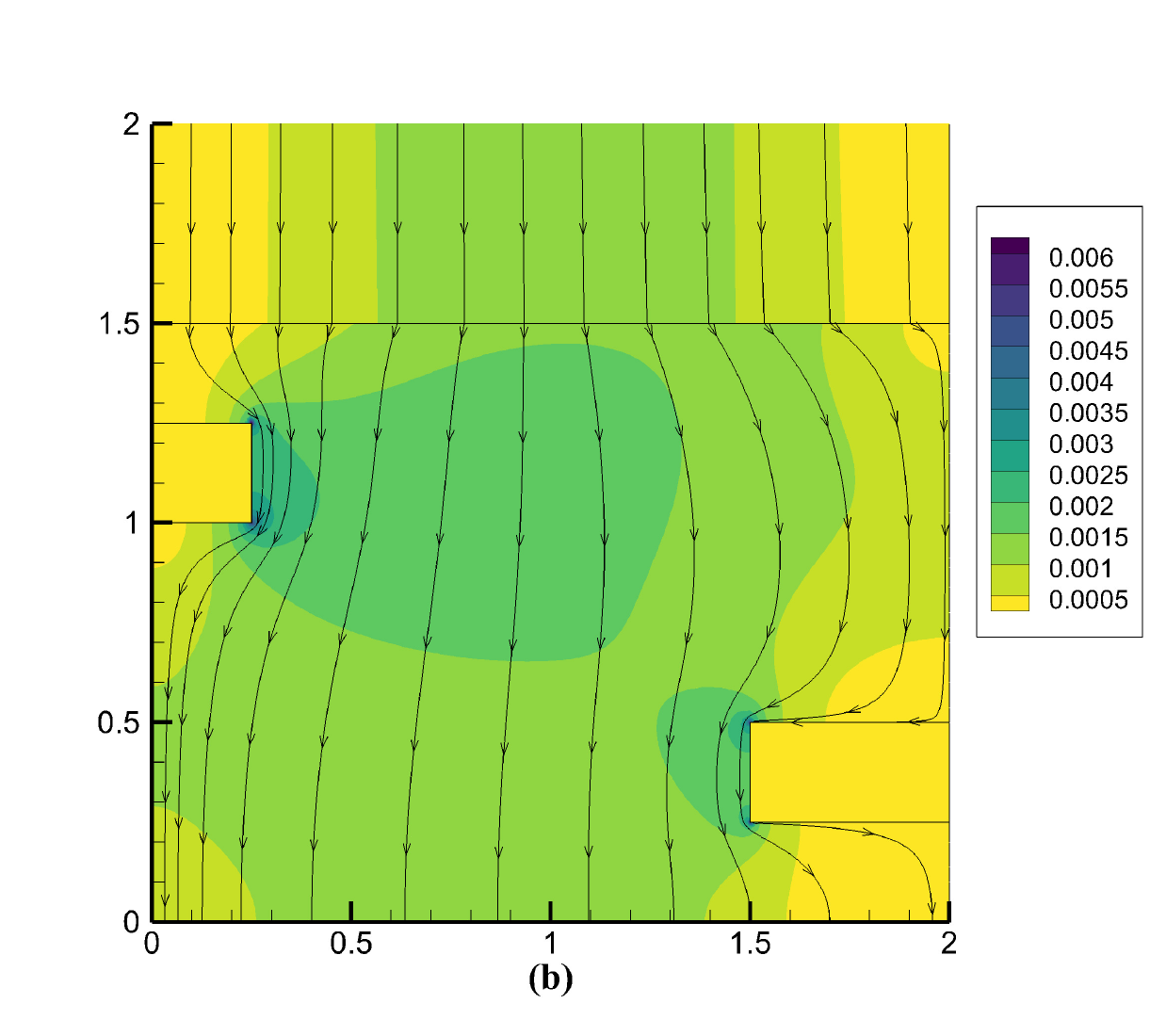}
  \includegraphics[scale=0.25]{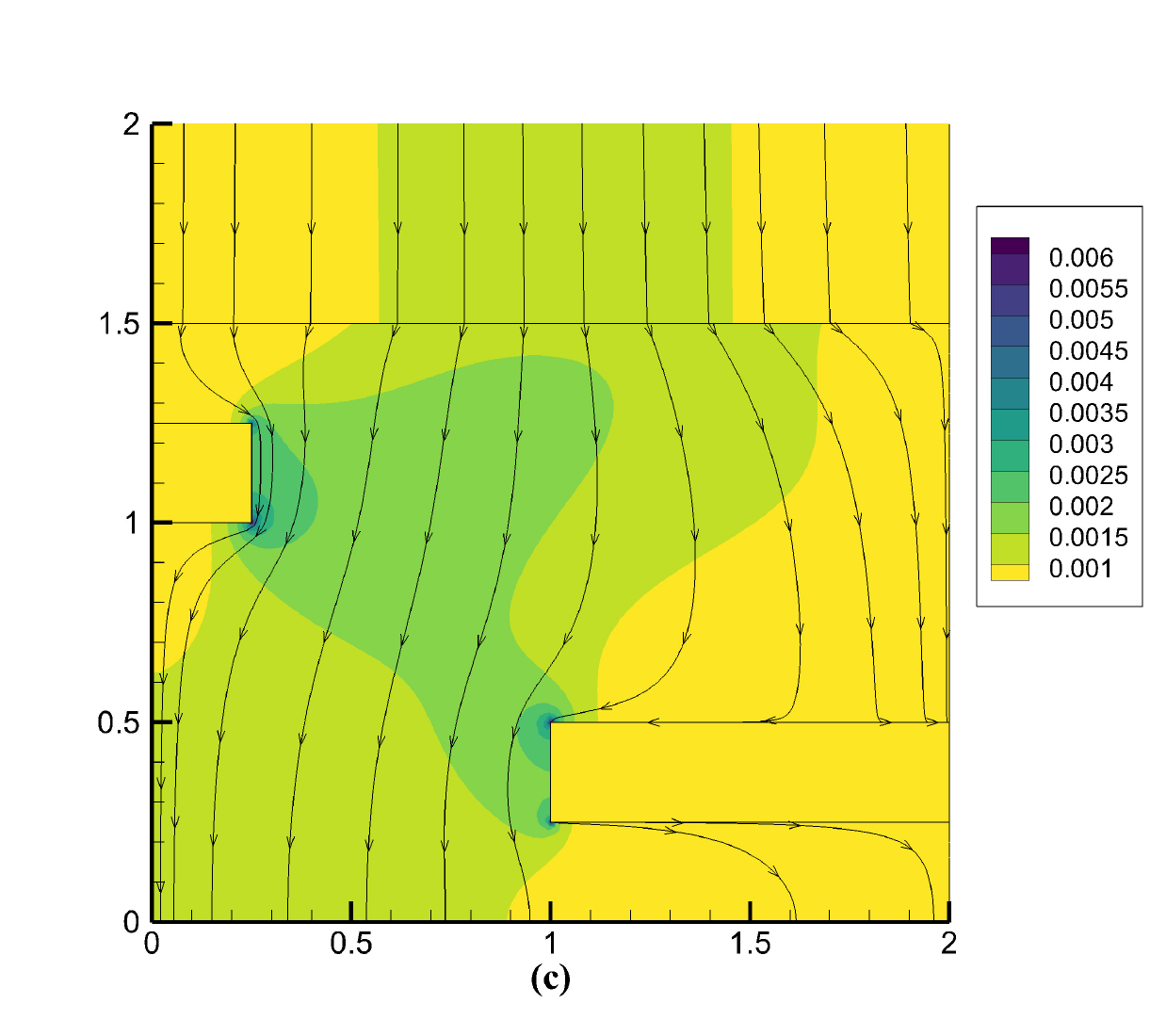}
  \includegraphics[scale=0.25]{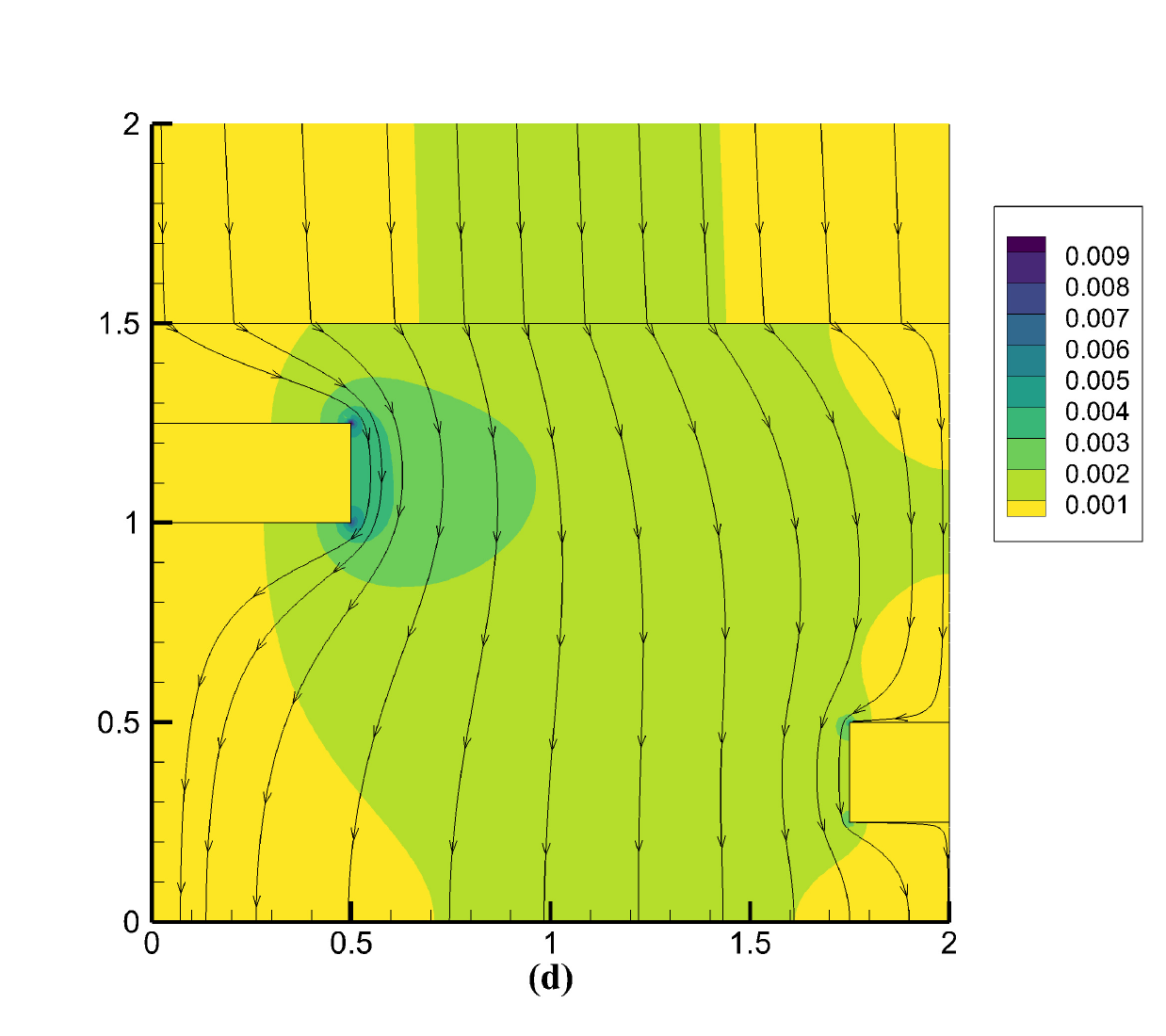}
  \includegraphics[scale=0.25]{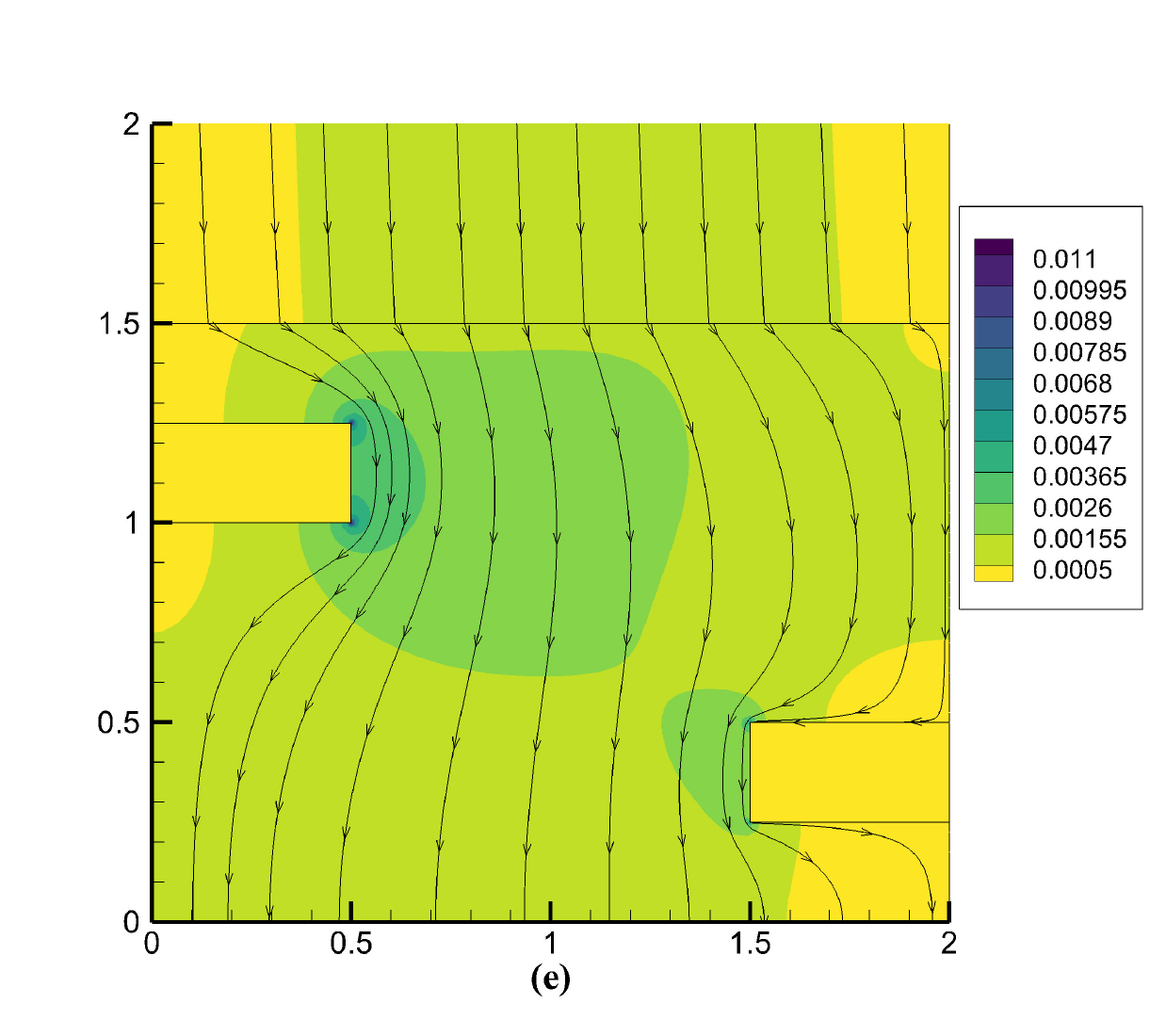}
  \includegraphics[scale=0.25]{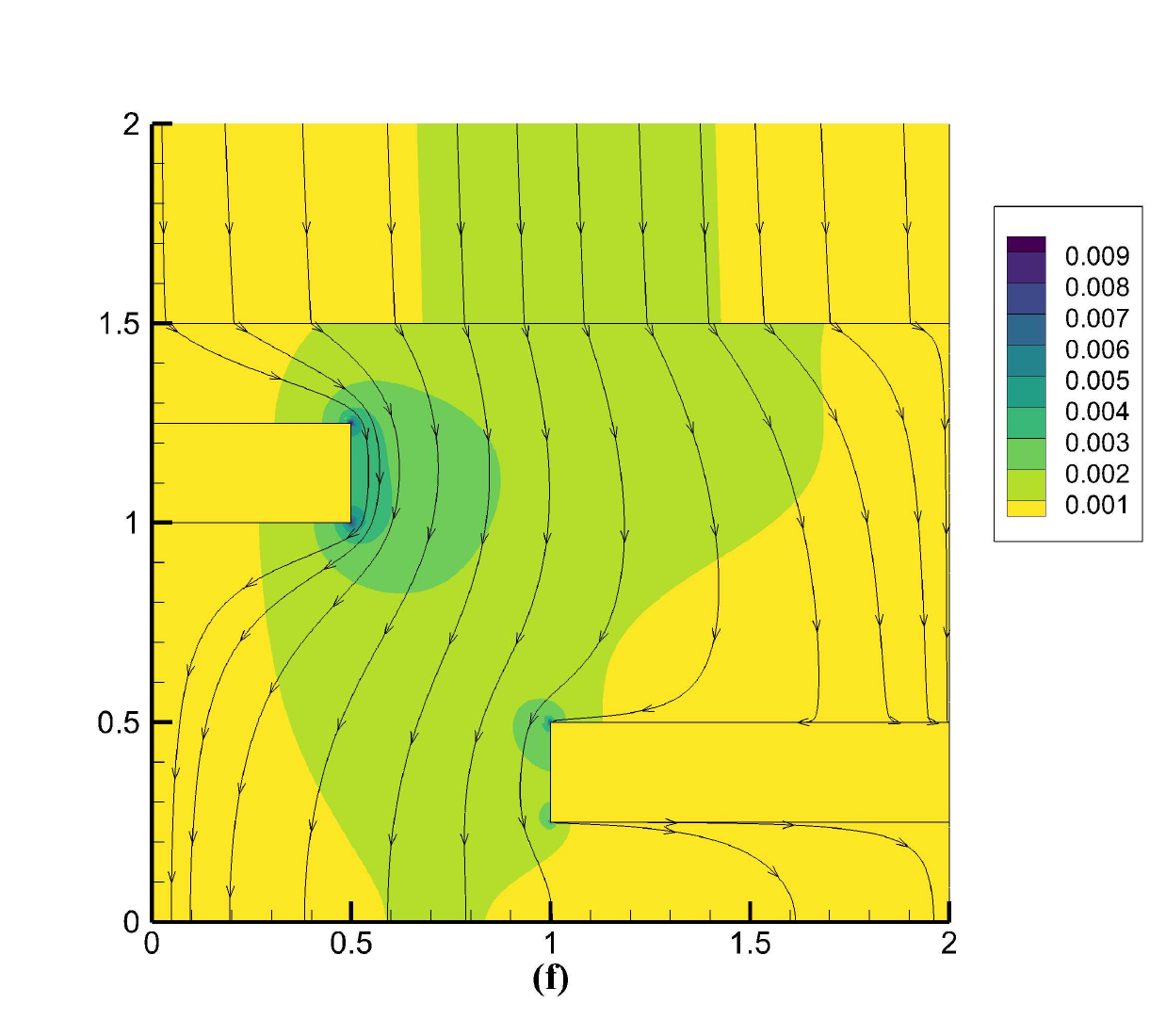}
     \caption{{\color{black}{Comparison of velocity field distributions driven by initial velocity in different porous-media configurations.}}}
\label{filtration}
\end{figure}

\begin{figure}[htbp]
  \centering
\includegraphics[scale=0.25]{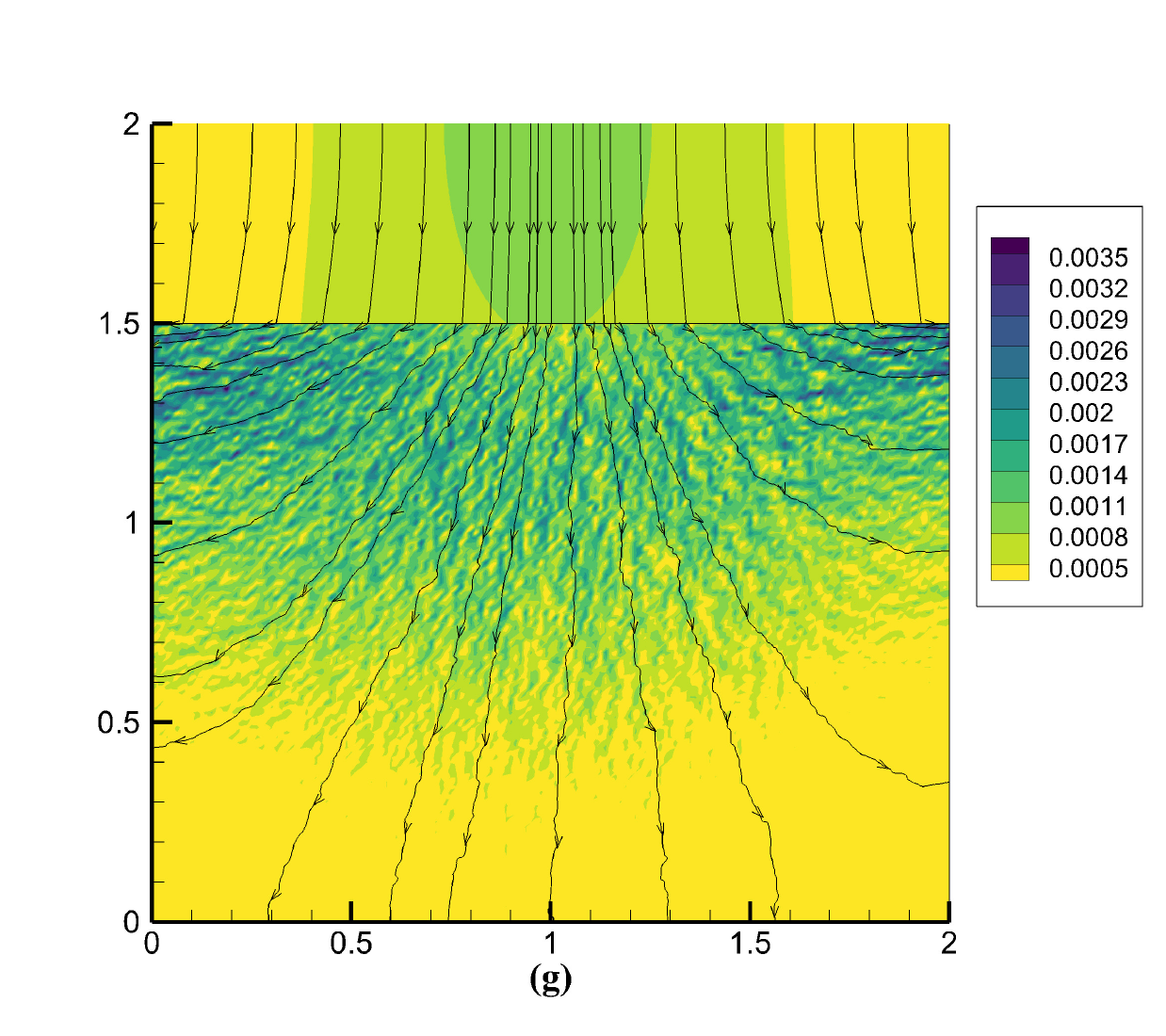}
\includegraphics[scale=0.25]{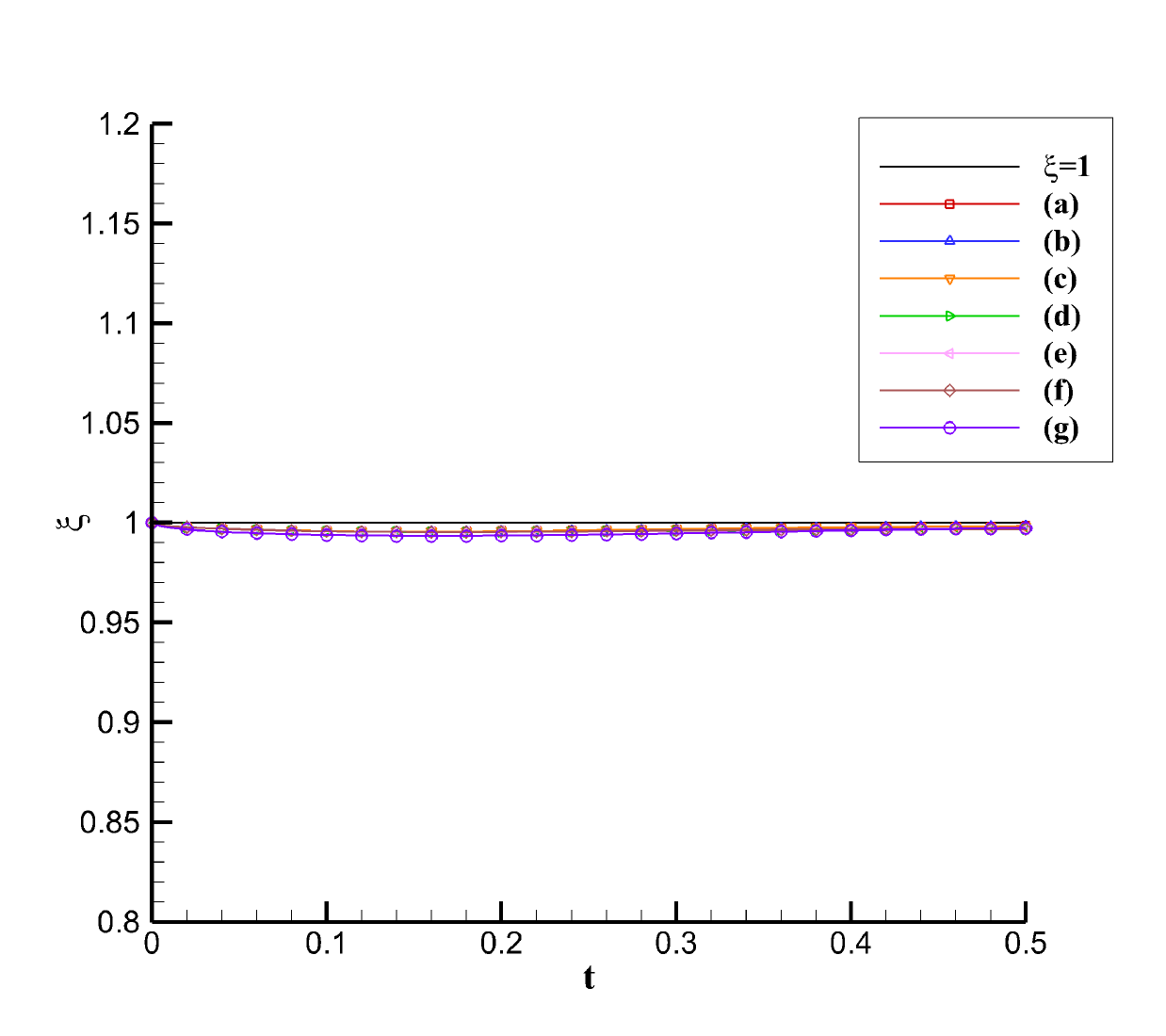}
     \caption{{\color{black}{Left: The velocity field distributions with random hydraulic conductivity tensor; Right: The comparison of $\xi$ with different settings.}}}
\label{sto_xi}
\end{figure}

\subsection{Example 4: Phase separation dynamics with random initial conditions}
In this numerical example, we validate {\textbf {Scheme \uppercase\expandafter{\romannumeral3}}} with EMAC formulation by simulating phase separation dynamics for the Cahn--Hilliard--Navier--Stokes--Darcy model. The initial phase field variable is set to $$\phi^0=y-1+0.01 rand(x,y),$$ where $rand(x, y)$ denotes a uniformly distributed random variable on $[-1,1]$, with parameters $\Delta t=0.005,h=1/100,\nu_f=\nu_p=0.1,\alpha=0.01,\chi=1,
\lambda=0.1,M(\phi)=\epsilon\sqrt{(1-\phi)^2+\epsilon^2},\epsilon=0.01$, and hydraulic conductivity tensor $$
{\color{black}{\mathbb{K}=\begin{bmatrix}
0.5 & 0.1 \\
0.1 & 0.2
\end{bmatrix}.}}
$$
The computational domain $\Omega=[0,1]\times [0,2]$ includes a free-flow region $\Omega_f=[0,1]\times [0,1]$ and a porous-media region $\Omega_p=[0,1]\times [1,2]$.

\begin{figure}[h!]
  \centering
  \includegraphics[scale=0.19]{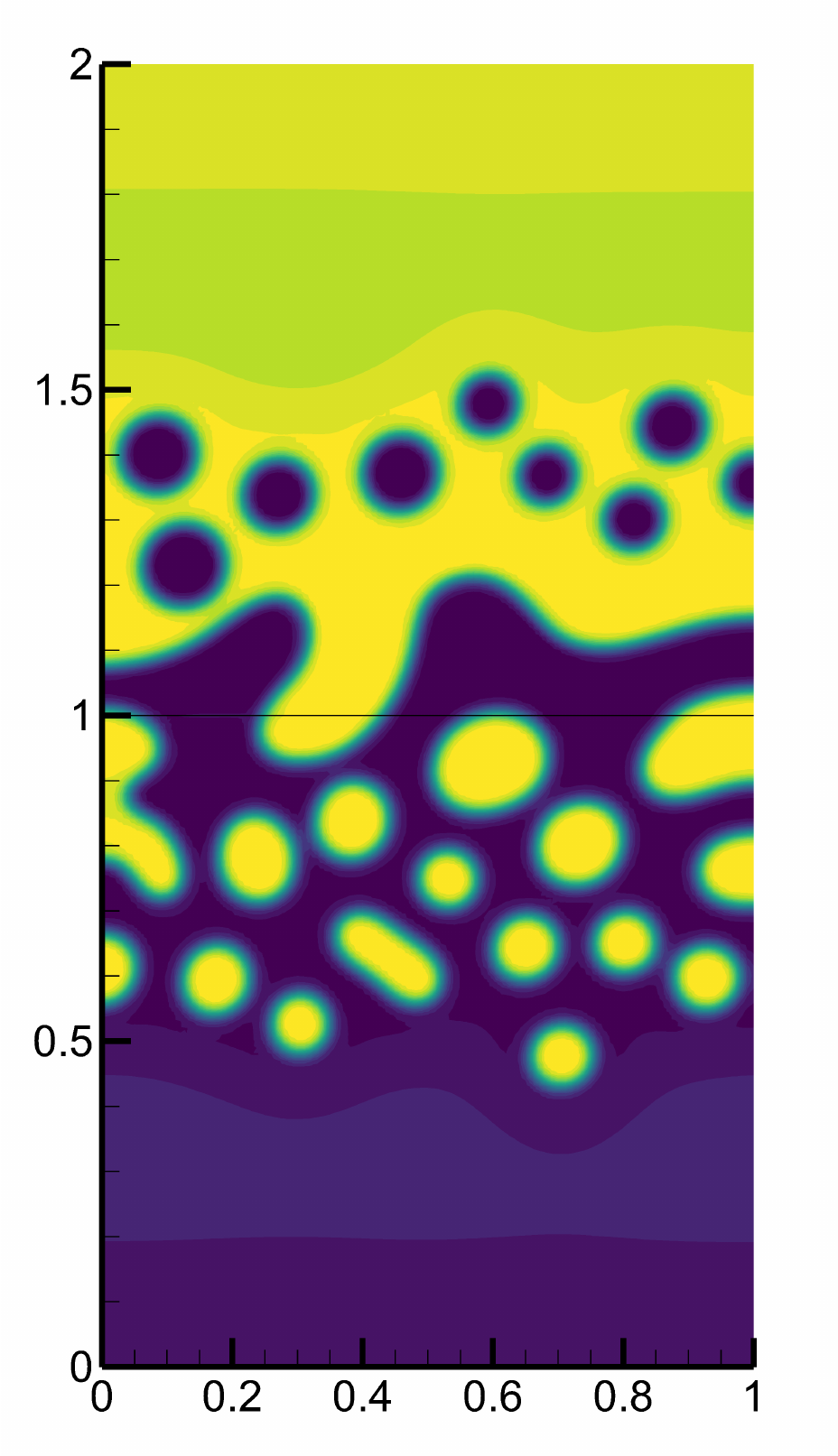}
  \includegraphics[scale=0.19]{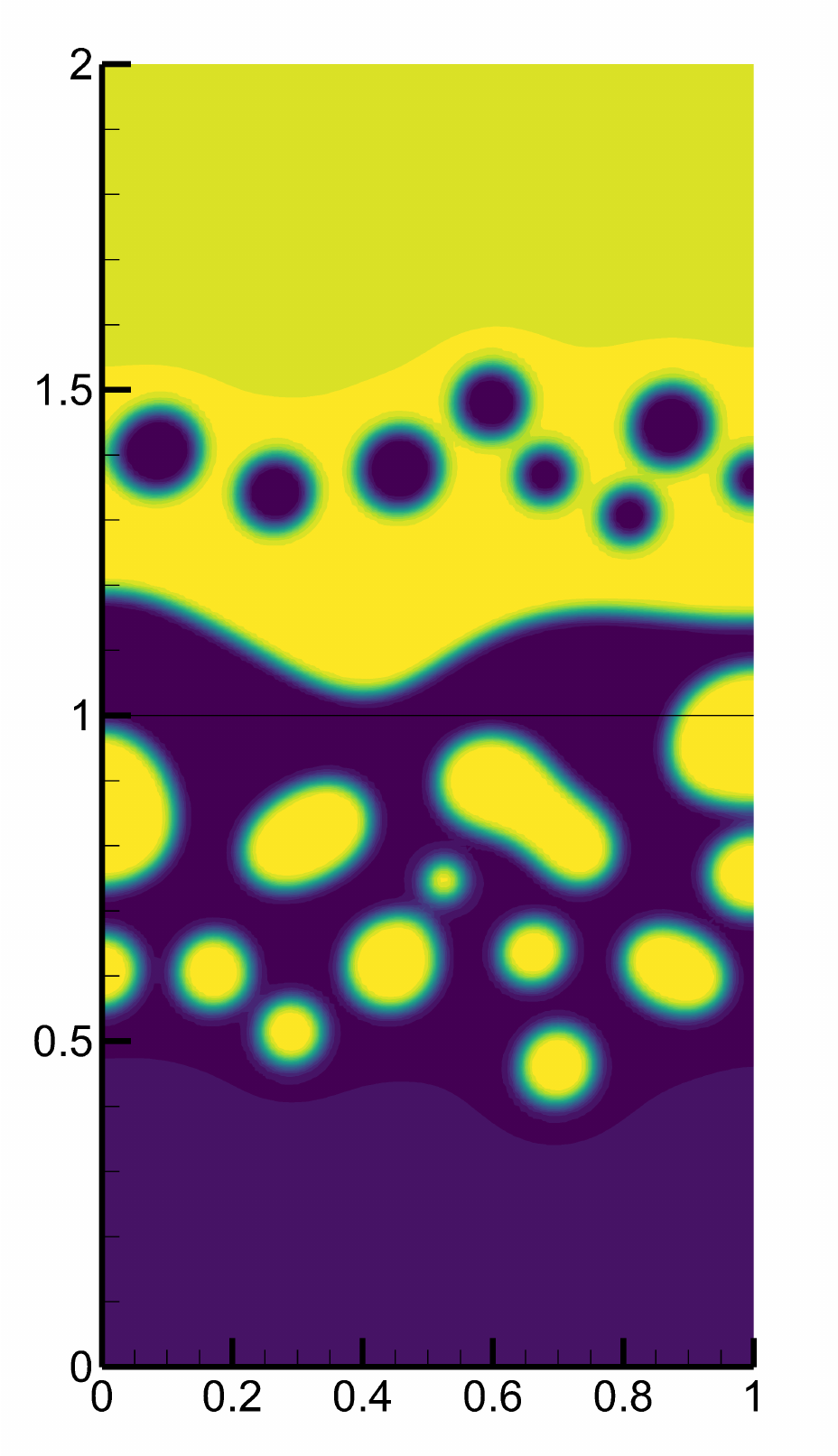}
  \includegraphics[scale=0.19]{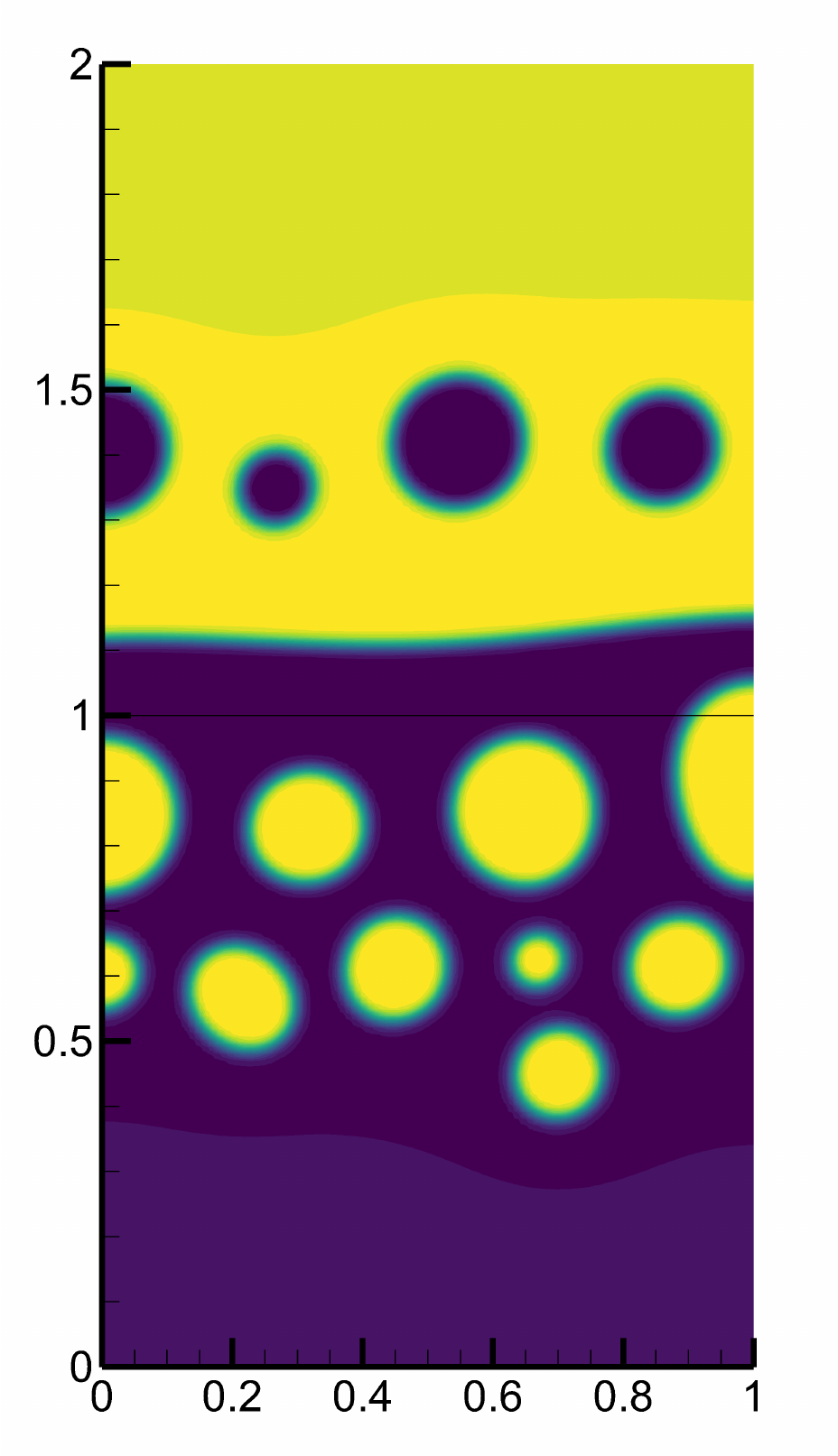}
  \includegraphics[scale=0.19]{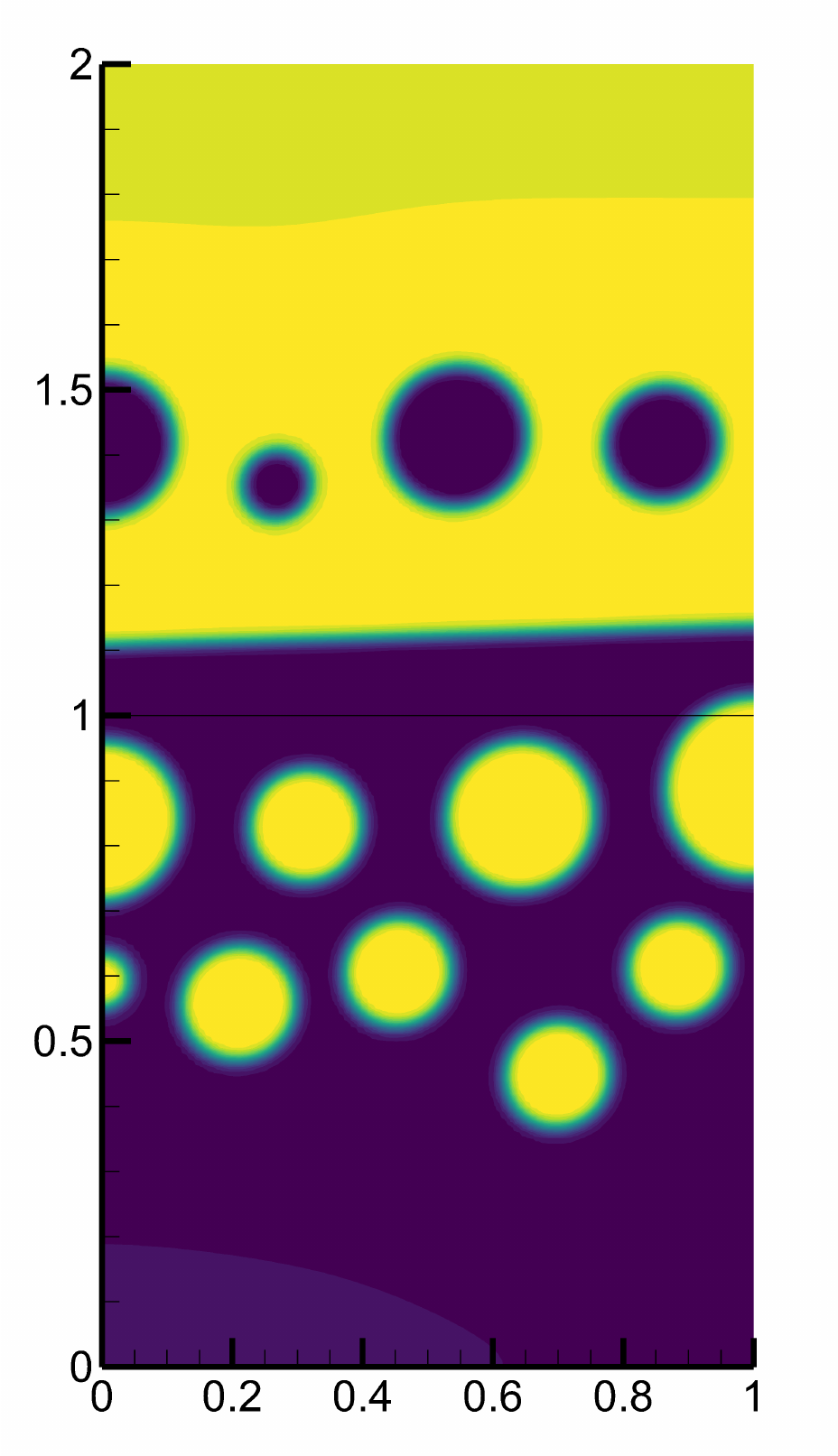}
  \includegraphics[scale=0.19]{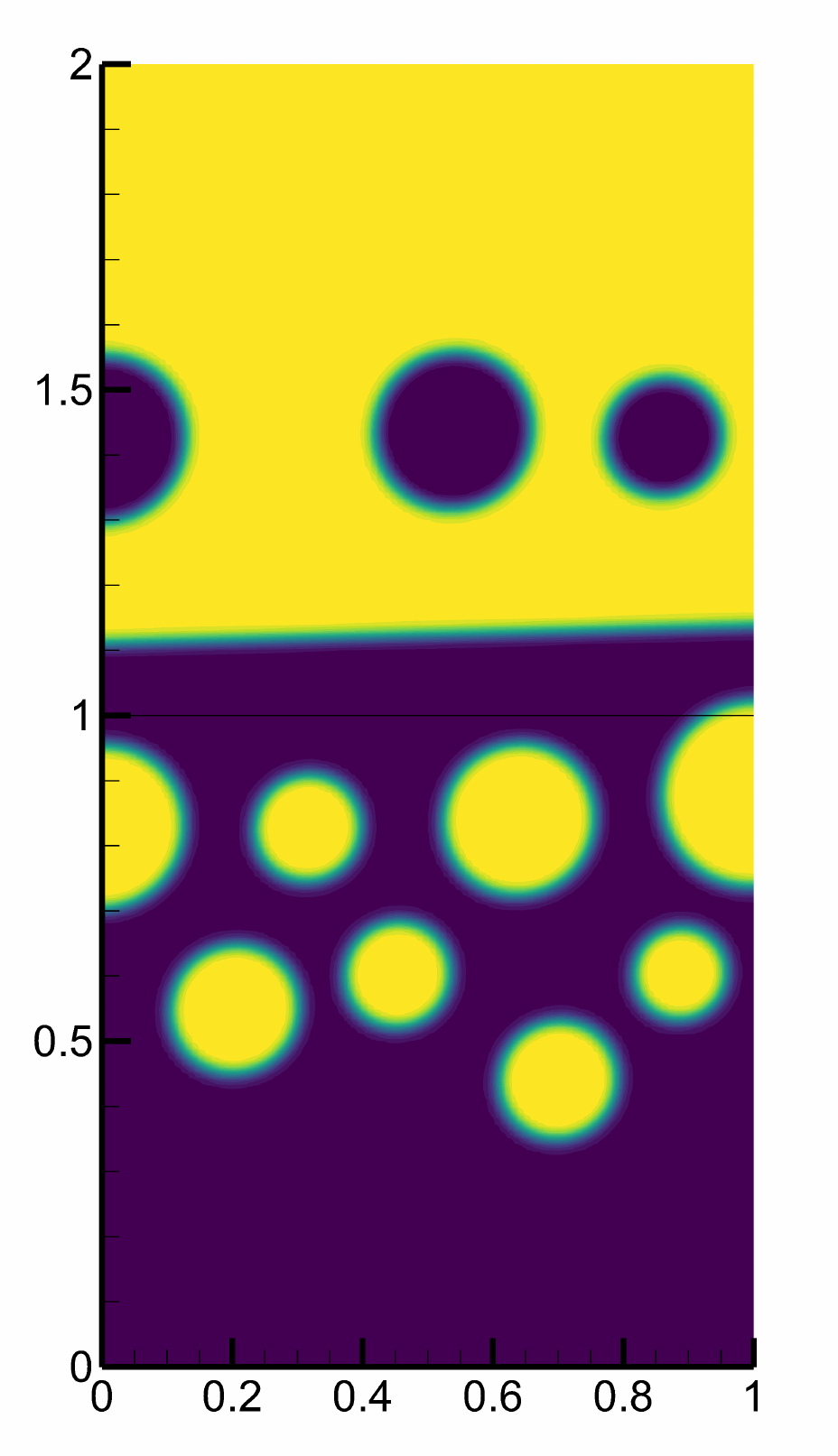}
     \caption{{\color{black}{Phase-field evolution from random initial conditions at T=0.5, 1, 2, 4, 10.}}}
\label{phasesaparation}
\end{figure}

As evidenced in Figure \ref{phasesaparation}, the phase field variable exhibits random spatial distribution throughout the domain at initial time. With time evolution, surface tension effects drive progressive phase separation, leading to the emergence of distinct domains characterized by sharp interfacial boundaries. The system ultimately evolves into stable, coarsened regions with minimized interfacial energy.

\subsection{Example 5: Surface tension-driven droplet relaxation}
This numerical example simulates the phase relaxation phenomenon. The initial phase-field variable is given by $$\phi^0=\tanh\left( \frac{ 0.25 + 0.1 \cos( n \arctan \frac{y-1}{x-0.5}  ) - \sqrt{(x-0.5)^2 + (y-1)^2}  }{ \sqrt{2} \epsilon}\right),$$
with the computational domain and all other parameters remain consistent with the previous example, except for the viscosity coefficients $\nu_f=0.1$, $\nu_p=1$, and {\color{black}{$\mathbb{K}=0.01\mathbb{I}$}}. The results for $n=4$ and $n=6$ {\color{black}{(where $n$ denotes the number of petals in the initial shape)}} are presented in Figures \ref{4phaserelaxation} and \ref{6phaserelaxation}, respectively. As shown in these figures, the initial droplet in the central region progressively evolves into a circular shape under the action of surface tension as the time evolution, consistent with the energy dissipation satisfied by the Cahn--Hilliard--Navier--Stokes--Darcy model. In addition, we observe that phase relaxation occurs more rapidly in the free-flow domain than in the porous-media domain due to the different viscosity coefficients assigned to these domains, which aligns with the physical property of the model and is consistent with \cite{gao2023fully,gao2024second,gao2024fully}.

\begin{figure}[h!]
  \centering
  \includegraphics[scale=0.19]{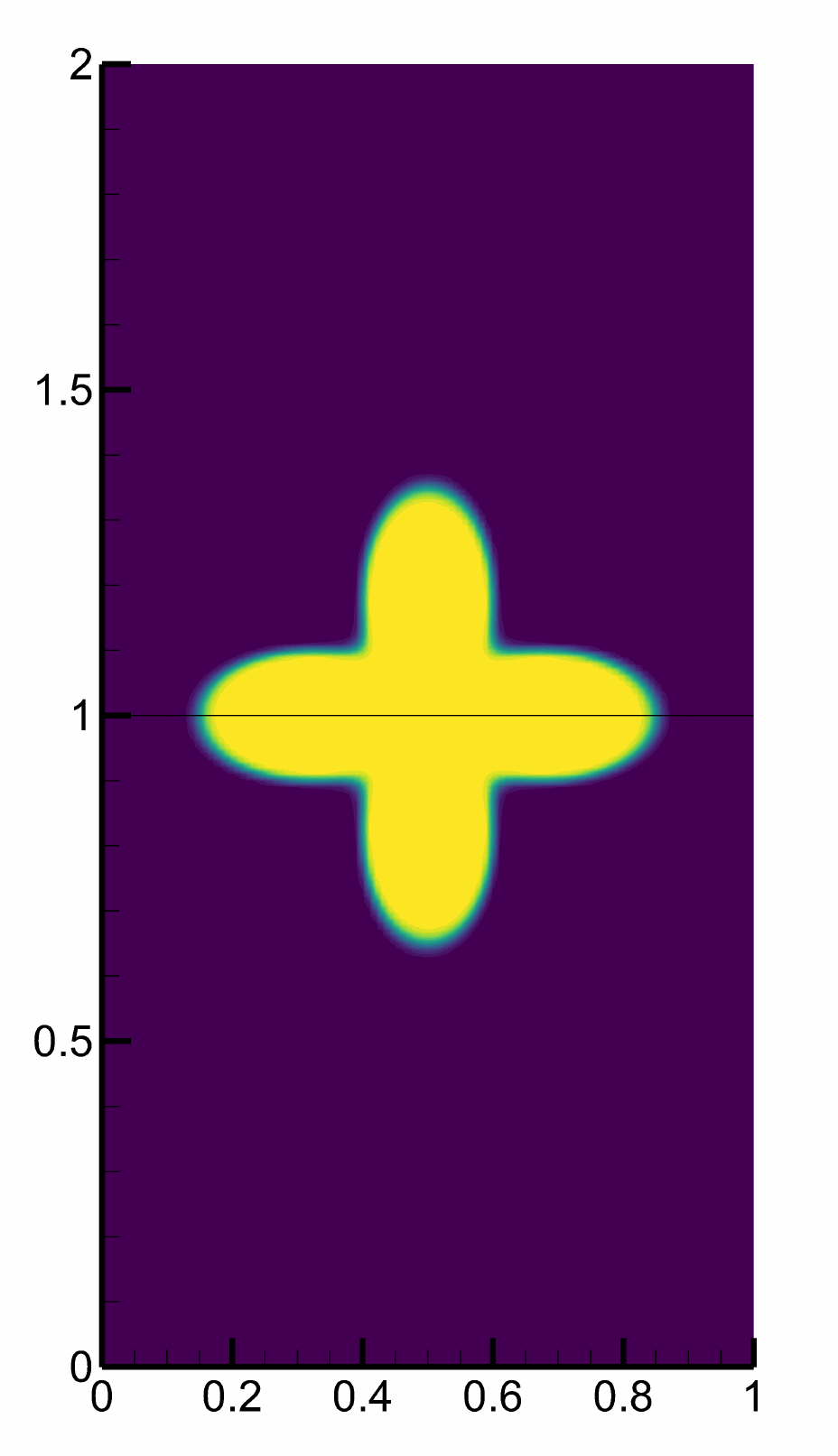}
  \includegraphics[scale=0.19]{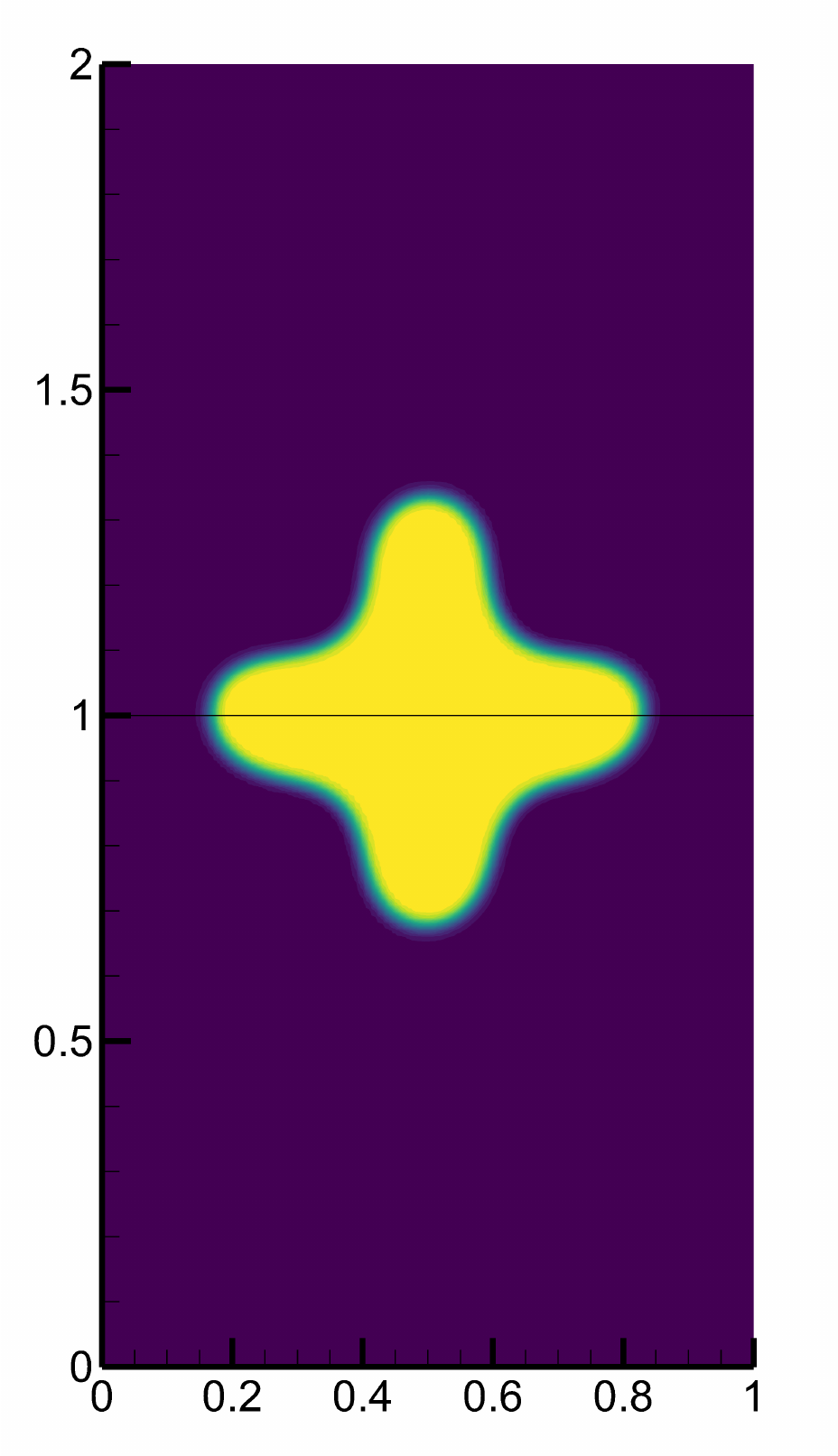}
  \includegraphics[scale=0.19]{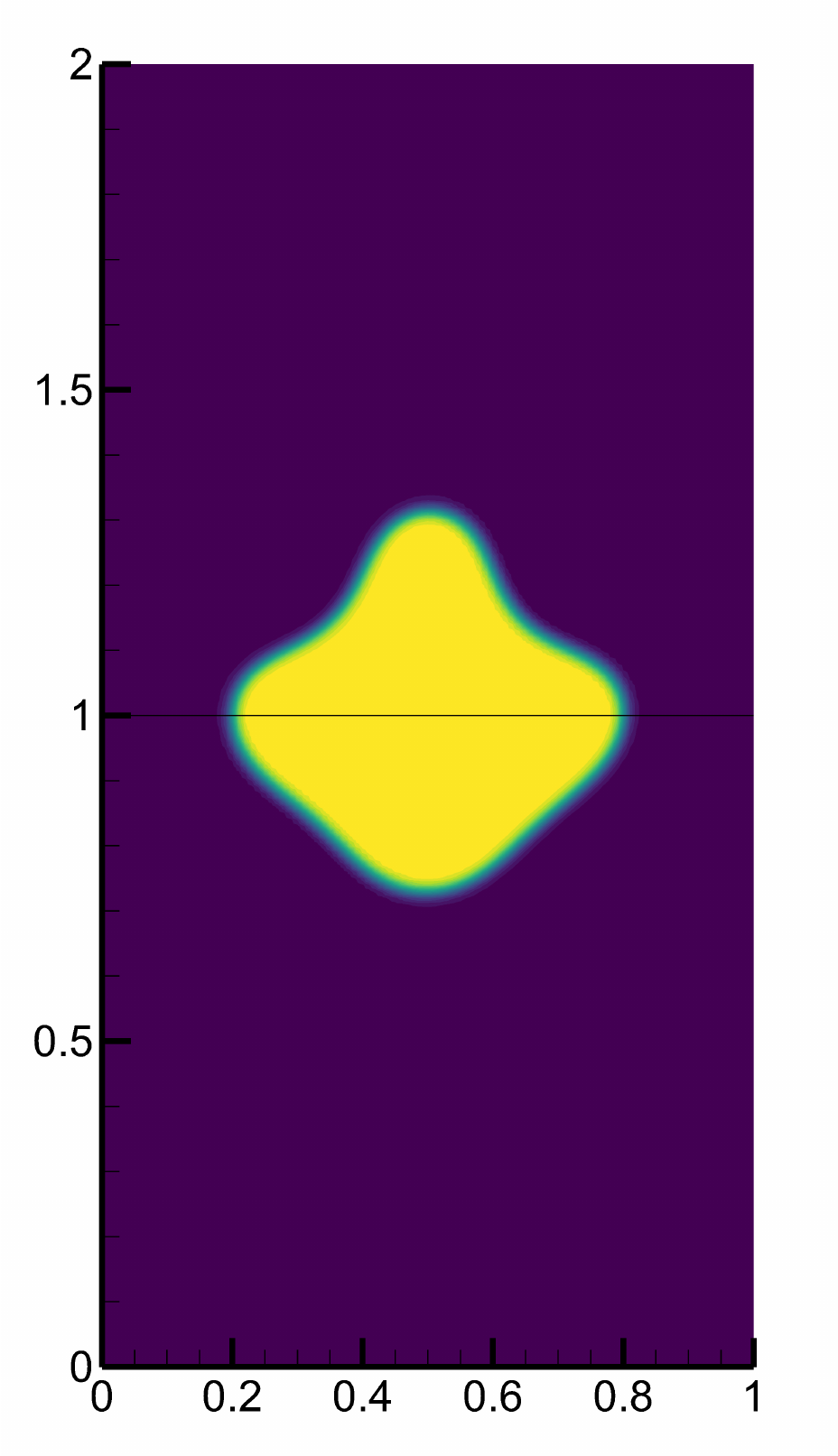}
  \includegraphics[scale=0.19]{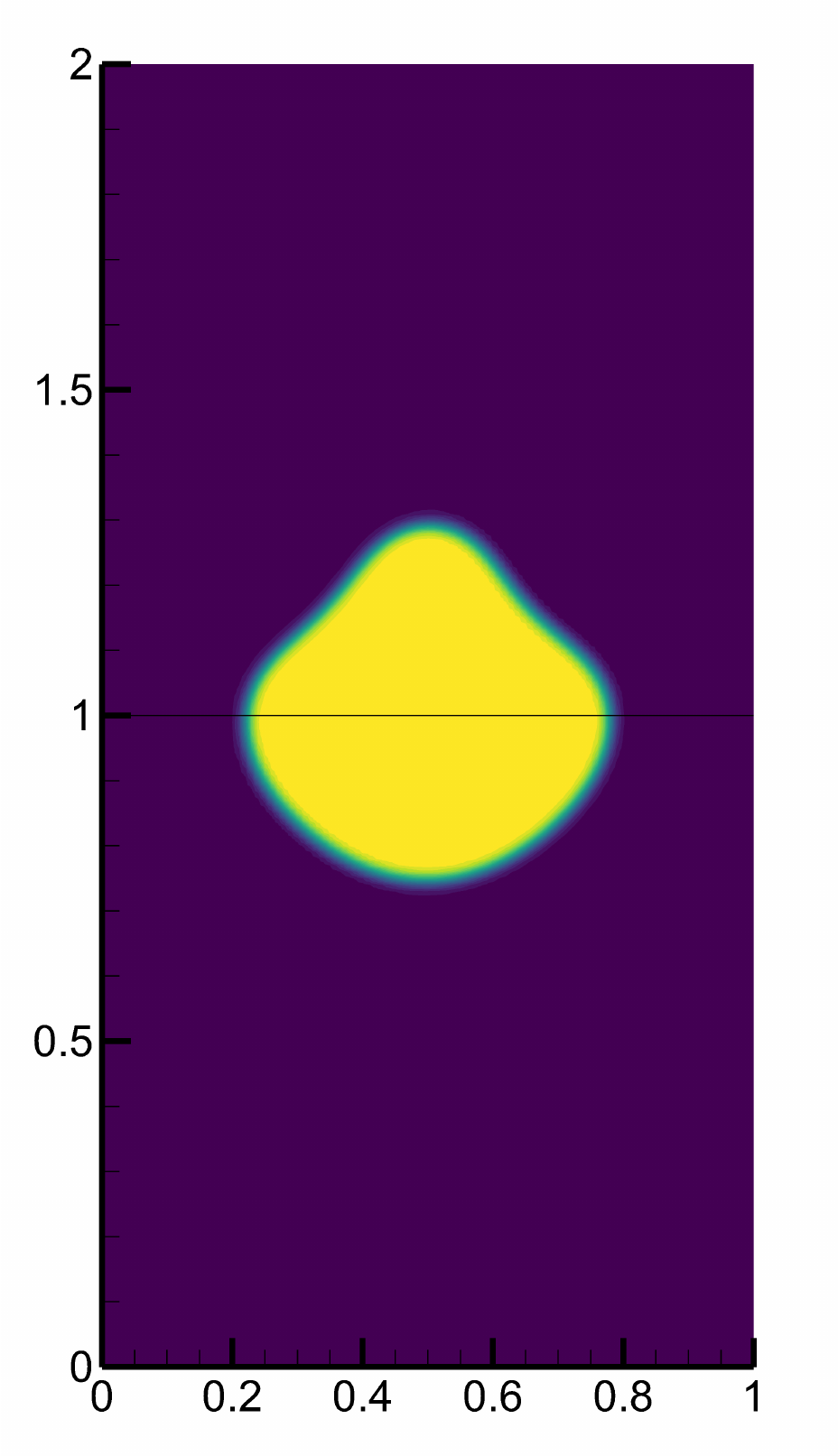}
  \includegraphics[scale=0.19]{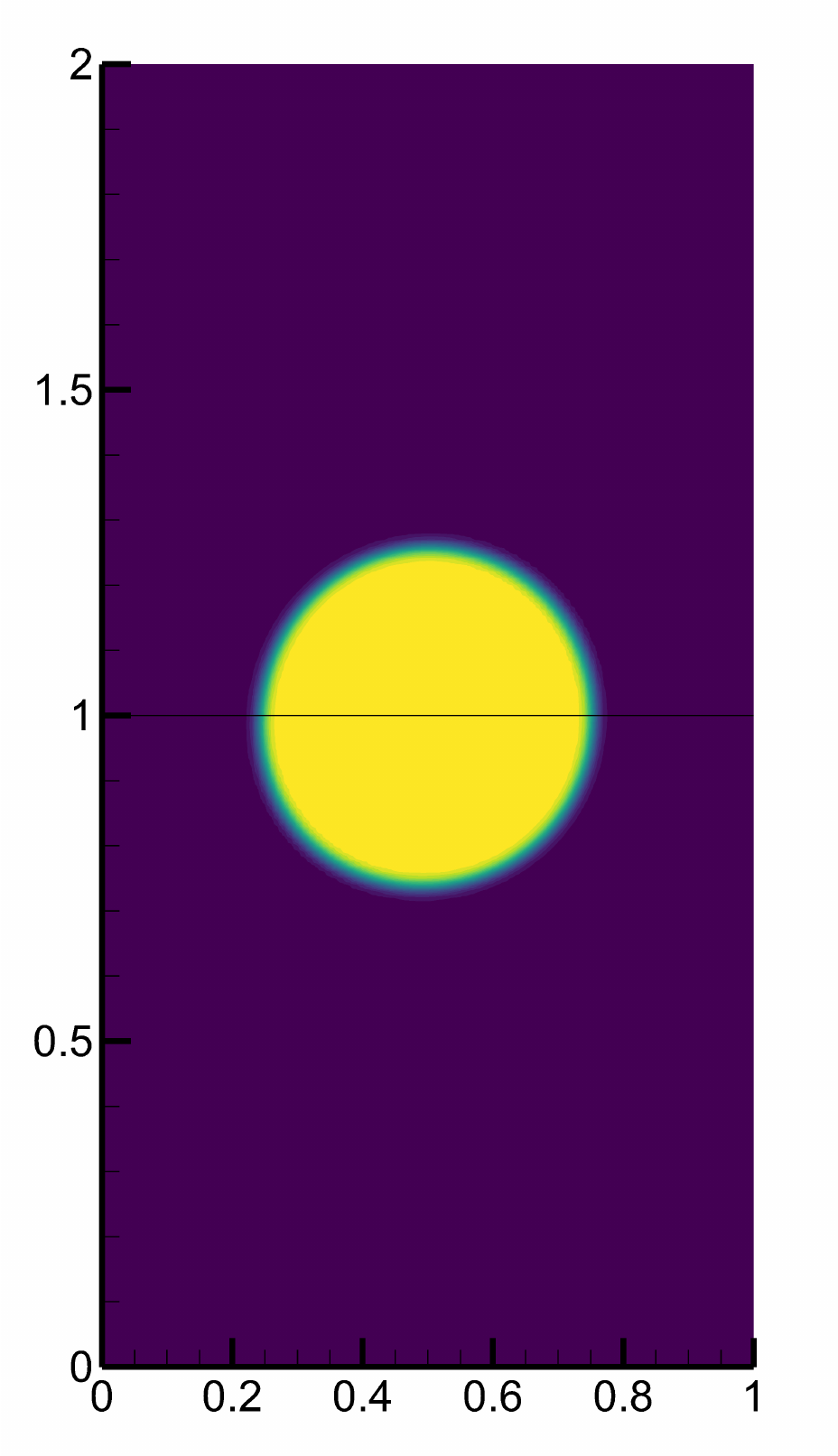}
     \caption{Surface tension-driven droplet relaxation (n=4) at T=0, 0.2, 0.6, 1, 4.}
\label{4phaserelaxation}
\end{figure}

\begin{figure}[h!]
  \centering
  \includegraphics[scale=0.19]{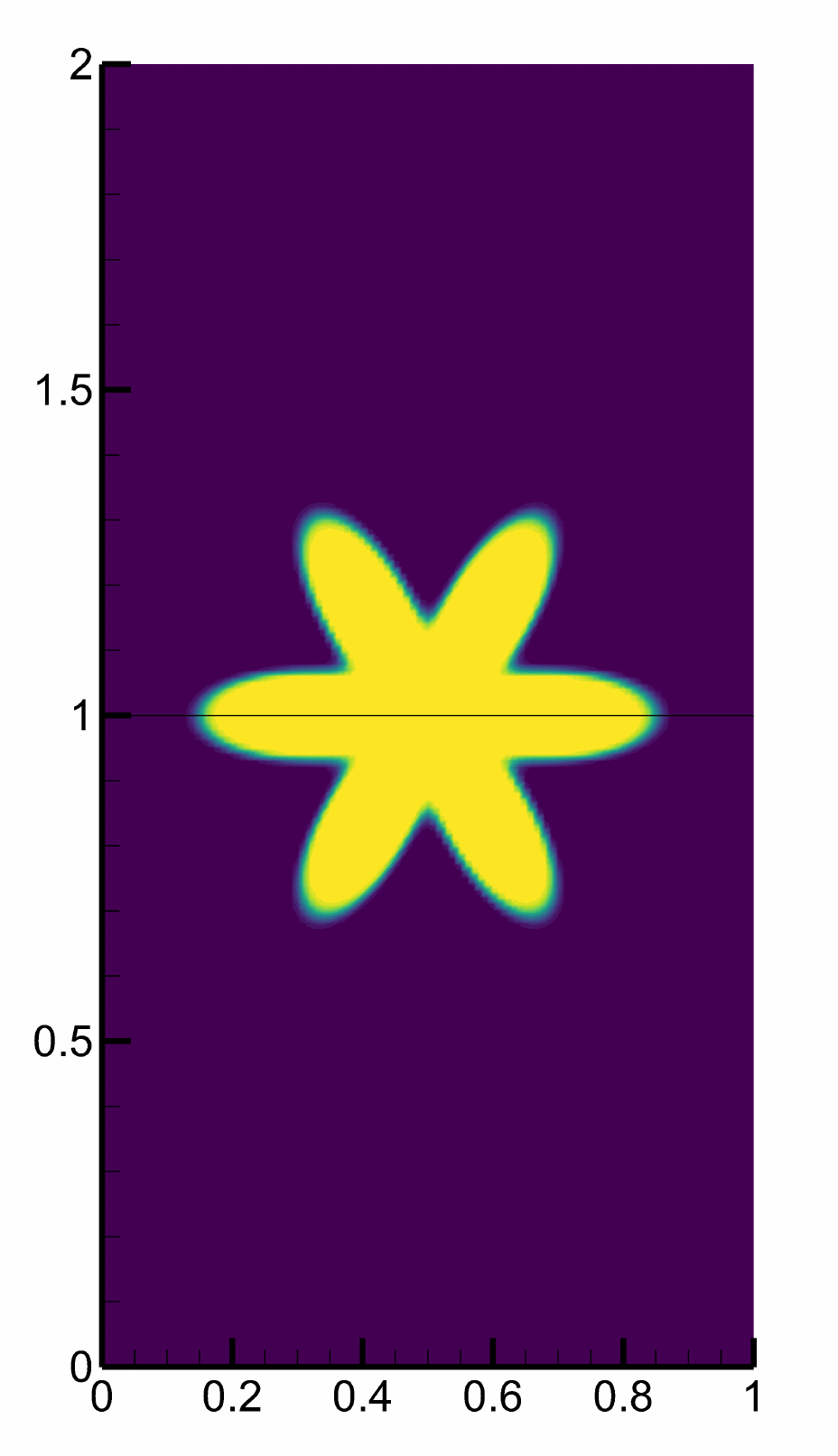}
  \includegraphics[scale=0.19]{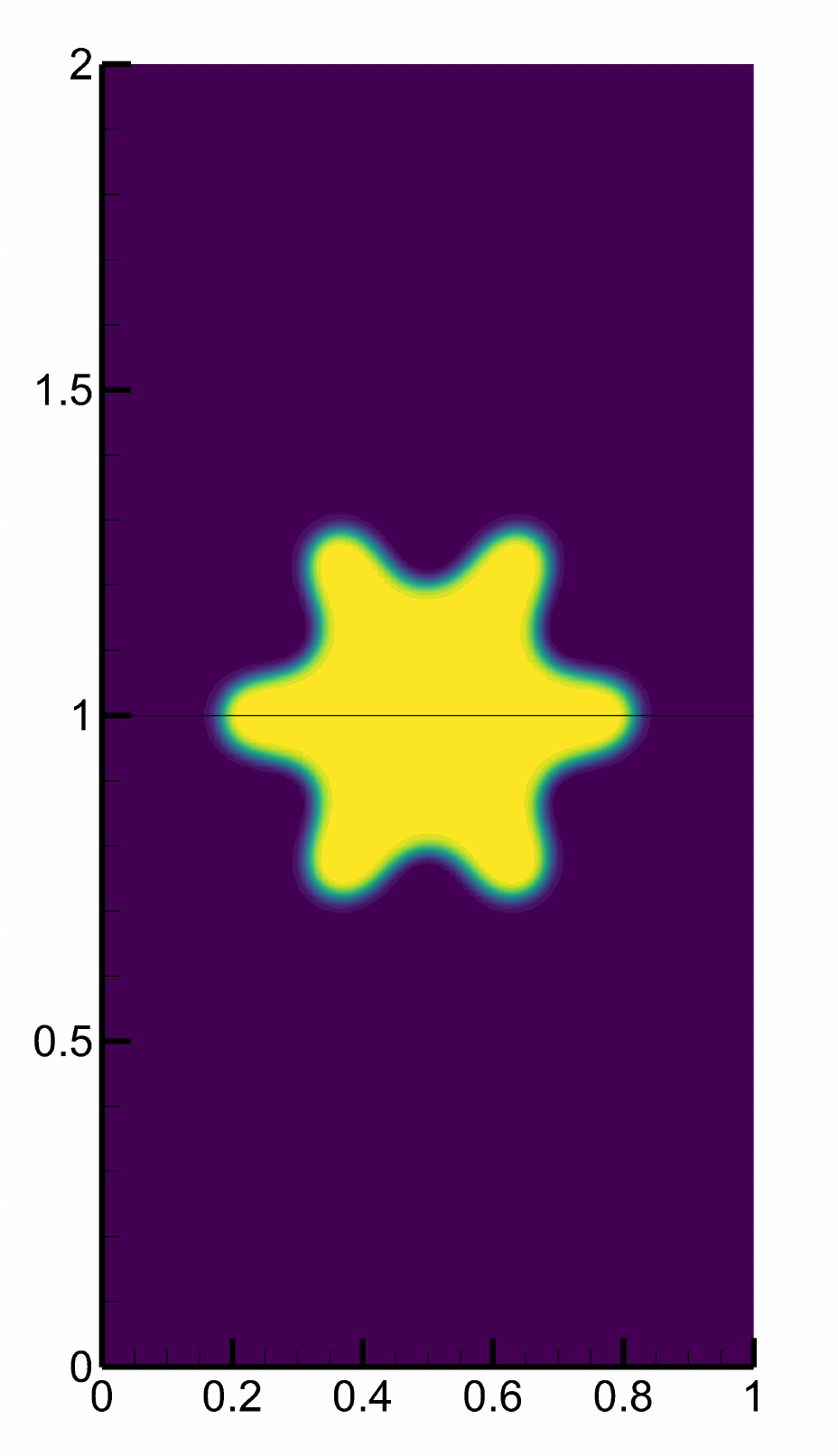}
  \includegraphics[scale=0.19]{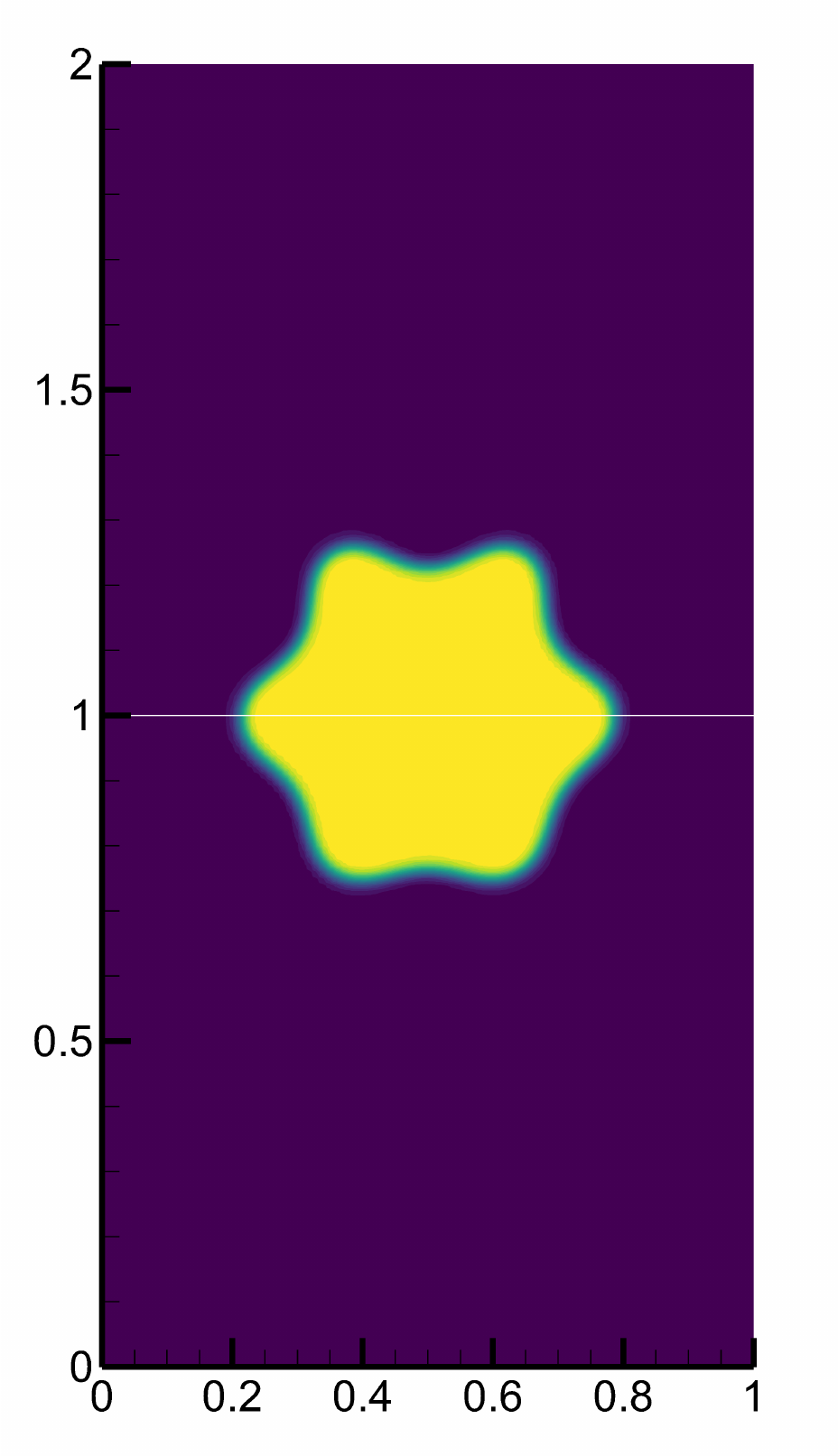}
  \includegraphics[scale=0.19]{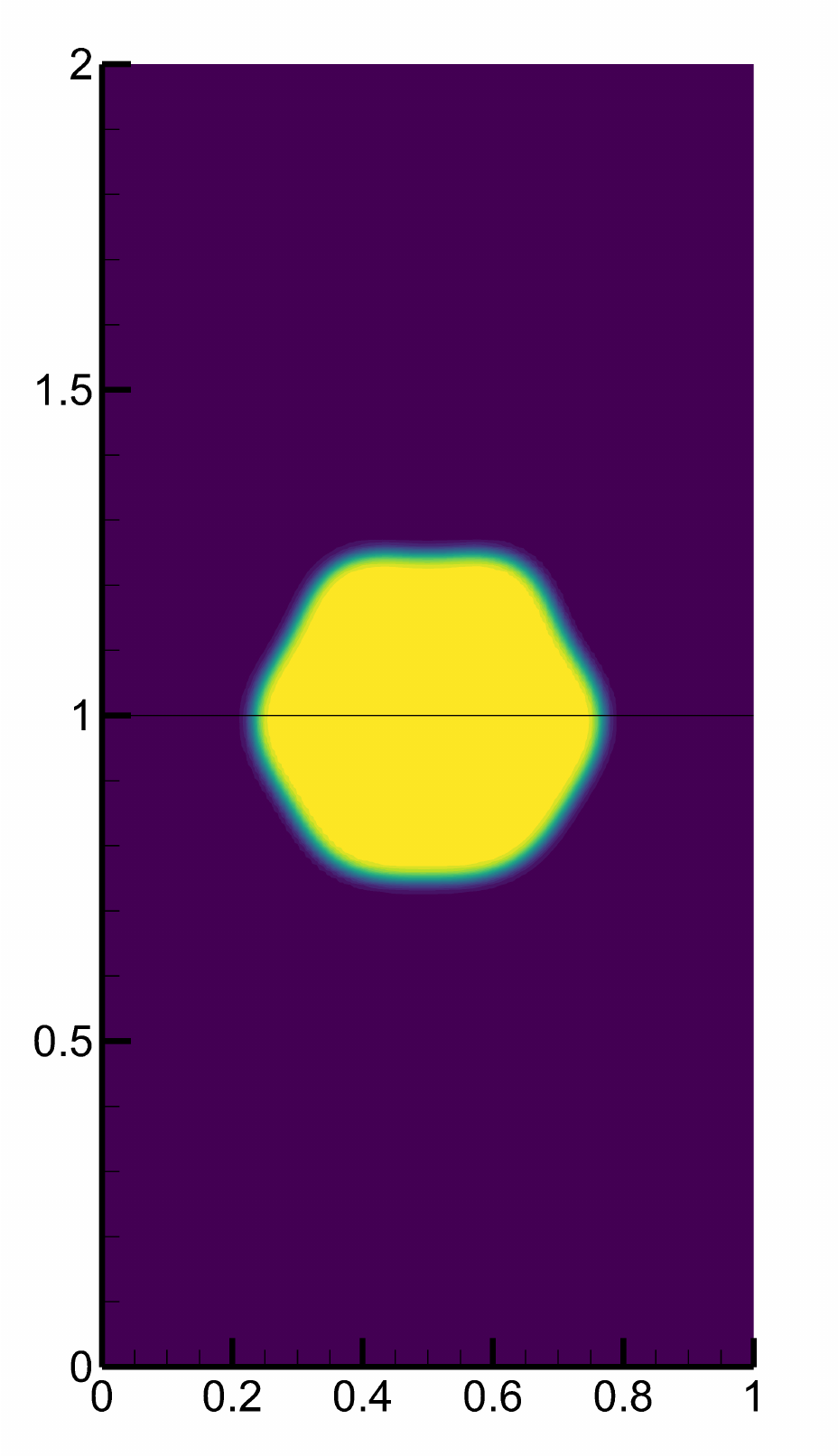}
  \includegraphics[scale=0.19]{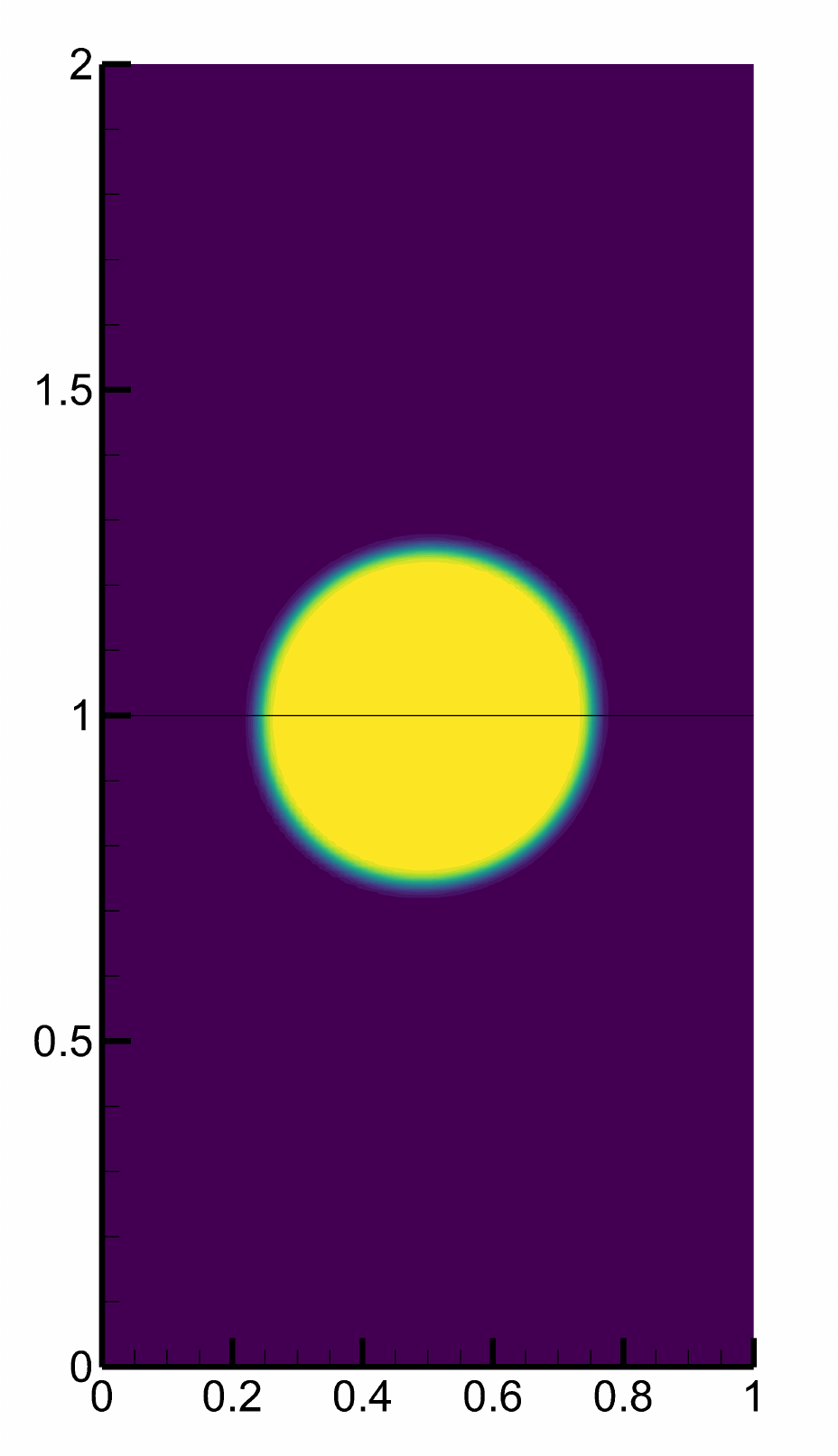}
     \caption{Surface tension-driven droplet relaxation (n=6) at T=0, 0.2, 0.4, 0.6, 4.}\label{6phaserelaxation}
\end{figure}

\subsection{Example 6: Buoyancy-driven bubble rising}
We investigate bubble dynamics in a two-phase flow driven by buoyancy force in this example. We add the buoyancy term $\mathbf{B} (\phi-\bar{\phi}) $ at the right of (\ref{CHNS1}) and (\ref{CHD1}), where $\mathbf{B} =(0,5)^T$ and $\bar{\phi}$ represents the spatially averaged order parameter. The computational domain and {\color{black}{other parameters are consistent with Example 5}}. Initially, we set the phase variable as
$$
\phi^0 =\tanh\left( \frac{R - \sqrt{(x-x_0)^2 + (y-y_0)^2}} { \sqrt{2} \epsilon} \right),
$$
where $R=0.3, x_0=0.5, y_0=0.5$. Figure \ref{Buoyancy} illustrates the evolution of a bubble in the free-flow domain gradually rising toward the porous-media domain under buoyancy force, precisely capturing the moment when the bubble crosses the interface between the two domains.

\begin{figure}[h!]
 \centering
 \includegraphics[scale=0.19]{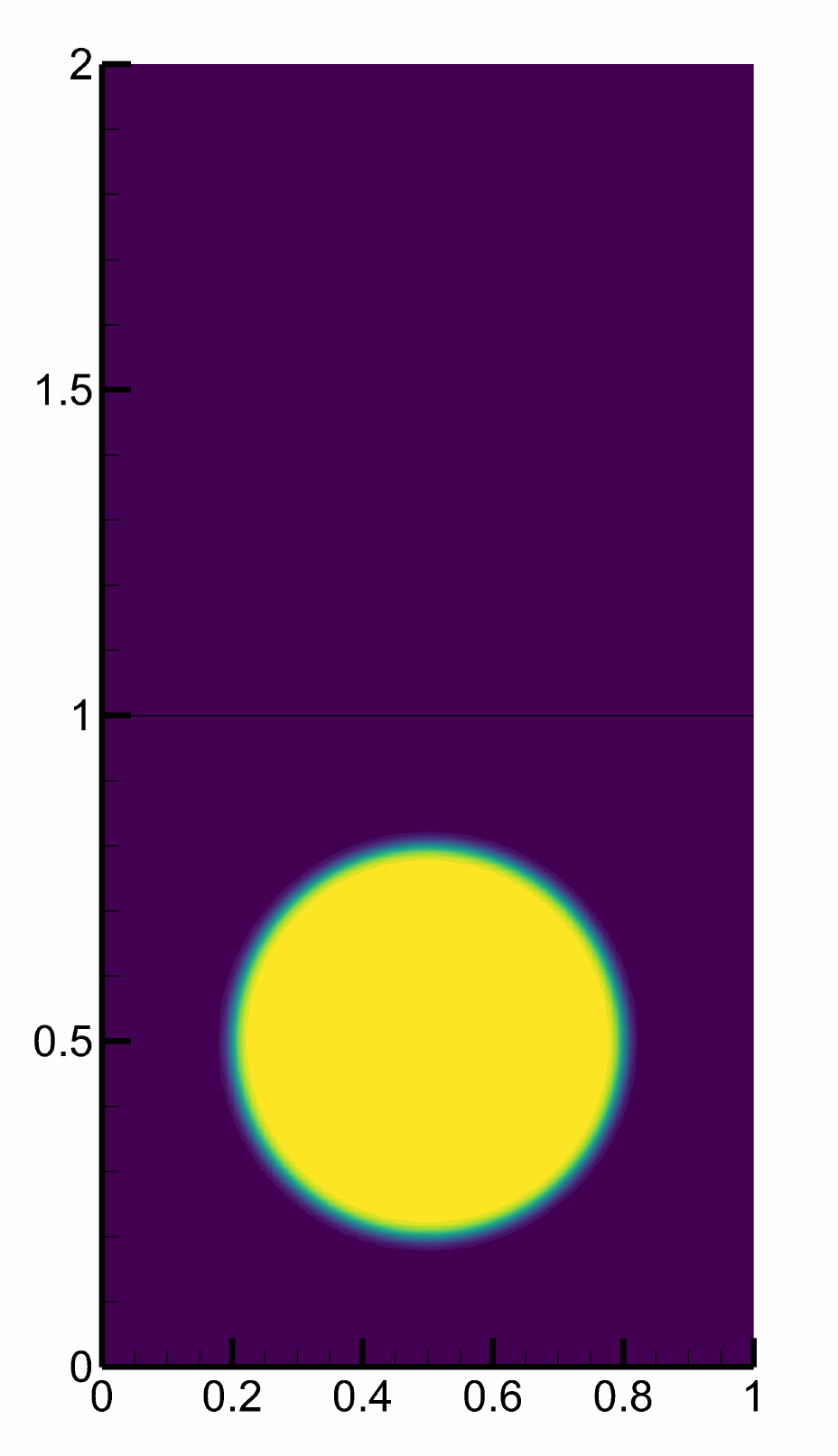}
 \includegraphics[scale=0.19]{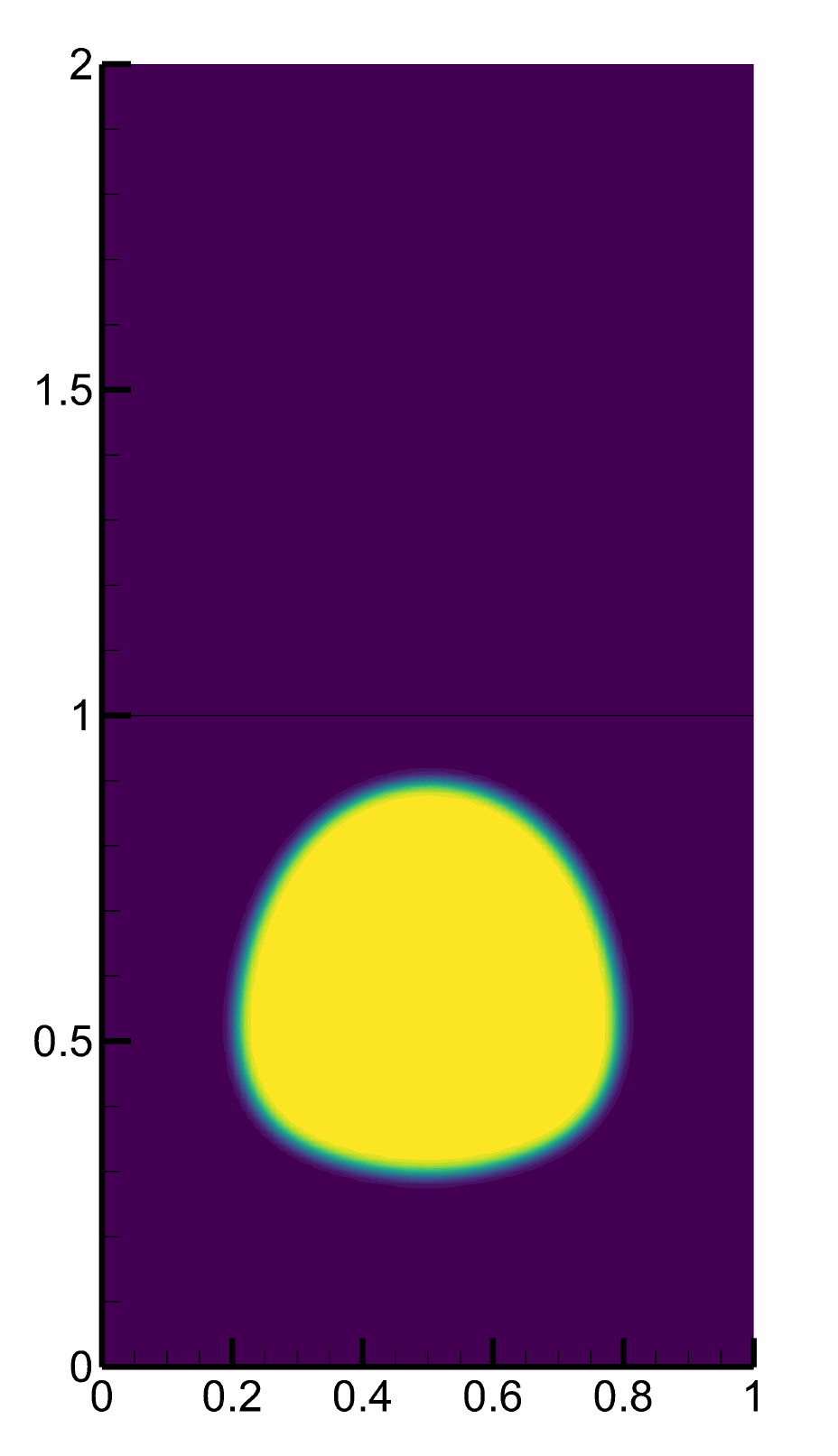}
 \includegraphics[scale=0.19]{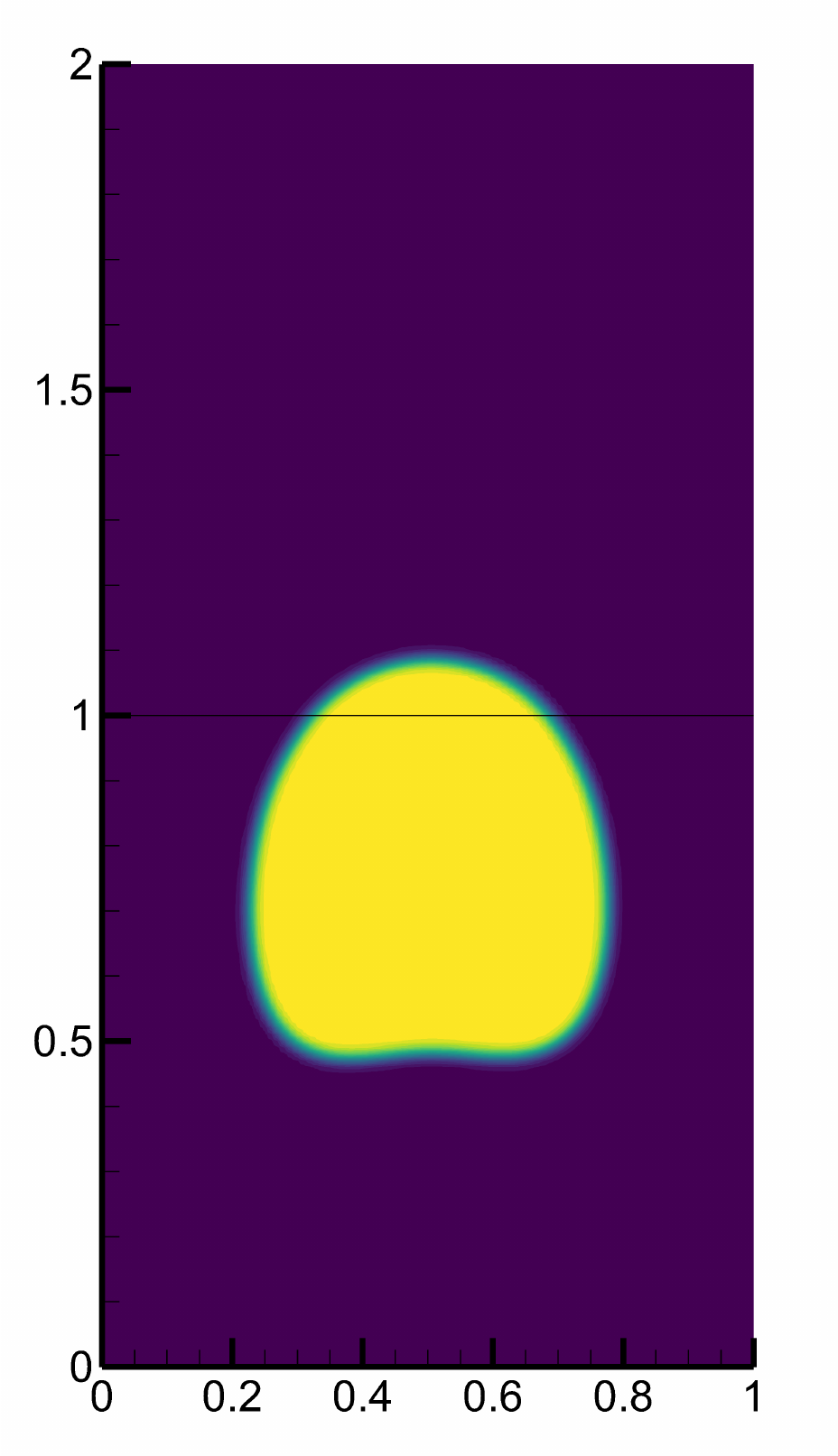}
 \includegraphics[scale=0.19]{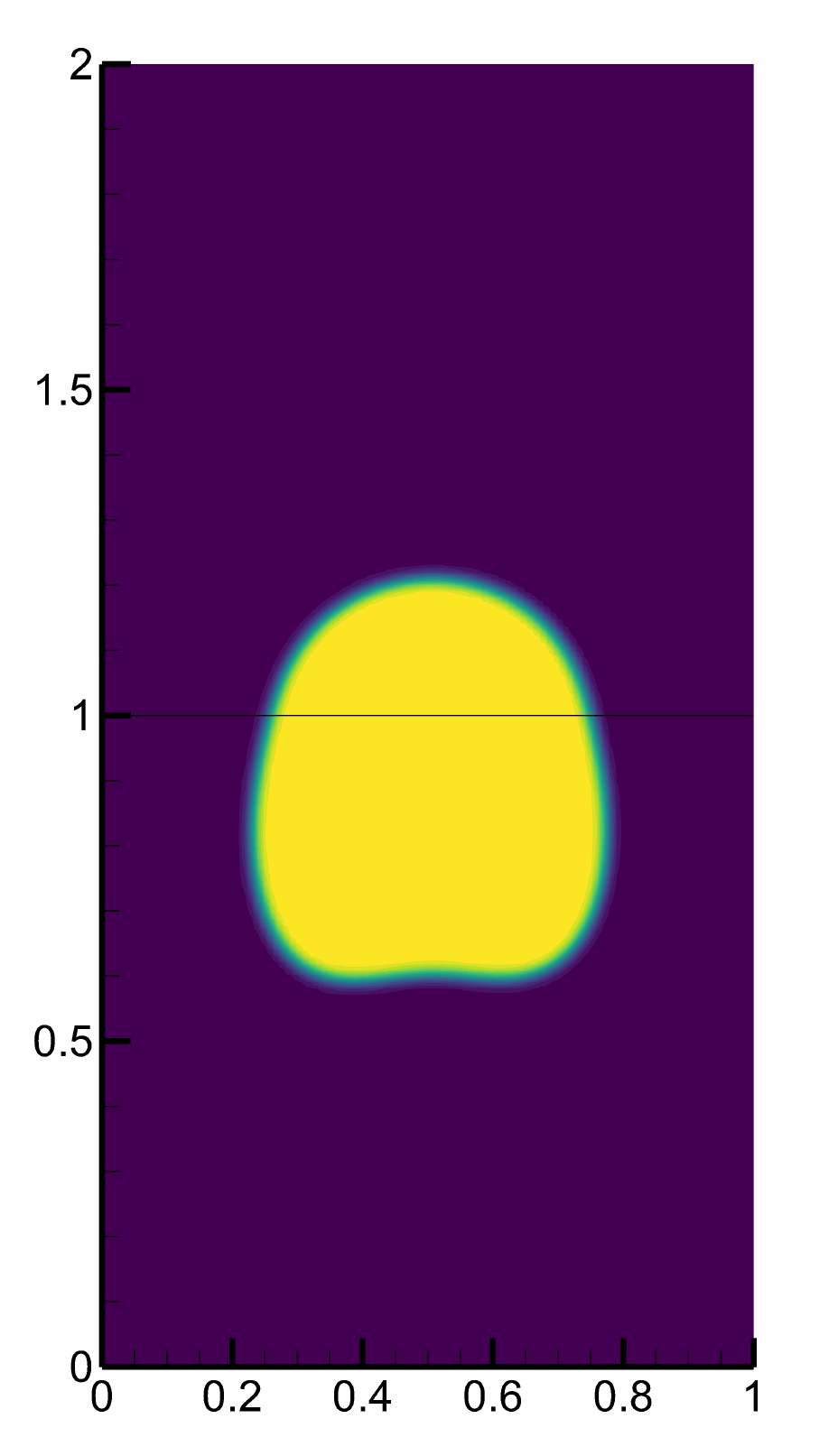}\\
 \includegraphics[scale=0.19]{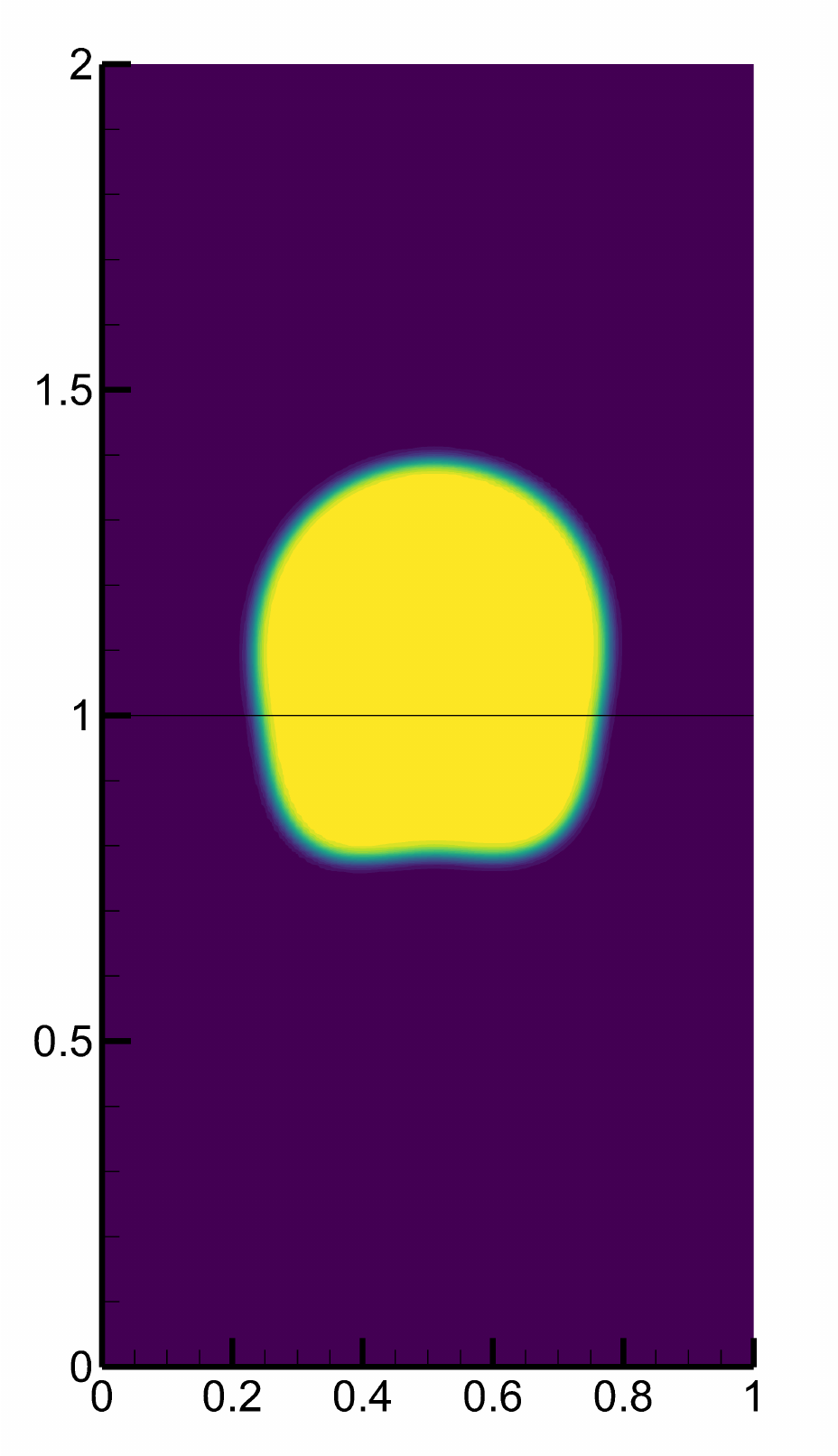}
 \includegraphics[scale=0.19]{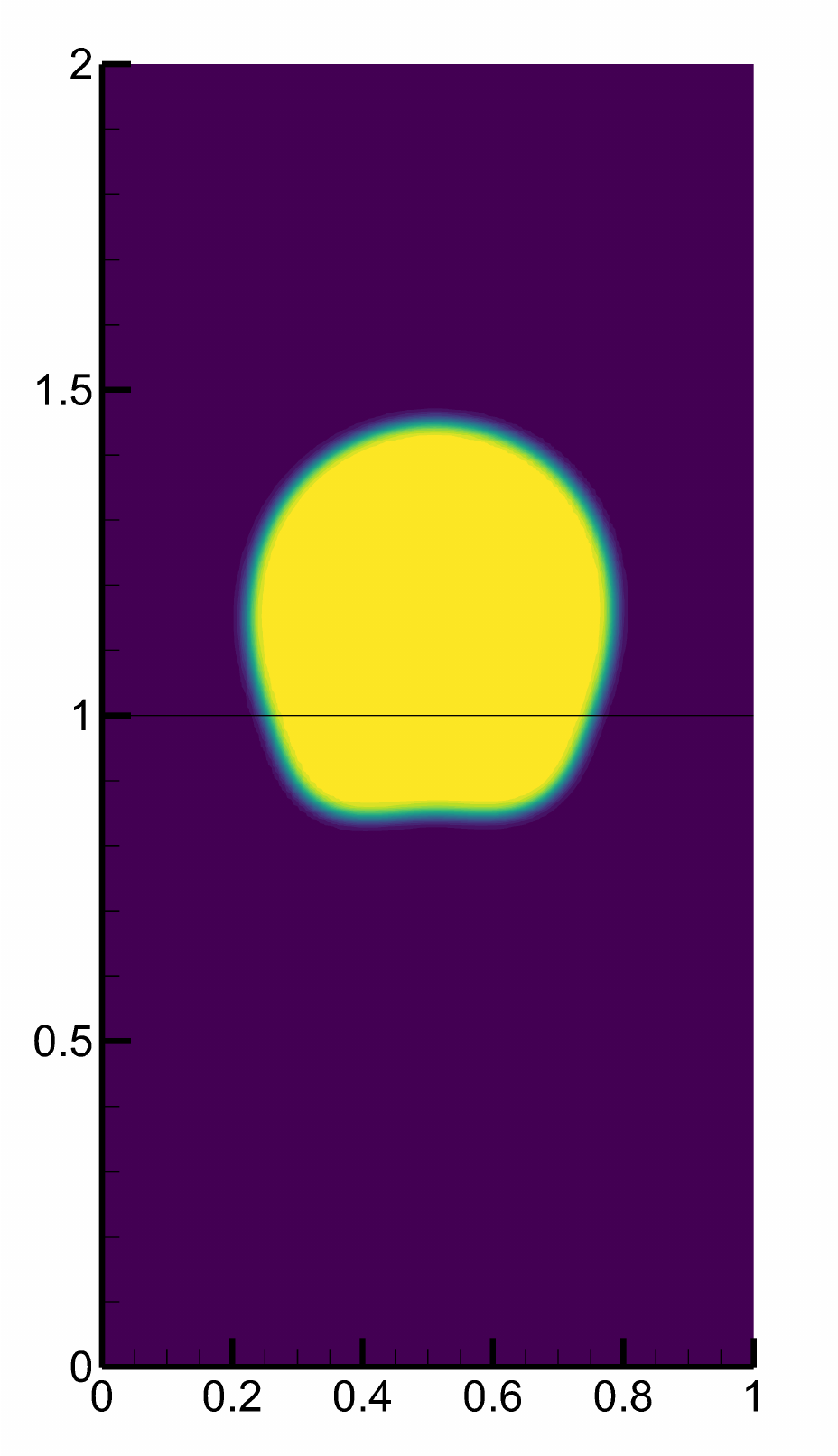}
 \includegraphics[scale=0.19]{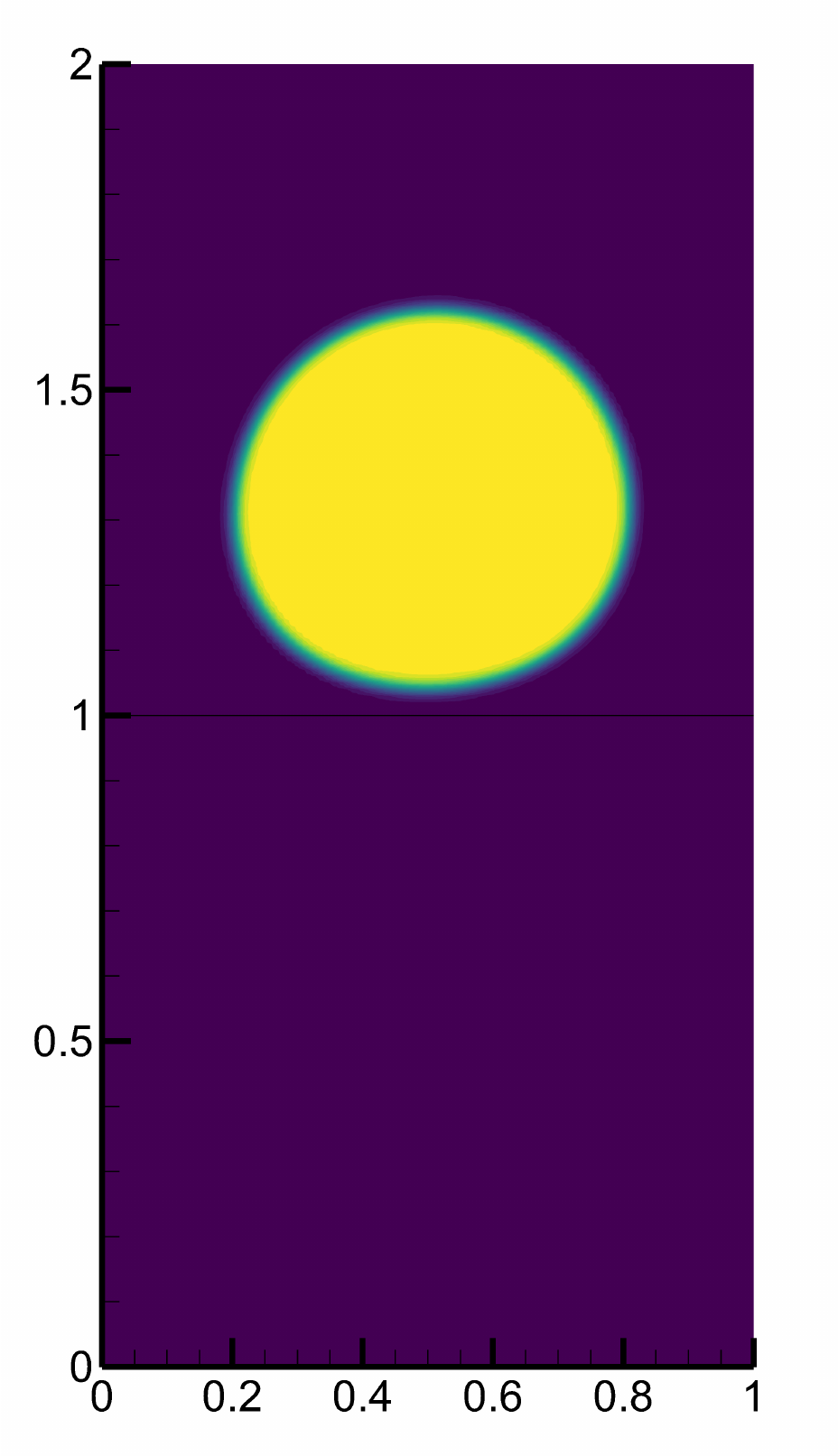}
 \includegraphics[scale=0.19]{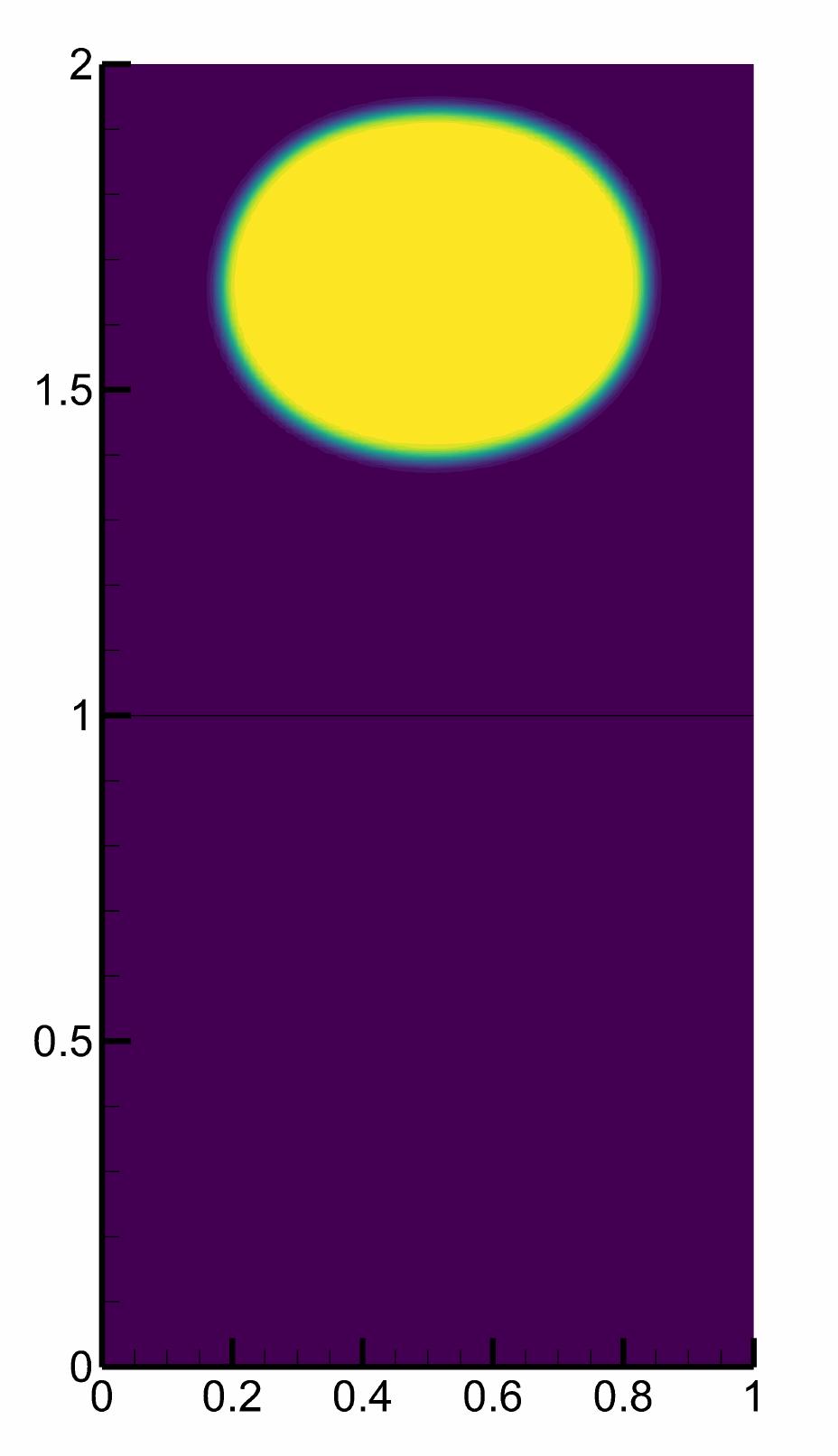}
  \caption{Buoyancy-driven bubble rising at T=0, 0.4, 1, 1.4, 2, 2.2, 2.8, 4.}
  \label{Buoyancy}
\end{figure}
Next, we set $\mathbf{B} =(0,8)^T$, and consider that there are two initial bubbles in the free-flow domain. As shown in Figure \ref{2Buoyancy}, the two bubbles rise by buoyancy and progressively merge while crossing the interface, which is consistent with \cite{chen2017uniquely,gao2023fully}.

\begin{figure}[h!]
 \centering
 \includegraphics[scale=0.19]{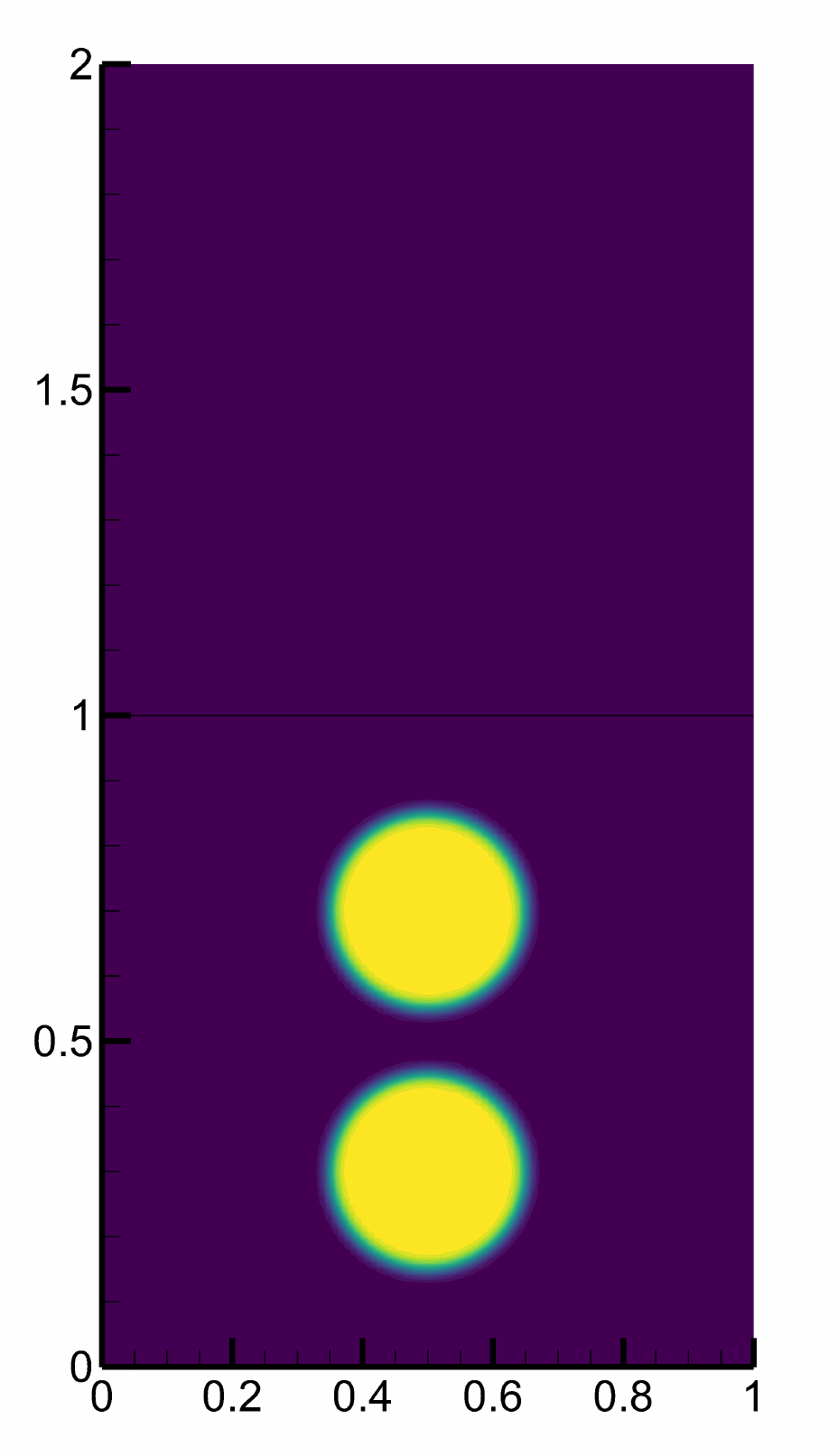}
 \includegraphics[scale=0.19]{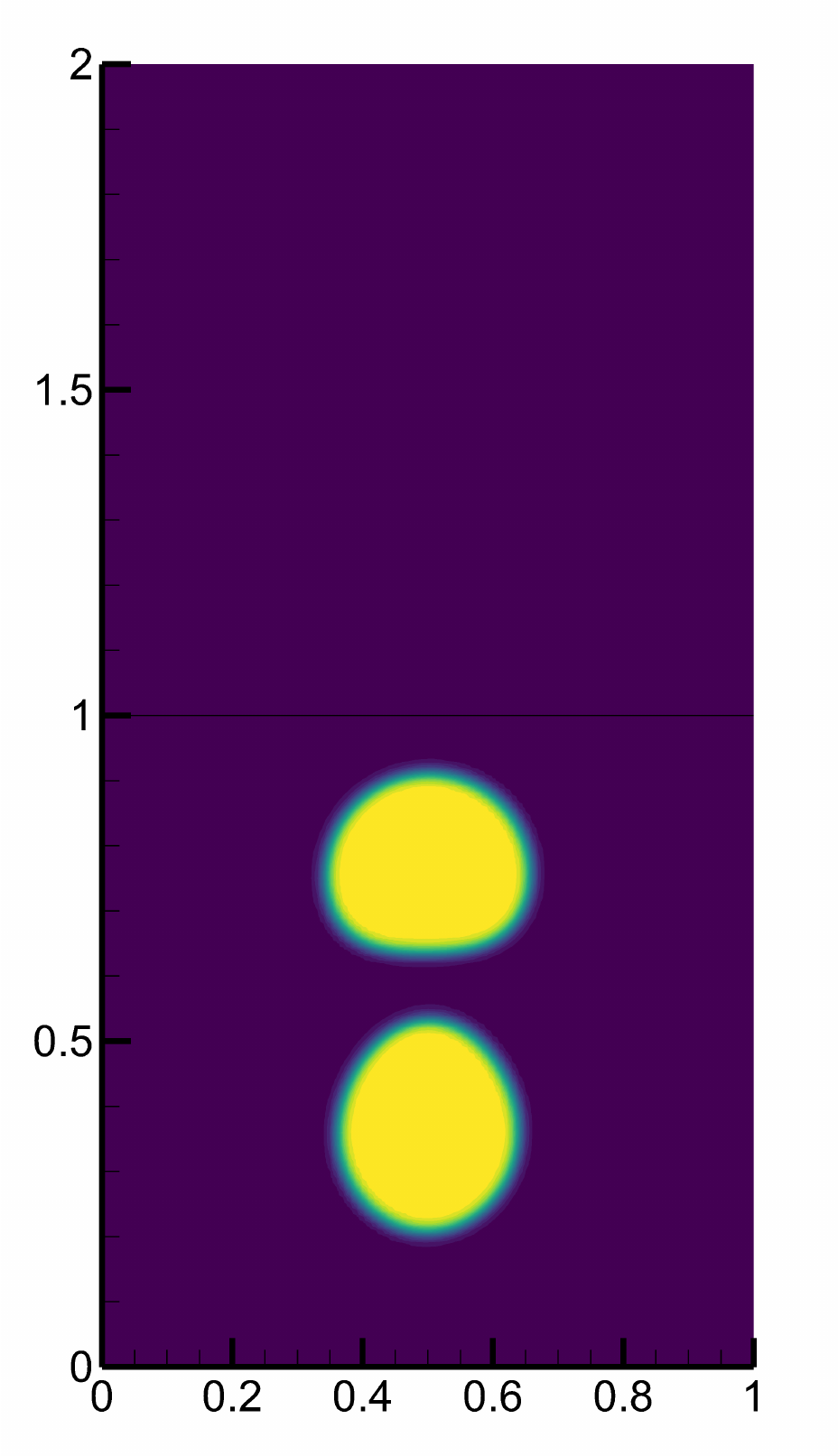}
 \includegraphics[scale=0.19]{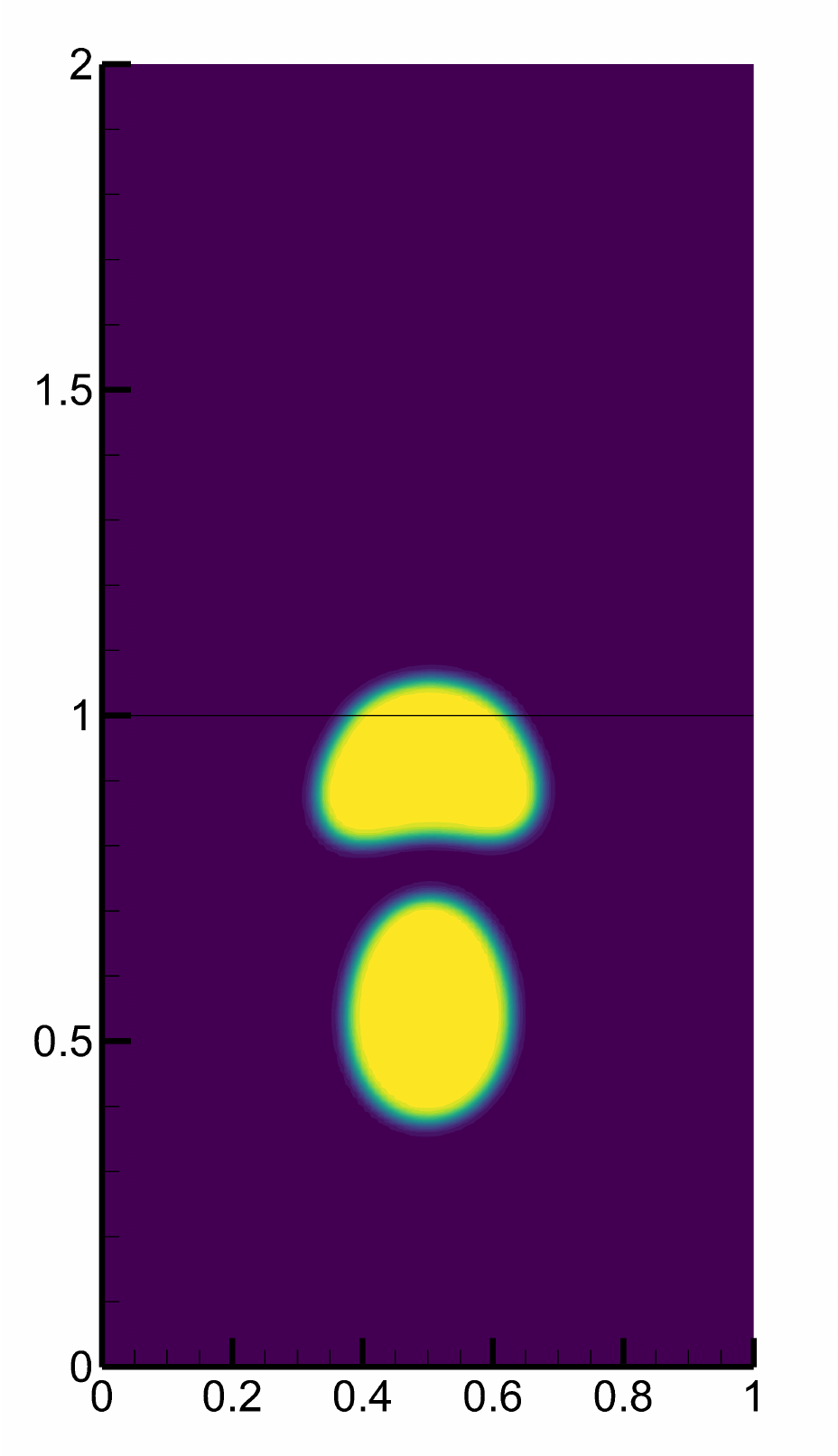}
 \includegraphics[scale=0.19]{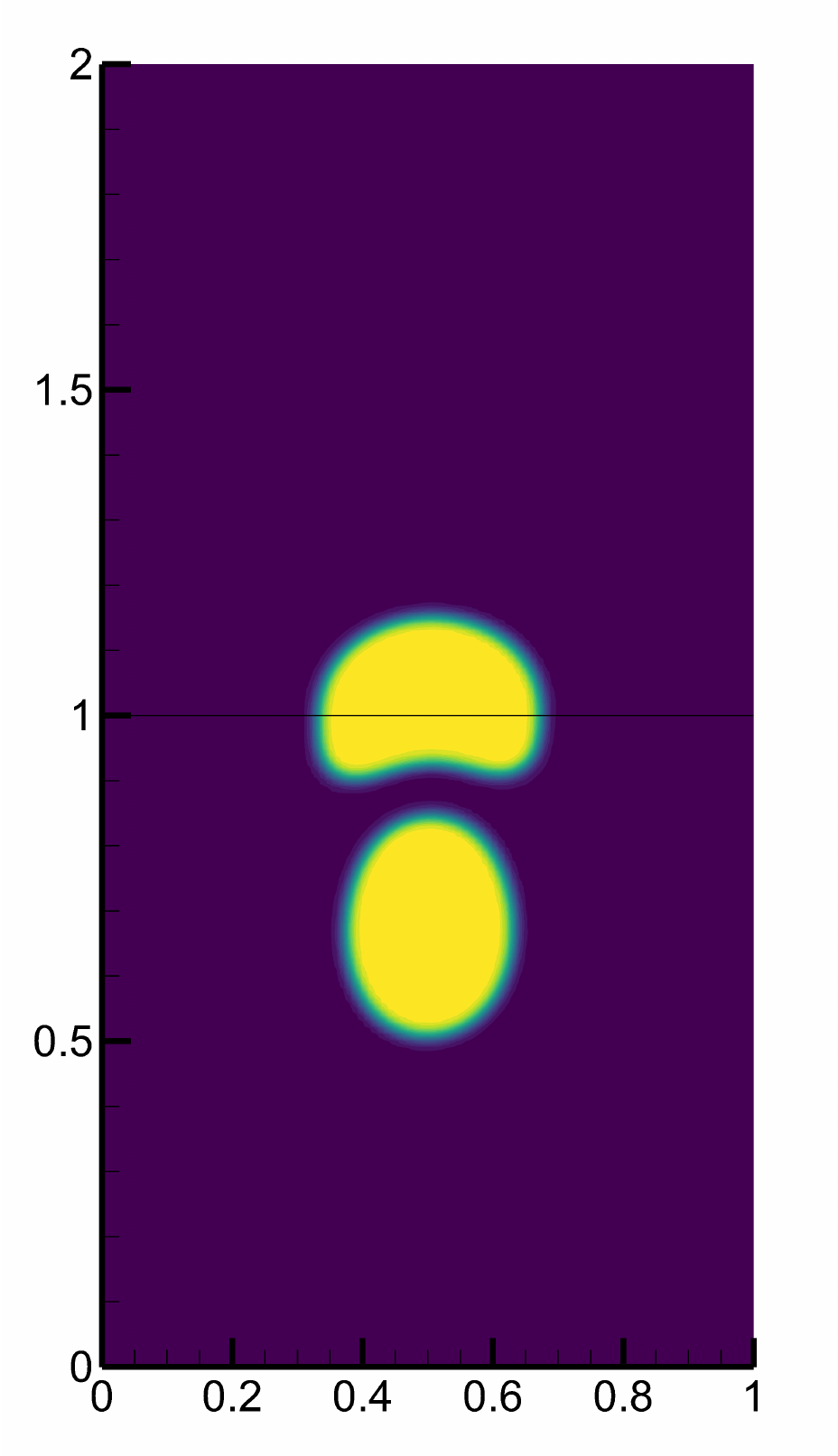}\\
 \includegraphics[scale=0.19]{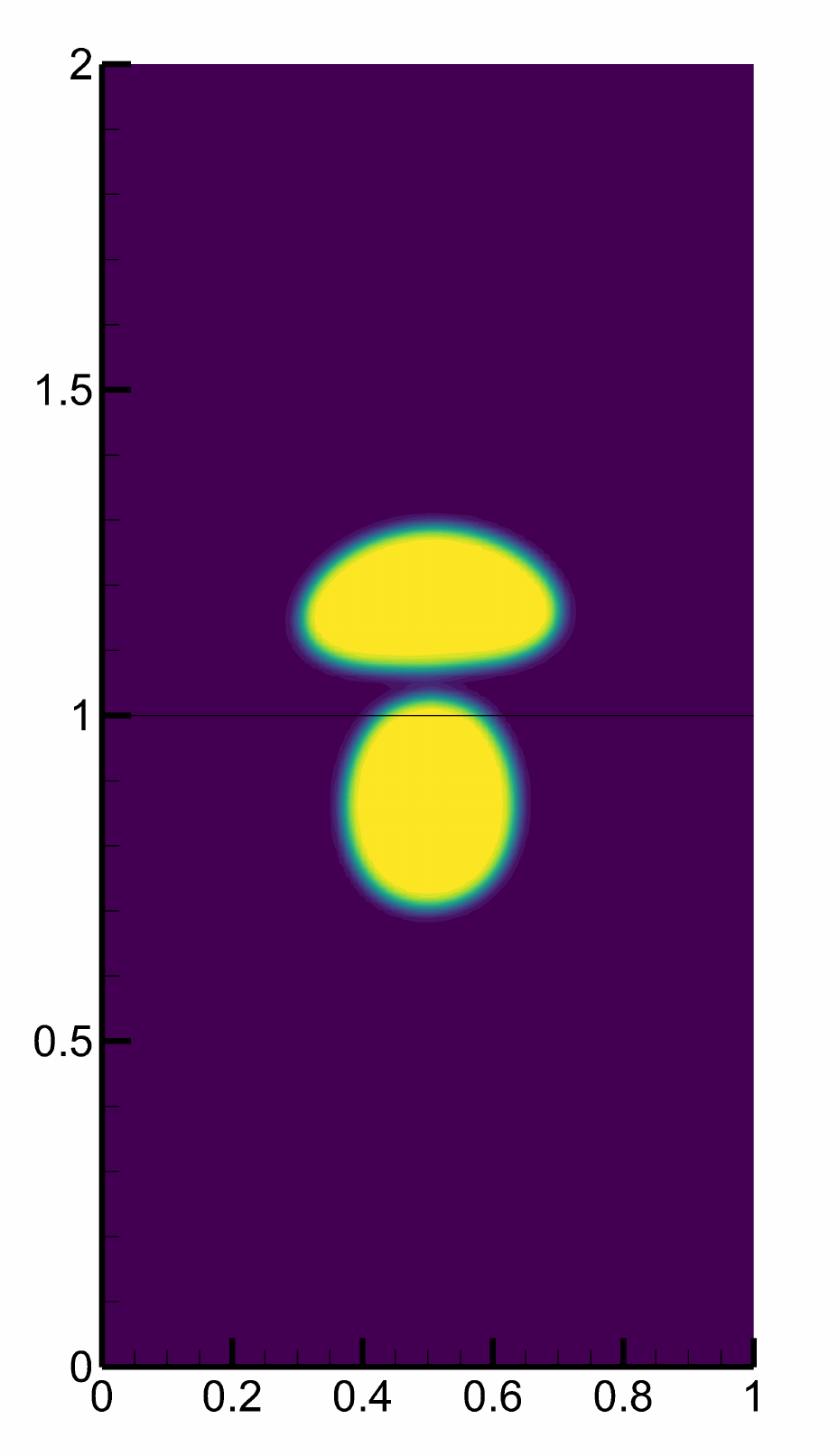}
 \includegraphics[scale=0.19]{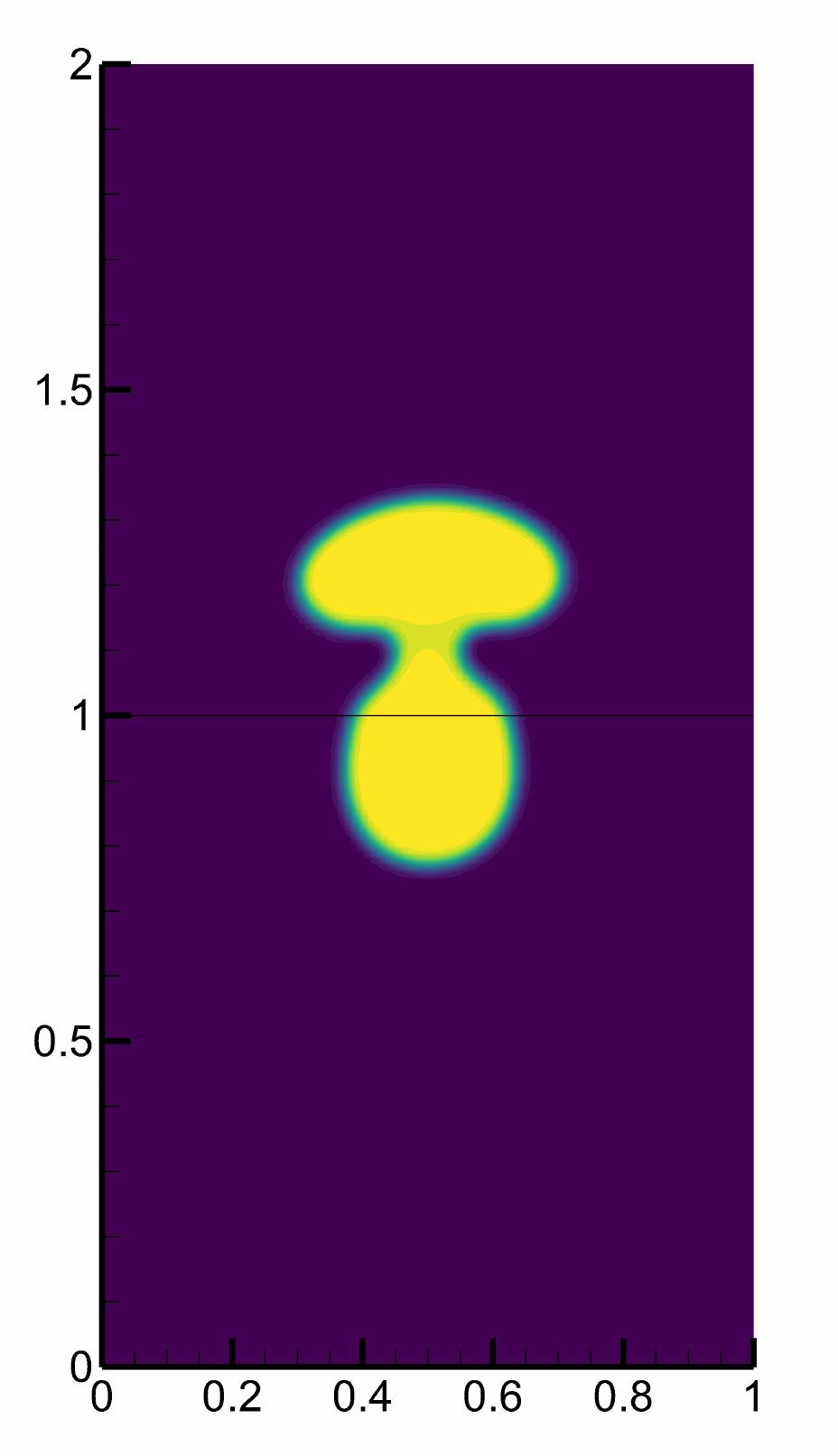}
 \includegraphics[scale=0.19]{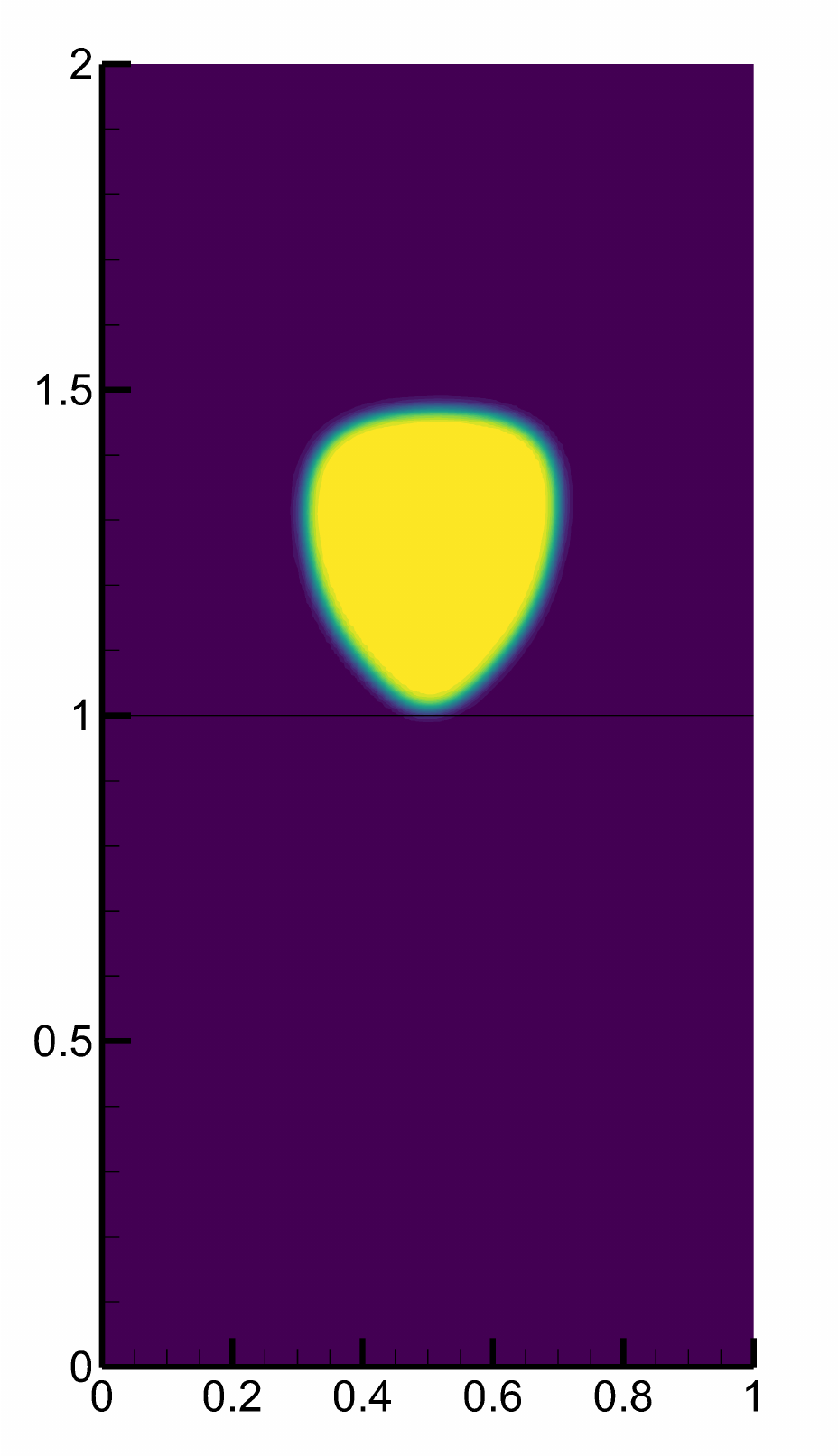}
 \includegraphics[scale=0.19]{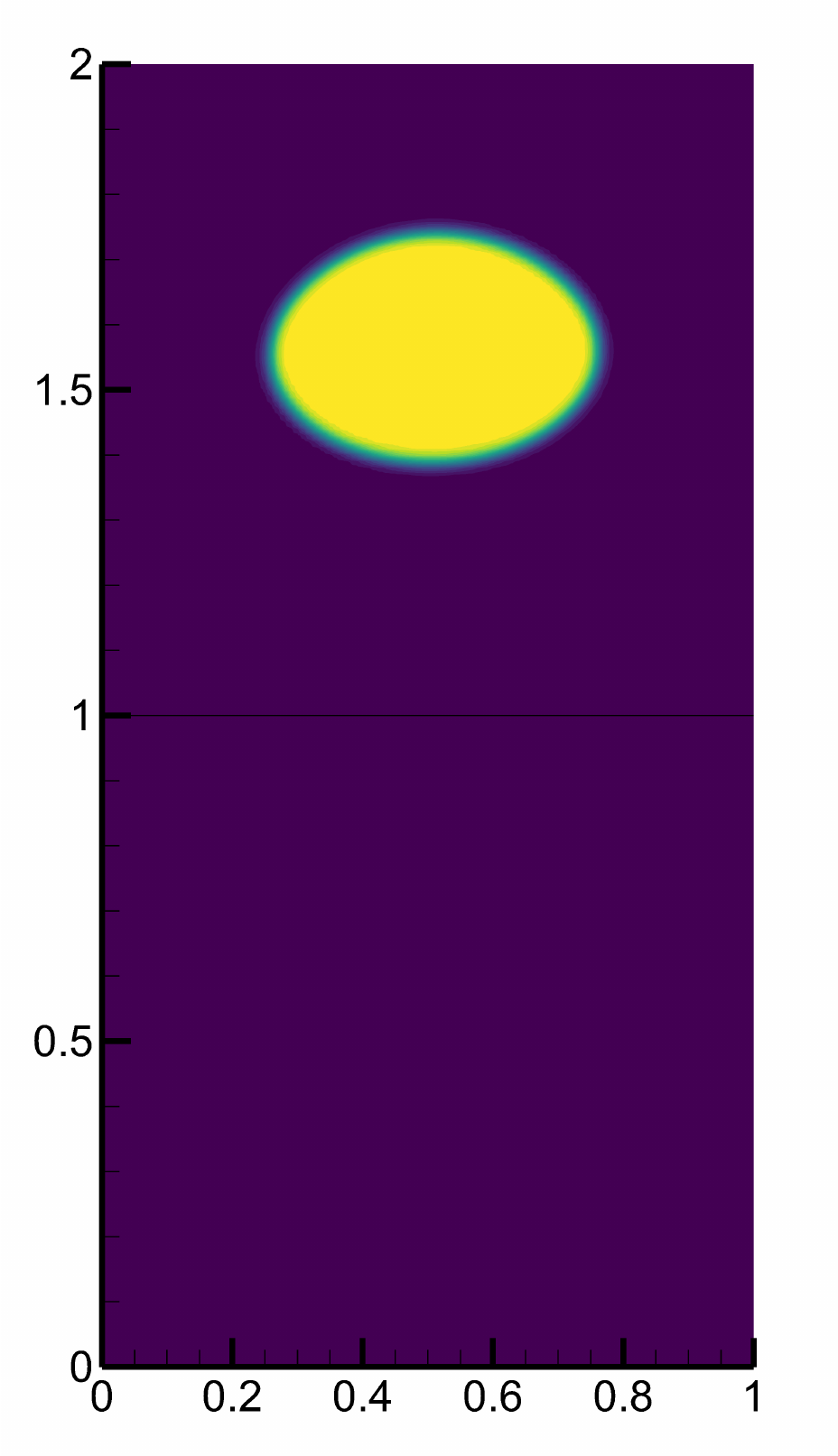}
  \caption{Buoyancy-driven bubble rising at T=0, 0.2, 0.5, 0.7, 1, 1.1, 1.4, 2.}
  \label{2Buoyancy}
\end{figure}

\section{Conclusion}
We developed in this paper novel  linear and decoupled numerical schemes  under a new relaxation mechanism  for the Navier--Stokes--Darcy model  and Cahn--Hilliard--Navier--Stokes--Darcy model. A key advantage of the proposed schemes is  their computational efficiency,  as they only   require solving  linear systems with constant coefficients at each time step.  An essential property of these schemes is that they unconditionally  dissipate the original energy. We rigorously established the unconditional stability of the proposed  schemes, and as an example, derived optimal error estimates of  the first-order scheme for the Navier--Stokes--Darcy model.

 In  numerical implementation, we adopted a finite element method with the EMAC formulation which can further preserve mass, momentum, and angular
 momentum at the discrete level. We presented several numerical experiments which  validated  the essential properties and robustness  of the proposed schemes.


\bibliographystyle{siamplain}
\bibliography{final}

\end{document}